\documentclass{article}

\usepackage{graphicx}
\usepackage{color}
\usepackage{amsfonts} 
\usepackage{amsmath}
\usepackage{amsthm}
\usepackage{mathtools}
\usepackage{float}
\usepackage{tikz}
\usetikzlibrary{arrows}
\usetikzlibrary{calc}
\usepackage{times}
\usepackage{latexsym}
\usepackage{pstricks}
\usepackage{pgfplots}
\usepackage{comment}
\usepackage{natbib}
\usepackage{appendix}

\usepackage[a4paper, total={6in, 9.5in}]{geometry}

\usepackage{subfig}
\usepackage{booktabs}
\usepackage{ulem}
\pgfplotsset{compat=1.17}

\usepackage{csquotes}

\newcommand{\zoneCunCdeux}{D_f(c_1, c_2)}
\newcommand{\zoneCdeuxCun}{D_f(c_2, c_1)}
\newcommand{\zoneCiCj}{D_f(c_i, c_j)}

\newcommand{\zoneVotant}{D_f(v) }
\newcommand{\hypSep}{boundary hypersurface}
\newcommand{\hyp}{hypersurface }
\newcommand{\hyps}{hypersurfaces }
\newcommand{\lprofile}{$\ell$-Euclidean }
\newcommand{\lunprofile}{$\ell_1$-Euclidean }
\newcommand{\linfprofile}{$\ell_\infty$-Euclidean }
\newcommand{\drectCiCj}{paral(c_i, c_j)}
\newcommand{\drect}[2]{paral(#1,#2)}
\newcommand{\os}[1]{#1}

\newtheorem{df}{Definition}
\newtheorem{prop}{Proposition}
\newtheorem{example}{Example}
\newtheorem{lem}{Lemma}
\newtheorem{thm}{Theorem}
\newtheorem{cor}{Corollary}

\newtheorem{obs}{Observation}

\title{Euclidean preferences in the plane under $\ell_1$, $\ell_2$ and $\ell_\infty$ norms
}

\author{Bruno Escoffier$^{1,2}$ \and Olivier Spanjaard$^1$ \and Magdal\'ena Tydrichov\'a$^1$}
\date{%
    $^1$Sorbonne Université, CNRS, LIP6, F-75005 Paris, France\\%
    $^2$Institut Universitaire de France\\[2ex]%
}

\begin{document}

\maketitle

\begin{abstract}
We present various results about Euclidean preferences in the plane under $\ell_1$, $\ell_2$ and $\ell_{\infty}$ norms. \os{
When there are four candidates, we show that the maximal size (in terms of the number of pairwise distinct preferences) of Euclidean preference profiles in $\mathbb{R}^2$ under norm $\ell_1$ or $\ell_{\infty}$ is 19.} Whatever the number of candidates, we prove that at most four distinct candidates can be ranked in last position of a two-dimensional Euclidean preference profile under norm $\ell_1$ or $\ell_\infty$, which generalizes the case of one-dimensional Euclidean preferences (for which it is well known that at most two candidates can be ranked last). \os{We generalize this result to $2^d$ (resp. $2d$) for $\ell_1$ (resp. $\ell_\infty$) for $d$-dimensional Euclidean preferences.} We also establish that the maximal size of a two-dimensional Euclidean preference profile on $m$ candidates under norm $\ell_1$ is in $\Theta(m^4)$, i.e., the same order of magnitude as under norm $\ell_2$. Finally, we provide a new proof that two-dimensional Euclidean preference  profiles under norm $\ell_2$ for four candidates can be characterized by three voter-maximal two-dimensional Euclidean profiles. This proof is a simpler alternative to that proposed by \cite{kamiya2011ranking}.
\end{abstract}

\section{Introduction}
\label{sec:intro}

The study of domain restrictions is a long standing research topic in modern social choice theory, dating back to the work of \citet{black1948rationale} on the single-peaked domain. As emphasized by \citet{barbera2020arrow}, Arrow already attached importance to studying the role of domain conditions
in determining the validity of his impossibility theorem, with two chapters of \emph{Social Choice and Individual Values} \citep{arrow1951social} devoted to this topic. For a survey about domain restrictions, the reader may refer to the works of Gaertner (\citeyear{gaertner2001domain,gaertner2002domain}) and \citet{barbera2013some}. For a computational perspective, one may refer to the survey recently conducted by \textcolor{black}{\cite{elkind_survey}}.

The spatial model of social choice is an important stream of research in this topic, pioneered by the works of \cite{hotelling1929stability} and \citet[chapter~8]{downs1957economic}. We focus here on \emph{Euclidean preferences}, where candidates and voters are viewed as points in $\mathbb{R}^d$, and the preferences of voters are decreasing with their Euclidean distance to the candidates. \os{Note that in this article, by abuse of language, we will use the expression \emph{Euclidean preferences in $\mathbb{R}^d$ under norm $\ell_k$} ($k\!\in\!\{1,2,\infty\}$) when measuring distances using norm $\ell_k$.} 

The most widely studied Euclidean preferences are those that are derived by measuring the distances with the $\ell_2$ norm \citep[see, e.g., the works of][]{Bennett1960,Euclidean_pref_bogomolnaia_laslier}, but the $\ell_1$ and $\ell_{\infty}$ norms have also been considered in the literature \citep[see, e.g., the work of][] {peters2017recognising}. \os{From a more operational point of view, spatial representations are used in particular in voting advice applications (e.g., Wahl-O-Mat in Germany, Smartvote in Switzerland, Vote Compass in the United States, and many others in multiple countries), i.e., online tools that helps the voter choose the candidate closest to her political stances, and actually often provides her a full ranking of candidates according to her answers to a survey on a range of policy statements. The answers are indeed converted into positions on different dimensions, each position reporting on the level of agreement on a particular policy statement. Norm $\ell_1$ is typically used when there are many dimensions, while norm $\ell_2$ is used when the number of dimensions is lower \citep{romero2022learning,isotalo2020}. For an overview of the topic of voting advice applications, the reader may refer to the survey by \cite{garzia2019voting}.}

In this work, we put a special emphasis on the case $d\!=\!2$ and we consider the norms $\ell_2$, $\ell_1$ and $\ell_{\infty}$. An original focus of our work is to try to identify the differences between the three norms. Hence, we are interested in the following questions:
\begin{itemize}
    \item Are there forbidden structures that make a profile not Euclidean, under some of the three norms?
    \item Given a set of $m$ candidates, what is the maximal size (in terms of the number of pairwise distinct preferences) of profiles that are Euclidean? 
    \item \textcolor{black}{Are there some differences or similarities between the norms in the expressivity of Euclidean preferences? Put another way, are there profiles that are Euclidean with one norm and not with another one?} 
\end{itemize}

We first show a \textcolor{black}{a structural result on \lunprofile and \linfprofile profiles in $\mathbb{R}^d$, namely that in an \linfprofile profile there are at most $2d$ candidates ranked last by at least one voter, while there are at most $2^d$ such candidates for an \lunprofile profile. While this result is not hard to prove, it is interesting in several aspects: first, it provides a strong difference with $\ell_2$, as we can easily build $\ell_2$-Euclidean profiles in the plane where each candidate is ranked last at least once\footnote{\textcolor{black}{Place for instance the candidates on a circle, and for each candidate $c$ place a voter $v$ which is the antipode of $c$ with respect to the center of the circle.}}. Second, it is an interesting generalization of the case of 1-dimensional Euclidean preferences, where it is well known that at most 2 candidates can be ranked last. Finally, while it is known that $\ell_1$ and $\ell_\infty$ are equivalent when $d\leq 2$ (meaning that, for $d\!\in\!\{1,2\}$, a profile is $\ell_1$-Euclidean if and only if it is $\ell_\infty$-Euclidean), an immediate  corollary of this structural result is that this equivalence does {\it not} hold for $d\geq 3$.}

\textcolor{black}{We then focus on the case $d=2$.} As it can easily be seen that every profile with 2 or 3 candidates is $\ell_1$-Euclidean (thus $\ell_\infty$-Euclidean) and $\ell_2$-Euclidean \os{\citep{Euclidean_pref_bogomolnaia_laslier}}, we focus in Section~\ref{sec:n=4} on the case of $m\!=\!4$ candidates. We first give an explicit example of a profile which is $\ell_1$-Euclidean but not $\ell_2$-Euclidean. We then focus on the maximal size \textcolor{black}{(in terms of the number of pairwise distinct preferences)} of profiles on 4 candidates that are Euclidean. It is known since the work of \citet{Bennett1960} that the maximal size is 18 for $\ell_2$. We show that this maximal size is exactly 19 for $\ell_1$. Then, we give a new proof that
a profile on 4 candidates is $\ell_2$-Euclidean if and only if it is a subprofile of one of three voter-maximal two-dimensional Euclidean profiles (involving 18 voters). \cite{kamiya2011ranking} proved the same result, but they rely on a link they establish with the problem of enumerating chambers of hyperplane arrangements \citep[for an introduction to the topic, see, e.g., the chapter of][]{stanley2004introduction}, while we use simpler and purely geometrical arguments. 

We then focus on the case $m\geq 5$. 
We focus on the the maximal size of profiles which are Euclidean. We show that, despite the strong restriction on the number of candidates ranked last by some voter \textcolor{black}{(at most $2d=4$)}, the maximal size of an $\ell_1$-Euclidean profile is $\Theta(m^4)$, i.e., of the same order of magnitude as for $\ell_2$ \textcolor{black}{\citep[as shown by][]{Bennett1960}}.

\paragraph{Organization of the article.}

We provide a brief overview of the related work in Section~\ref{sec:related}. We then give in Section~\ref{sec:def} some formal definitions, examples, and \textcolor{black}{focus on  the relation between} $\ell_1$ and $\ell_\infty$ norms for Euclidean profiles.
Some geometric properties of representations of $\ell_1$-Euclidean profiles are given in Section~\ref{sec:prop}. We highlight some differences with the $\ell_2$ norm, and derive some properties that will be useful for the results of subsequent sections.
Then we present in Section~\ref{sec:n=4} our results with $m=4$ candidates, while Section~\ref{sec:n>=5} deals with the general case, i.e., for an arbitrary number of candidates.
We conclude the article in Section~\ref{sec:conclu} by providing some research directions.

\section{Related work} 
\label{sec:related}

\paragraph{The one-dimensional case.} \cite{chen2015onedimensional} proved that one-dimensional Euclidean preference profiles cannot be characterized in terms of finitely many forbidden substructures, i.e., one cannot enumerate a finite set of substructures (also called \emph{obstructions}) such that a profile is one-dimensional Euclidean if and only if it contains none of the substructures in the list. It is nevertheless known that one-dimensional Euclidean preference profiles can be recognised in polynomial time in the number of voters and candidates, \textcolor{black}{as first shown by \citet{doignon1994polynomial}}, and then by \citet{knoblauch_2010} and \citet{DBLP:conf/sagt/ElkindF14}. Very recently, \cite{chen2021small} characterized one-dimensional Euclidean preference profiles with a small number of candidates and voters. In particular, they showed that any profile with at most 5 candidates is Euclidean if and only if it is single-peaked and single-crossing \textcolor{black}{(where both single-peaked and single-crossing profiles can be characterized via finitely many finite obstructions)}. They finally identified the smallest single-peaked and single-crossing profile which is not Euclidean. 

\paragraph{The multidimensional case.} \cite{Bennett1960} as well as \cite{Hays1961} proposed several methods to estimate the minimum value of $d$ to be able to embed a preference profile in a $d$-dimensional space, i.e., to associate a point in $\mathbb{R}^d$ to each voter and each candidate so that the voters' preferences are decreasing with the distance to the candidates. In particular, they established that the maximum cardinality of a $\ell_2$-Euclidean profile on $m$ candidates in dimension $d$ is equal to $\sum_{k=m-d}^m |s(m,k)|$, where $s(m,k)$ are the (unsigned) Stirling numbers of the first kind. The same result has been found by \cite{good1977stirling}. Later on,  \cite{Euclidean_pref_bogomolnaia_laslier} showed that to guarantee any profile of $n$ preferences on $m$ candidates to be $d$-Euclidean, it is necessary and sufficient to have $d$ between $\min \{ n-1, m-1 \}$ and $\min \{ n, m-1\}$. \textcolor{black}{Recently, an analogous result was shown by \cite{chen_manhattan_multiD} for preference profiles using an $\ell_1$ metric. More precisely, \cite{chen_manhattan_multiD} showed that each preference profile with $m$ alternatives and $n$ voters is $d$-Euclidean with respect to the norm $\ell_1$ whenever $d\!\geq\! \min\{n, m-1\}$. Also, they studied the smallest non-Euclidean profiles in case of $d=2$.} As mentioned earlier, \cite{kamiya2011ranking} studied the question of counting and enumerating voter-maximal $\ell_2$-Euclidean profiles in $\mathbb{R}^d$, according to the number $m$ of candidates. They provide a formula for the number of voter-maximal profiles if $d\!=\!m-2$, and they were able to enumerate them for $m\!=\!4$. Regarding the computational aspects, \cite{peters2017recognising} proved that the recognition problem (i.e., deciding whether or not a preference profile is $\ell_2$-Euclidean in dimension $d$) is NP-hard for $d>1$, and that some  Euclidean  preference  profiles  require exponentially  many  bits  in  order  to  specify  any  Euclidean embedding.

\section{Preliminaries}\label{sec:def}

\subsection{Euclidean preference profile}

We consider a (finite) set $V$ of $n$ voters, and a (finite) set $C$ of $m$ candidates. Each voter gives her preference \textcolor{black}{$>_v$} over the set of candidates as a ranking (total ranking, without tie). The set \textcolor{black}{$\{ >_{v_1}, >_{v_2}, \hdots, >_{v_n}\}$} of preferences of voters in $V$ on candidates in $C$ is \textcolor{black}{denoted by R, and the couple $\mathcal{P} = (C,R)$ is} called a preference profile. We write $c_1\!>_v\!c_2$ if voter $v$ prefers $c_1$ to $c_2$. For conciseness, we will often write the preference $c_i\!>_v\!c_j\!>_v\! \dots\!>_v\!c_k$ for a voter $v$ as $(c_i,c_j,\dots,c_k)$.

\begin{df}\label{def:euclid}
Let $d\!\geq\!1$ be an integer, and $\| \cdot \|_\ell$ a norm on $\mathbb{R}^d$. The profile $\mathcal{P}\!=\! (C,R)$ of \textcolor{black}{preferences of} $n$ voters over $m$ candidates is \textit{\lprofile}in $\mathbb{R}^d$ if there exists a mapping $f\!:\!V\!\cup C\! \rightarrow\!\mathbb{R}^d$ such that for each $v \in V$ and each $\{c_1, c_2\} \subseteq C$: 
$$ c_1 >_v c_2 \Rightarrow \| f(v) - f(c_1) \|_\ell < \|  f(v) - f(c_2) \|_\ell $$
\end{df}

The mapping is called an $\ell$-Euclidean representation of the profile in $\mathbb{R}^d$. Obviously, such a representation is not necessarily unique. A profile for which there exists a $\ell$-Euclidean representation in $\mathbb{R}^d$ is called $\ell$-Euclidean in $\mathbb{R}^d$.

We note that if two voters $v$ and $v'$ have the same preference, then $\mathcal{P} = (C,R)$ is $\ell$-Euclidean in $\mathbb{R}^d$ if and only if $(C, R \setminus \{>_{v'}\})$ is $\ell$-Euclidean in $\mathbb{R}^d$. So, without loss of generality, throughout the article we consider preference profiles where any pair of voters have different preferences. \os{We define the {\it size} (or cardinality) of a profile as the number of votes (or, equivalently, voters).}

We note also that, as preferences in $\mathcal{P}$ are strict orders, we could replace \textcolor{black}{the implication} in Definition~\ref{def:euclid} by  \textcolor{black}{an equivalence - hence, the last line of the definition becomes: } $$ c_1 >_v c_2 \Leftrightarrow \| f(v) - f(c_1) \|_\ell < \| f(v) - f(c_2) \|_\ell .$$

\subsection{Boundary hypersurfaces}

Consider a profile $\mathcal{P}\!=\!(C,R)$, an integer $d$ and a norm $\| \cdot \|_\ell$ on $\mathbb{R}^d$. Given a set of points $p_1, \hdots , p_m \in \mathbb{R}^d$, we now study the question of determining whether there exists a mapping $f \!:\!V\!\cup C\!\rightarrow\!\mathbb{R}^d$ such that: 
\begin{enumerate}
    \item for each $i \in \lbrace 1, \hdots , m \rbrace$, $f(c_i) = p_i$;
    \item $f$ is a $\ell$-Euclidean representation of $\mathcal{P}$ in $\mathbb{R}^d$.
\end{enumerate}
To build an \lprofile representation in $\mathbb{R}^d$, it is sufficient (and necessary) to find for each $v\!\in\!V$ a value $f(v)$ such that $f$ fulfills the condition in  Definition~\ref{def:euclid}. Let us define, for each $v$, the set $\zoneVotant$  of such possible values:
$$ \zoneVotant = \{f(v)\!\in\!\mathbb{R}^d: \forall \{c_1, c_2\}\!\subseteq\!C, c_1\!>_v\!c_2 \Rightarrow \| f(v) - f(c_1) \|_\ell < \| f(v) - f(c_2) \|_\ell \}$$
With this notation, the profile is \lprofile in $\mathbb{R}^d$ if and only if \textcolor{black}{there exists a mapping $f$ such that} for each $v$, $\zoneVotant$ is a non-empty set. The natural question is to characterise $\zoneVotant$ for each voter $v$. To this end, we introduce the following notion: 
\begin{df}
\label{def:euclidian_zones}
For a pair $\{c_1, c_2\}\!\subseteq\!C$ of candidates mapped in positions $f(c_1)$ and $f(c_2)$, the set of points $p\!\in\!\mathbb{R}^d$ such that $\| f(c_1) - p \|_\ell = \| f(c_2) - p \|_\ell $ is called \textit{the \hypSep~of $c_1$ and $c_2$} (or just \hyp in what follows), and is denoted by \textcolor{black}{$H_f(c_1, c_2)$}. We denote then by \textcolor{black}{$\zoneCunCdeux$} the set of points $p\!\in\!\mathbb{R}^d$ such that $\| f(c_1) - p\|_\ell < \| f(c_2) - p \|_\ell $, and by \textcolor{black}{$\zoneCdeuxCun$} the set of points $p\!\in\!\mathbb{R}^d$ such that $\| f(c_1) - p \|_\ell > \| f(c_2) - p \|_\ell $. 
\end{df}
It is easy to convince oneself that:
\begin{center}
$\zoneVotant = \displaystyle\bigcap_{c_i >_v c_j}\zoneCiCj$
\end{center}
Note that if $\ell\!=\!\ell_2$,   $\zoneVotant$ is convex \os{(as an intersection of half spaces bounded by a hyperplane)} for each $v\!\in\!V$. However, we will see later that $\zoneVotant$ is not necessarily convex if $\ell \!=\!\ell_1$ or $\ell\!=\!\ell_\infty$.  
\textcolor{black}{For conciseness, and only if no confusion is possible, we will omit the representation function $f$ in the notions introduced in Definition~\ref{def:euclidian_zones}. Thus, we will write $H(c_i, c_j)$, resp. $D(c_i, c_j)$ and $D(v)$, instead of $H_f(c_i, c_j)$, resp. $D_f(c_i, c_j)$ and $D_f(v)$.}
\\ \\
\noindent
As $D_f(c_1,c_2)$ depends only on the positions of $c_1$ and $c_2$ in $\mathbb{R}^d$, and hence $D_f(v)$ on the positions of $c_1,\ldots,c_m$, the definition of \lprofile profiles in $\mathbb{R}^d$ can be reformulated as follows:
\begin{prop}
Let $d\!\geq\!1$ be an integer, and $\| \cdot \|_\ell$ be a norm on $\mathbb{R}^d$. The profile $\mathcal{P} = (C,R)$ of \textcolor{black}{preferences of} $n$ voters over $m$ candidates is \textit{\lprofile}in $\mathbb{R}^d$ if and only if there exists a mapping $f: C \rightarrow \mathbb{R}^d$ such that $\zoneVotant$ is non-empty for each $v\!\in\!V$.
\end{prop}
\noindent 
\textcolor{black}{Given a representation function $f$ and a voter $v$, we will call the set $D_f(v)$ \textit{an area}, as geometrically, it corresponds to an area of the plane. Thus, $D_f(v)$ is an area of preference ranking $>_f$ with respect to the representation $f$. By abuse of notation, the terms of area and (its corresponding) preference ranking will be used interchangebly.}
\begin{example}\label{ex:n=3ell2}
Consider a preference profile $\mathcal{P}$ with 3 candidates $\{c_1,c_2,c_3\}$ and the 6 possible preferences on these 3 candidates. Figure~\ref{fig:l2_c3} shows a $\ell_2$-Euclidean representation of $\mathcal{P}$ in $\mathbb{R}^2$: the 3 hypersurfaces $H(c_1,c_2)$, $H(c_1,c_3)$ and $H(c_2,c_3)$, as well as the 6 non-empty areas $D(v)$ (with the corresponding preference written in the area).

\begin{figure}[tb]
    \centering
    \begin{tikzpicture}
   \begin{axis}
   [axis x line=bottom,axis y line = left, 
   grid = major,
   axis equal image,
   ytick = {1,2,3,4,5,6,7,8,9,10,11},
   xtick = {1,2,3,4,5,6,7,8,9,10,11},
   xmin=0,
   xmax=11,
   ymin=0,
   ymax=11,
   nodes near coords,
   point meta=explicit symbolic]
   \addplot+[only marks] coordinates{(3,3)[$c_1$] (8,6)[$c_2$] (6,2)[$c_3$]};
    \addplot+[mark = none, blue, thick] coordinates{(1,12) (5.5,4.5) (10,-3)} node[xshift = -0.5cm, yshift = 2cm] {$H(c_1, c_2)$};
    
     \addplot+[mark = none, blue, thick] coordinates{(7.5,11.5) (4.5,2.5) (3.5,-0.5)} node[xshift = -0.5cm, yshift = 1cm] {$H(c_1, c_3)$};
     
     \addplot+[mark = none, blue, thick] coordinates{(-1,8) (7,4) (11,2)} node[xshift = -0.75cm, yshift = 1cm] {$H(c_2, c_3)$};
     
     \node[draw] at (axis cs:2,4.5) {$c_1 > c_3 > c_2 $};
     \node[draw] at (axis cs:2,8) {$c_1 > c_2 > c_3 $};
     \node[draw] at (axis cs:5,10) {$c_2 > c_1 > c_3 $};
     \node[draw] at (axis cs:9,8) {$c_2 > c_3 > c_1 $};
     \node[draw] at (axis cs:9,2) {$c_3 > c_2 > c_1 $};
      \node[draw] at (axis cs:6,1) {$c_3 > c_1 > c_2 $};
    
    \end{axis}
    \end{tikzpicture}
    \caption{An $\ell_2$-representation of the complete profile on 3 candidates.}
    \label{fig:l2_c3}
\end{figure}
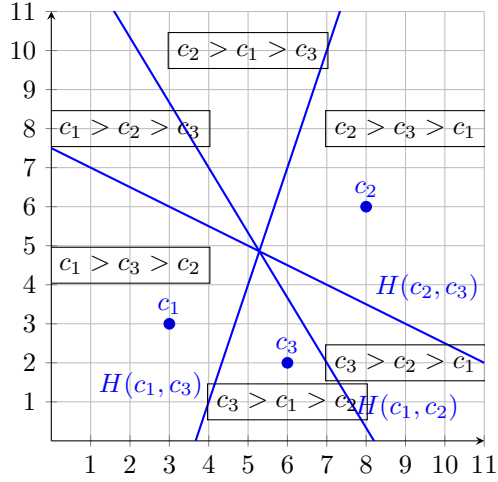

\end{example}

We now give a profile on 4 candidates which is {\it not} $\ell_2$-Euclidean in $\mathbb{R}^2$. As we will see, the fact that it is not Euclidean is proved by geometric arguments. 

\begin{example}\label{ex:n=4ell2}
Let us consider the following profile $\mathcal{P}\!=\!(C,R)$ with 9 voters and 4 candidates, where $R_v$ is the preference of voter $v$: 
\begin{align*}
    R_1: (c_4,c_3,c_1,c_2) & \hspace{0.5cm} R_2: (c_3,c_4,c_1,c_2) \\
    R_3: (c_4,c_3,c_2,c_1) & \hspace{0.5cm} R_4: (c_3,c_4,c_2,c_1) \\
    R_5: (c_2,c_1,c_4,c_3) & \hspace{0.5cm} R_6: (c_2,c_1,c_3,c_4) \\
    R_7: (c_1,c_2,c_4,c_3) & \hspace{0.5cm} R_8: (c_1,c_2,c_3,c_4) \\
    R_9: (c_2,c_3,c_1,c_4) \\
\end{align*}

Let us show that this profile is not $\ell_2$-Euclidean in $\mathbb{R}^2$ (while we will see later that it is $\ell_1$-Euclidean in $\mathbb{R}^2$). By contradiction, assume that a $\ell_2$-Euclidean representation in $\mathbb{R}^2$ exists. The points $c_1, c_2, c_3$ form necessarily a (non-degenerate) triangle, as 5 different rankings over $\{ c_1, c_2, c_3\}$ are present in the profile, and at most 4 can be represented if $c_1$, $c_2$, $c_3$ are aligned in $\mathbb{R}^2$. Figure~\ref{fig:ex_intro_l2} illustrates the different preference areas obtained from candidates $c_1$, $c_2$, $c_3$ forming a triangle.

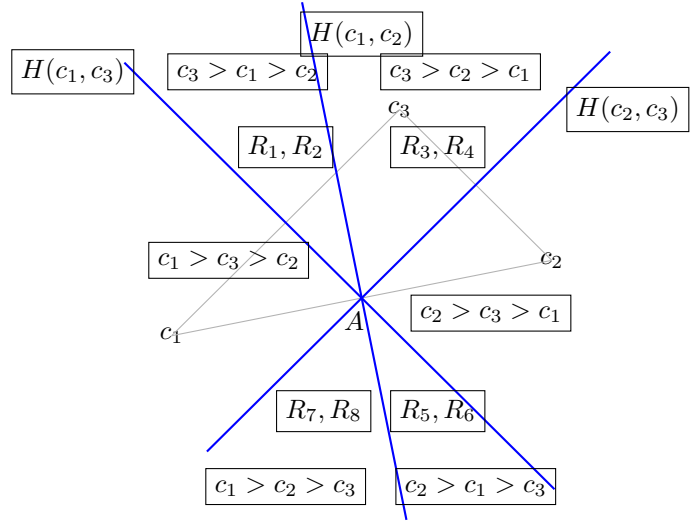
\begin{figure}[tb]
    \centering
    \begin{tikzpicture}
        \node (c1) at (0,0) {$c_1$};
        \node (c2) at (5,1) {$c_2$};
        \node (c3) at (3,3) {$c_3$}; 
        
        \draw[black!30] (0,0) -- (5,1) coordinate[midway] (M1);
        \draw[black!30] (0,0) -- (3,3) coordinate[midway] (M3);
        \draw[black!30] (3,3) -- (5,1) coordinate[midway] (M2);
        
        \draw [blue,thick] ($(M1)!4cm!270:(c1)$) -- ($(M1)!3cm!90:(c1)$);
        \draw [blue,thick] ($(M2)!5cm!270:(c2)$) -- ($(M2)!2.5cm!90:(c2)$);
        \draw [blue,thick] ($(M3)!5cm!270:(c3)$) -- ($(M3)!3cm!90:(c3)$);
        
        \node at (2.4,0.2) {$A$};
        \node[draw] at (0.75,1) {$c_1 > c_3 >c_2$};
        \node[draw] at (-1.3,3.5) {$H(c_1, c_3)$};
         \node [draw] at (1, 3.5) {$c_3 > c_1 > c_2$};
         \node[draw] at (2.5,4) {$H(c_1, c_2)$};
         \node [draw] at (3.8, 3.5){$c_3 > c_2 >c_1$};
         \node [draw] at (4.2, 0.3) {$c_2 > c_3 > c_1$};
         \node [draw] at (4, -2) {$c_2 > c_1 > c_3$};
          \node [draw] at (1.5, -2) {$c_1 > c_2 > c_3$}; 
          \node[draw] at (6,3){$H(c_2, c_3)$};
          
         \node[draw] at (1.5,2.5){$R_1, R_2$};
         \node[draw] at (3.5, 2.5){$R_3, R_4$};
         \node[draw] at (3.5,-1){$R_5, R_6$};
          \node[draw] at (2,-1){$R_7, R_8$};
    \end{tikzpicture}    
    \caption{The different preference areas obtained from candidates $c_1$, $c_2$ and $c_3$ forming a triangle. Each preference $R_v$ is a subset of an area, the precise contours of which depends on the position of $c_4$ in $\mathbb{R}^2$.}
    \label{fig:ex_intro_l2}
\end{figure}
\noindent
Note that for each $i\!\in\!\{ 1,2,3,4\}$, we obtain $R_{2i -1}$ from $R_{2i}$\footnote{\textcolor{black}{These pairs correspond to rows in the display of the profile given at the beginning of example}} by swapping $c_3$ and $c_4$. Thus, $H(c_3, c_4)$ has to go through the area $c_3 > c_1 >c_2$ to separate $R_1$ and $R_2$, through the area $c_3 > c_2 >c_1$ to separate $R_3$ and $R_4$ and finally through the areas $c_2 > c_1 >c_3$ and $c_1 >c_2 >c_3$ to separate $R_5$ and $R_6$, and $R_7$ and $R_8$ (see Figure \ref{fig:ex_intro_l2} for more clarity). This is not possible, as any straight line can cross at most 3 of these 4 areas. \textcolor{black}{Indeed, if a straight line crosses both the area containing $\{R_7,R_8\}$ and the one containing $\{R_5,R_6\}$, then it must intersect $H(c_1,c_2)$ below point $A$. Similarly, if it crosses both the area containing $\{R_1,R_2\}$ and the one containing $\{R_3,R_4\}$, then it must intersect $H(c_1,c_2)$ above point $A$. Thus, to cross the 4 areas, it must intersect $H(c_1,c_2)$ twice, a contradiction.}

Hence, no $\ell_2$-Euclidean representation of $\mathcal{P}$ exists in $\mathbb{R}^2$. 
\qed
\end{example}

\subsection{\textcolor{black}{Relation} between norms $\ell_1$ and $\ell_\infty$}

We consider here the case where $\ell = \ell_1$ or $\ell = \ell_\infty$. Given $\delta \geq 0$, we denote by $\mathcal{S}^\ell_\delta(p)$ the $\ell$-sphere of radius $\delta$ centered in $p \in \mathbb{R}^d$. Formally: 
\begin{equation}\nonumber \mathcal{S}^\ell_\delta (p) = \{ q \in \mathbb{R}^d: \| p - q \|_\ell = \delta \} \end{equation}
With this notation, we characterise $H(c_1, c_2)$ as: 
\begin{equation}\label{eq:1} H(c_1, c_2) = \bigcup_{\delta \geq 0} (\mathcal{S}^\ell_\delta(f(c_1)) \cap \mathcal{S}^\ell_\delta(f(c_2)) ) \end{equation}
\textcolor{black}{For $d=2$, for all $\delta\!\geq\!0$ and $p\!\in\! \mathbb{R}^2$, the spheres $\mathcal{S}^{\ell_1}_\delta(p)$ and $\mathcal{S}^{\ell_\infty}_{\delta/\sqrt{2}}(p)$ are homothetic via the rotation of $45$° (see Figure~\ref{fig:spheres_l1_l2}).  Together with the characterisation of $H(c_1, c_2)$ in Equation~(\ref{eq:1}), this yields the following observation, already noted by~\cite{LeeWongSIAM}.} 
\begin{obs}[\citealp{LeeWongSIAM}]\label{obs:l1linf}
A preference profile  is \lunprofile in $\mathbb{R}^2$ if and only if it is \linfprofile in $\mathbb{R}^2$. 
\end{obs} 

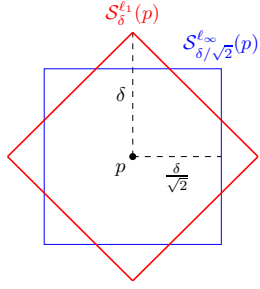
\begin{figure}[tb]
    \centering
    \scalebox{0.75}{\begin{tikzpicture}
   \begin{axis}
   [axis equal,hide axis,axis x line=bottom,axis y line = left, 
   xmin=2,
   xmax=7.5,
   ymin=2,
   ymax=7.5,
   nodes near coords,
   point meta=explicit symbolic]
   \draw[fill] (axis cs:4.5,4.5) circle [radius=1.5pt];
   \node[below left] at (axis cs:4.5,4.5) {$p$};
   \node[above,red] at (axis cs:4.5,6.62132) {$\mathcal{S}^{\ell_1}_\delta(p)$};
   \node[above,blue] at (axis cs:6,6) {$\mathcal{S}^{\ell_\infty}_{\delta/\sqrt{2}}(p)$};
   \draw[-, dashed](axis cs:4.5,4.5)--(axis cs:4.5,6.62132) node[midway, left] {$\delta$};
   \draw[-, dashed](axis cs:4.5,4.5)--(axis cs:6,4.5) node[midway, below] {$\frac{\delta}{\sqrt{2}}$};
    \draw[-, blue] (axis cs:3,3) rectangle (axis cs:6,6);
    \addplot+[mark = none, red, thick] coordinates{(4.5,6.62132) (6.62132,4.5)};
    \addplot+[mark = none, red, thick] coordinates{(6.62132,4.5) (4.5,2.37868)};
    \addplot+[mark = none, red, thick] coordinates{(4.5,2.37868) (2.37868,4.5)};
    \addplot+[mark = none, red, thick] coordinates{(2.37868,4.5) (4.5,6.62132)};
    \end{axis}
    \end{tikzpicture}}
    \caption{\textcolor{black}{For $d=2$,} the spheres $\mathcal{S}^{\ell_1}_\delta(p)$ and $\mathcal{S}^{\ell_\infty}_{\frac{\sqrt{2}}{2}\delta}(p)$ are homothetic via the rotation of $45$°.}
    \label{fig:spheres_l1_l2}
\end{figure}

\textcolor{black}{We now show that this equivalence is {\it not} true for $d\!\geq\!3$. This is actually a corollary of the following proposition, which provides a structural property of \linfprofile and \lunprofile profiles in $\mathbb{R}^d$.}

\begin{prop}\label{prop:2d}
\textcolor{black}{In an \linfprofile profile in $\mathbb{R}^d$, at most $2d$ candidates are ranked last by at least one voter. In an \lunprofile profile in $\mathbb{R}^d$, at most $2^d$ candidates are ranked last by at least one voter.  These bounds are tight for all $d$.}
\end{prop}
\begin{proof}
\textcolor{black}{Let us first consider an \linfprofile profile $\mathcal{P} = (V,C)$, and a corresponding mapping $f$. We denote by $f_i(x)$ the position of candidate/voter $x$ on the $i^{th}$ coordinate. For $i=1,\dots,d$, let us denote by $c_{j_i}^m$ and $c_{j_i}^M$ the candidates that have minimal and maximal $i^{th}$ coordinate. There are at most $2d$ of them (some candidates may be extremal on several coordinates). Take a candidate $c$ which is {\it not} among these extremal candidates, and take a voter $v$. We show that $c$ cannot be ranked last by $v$. Let $i$ be such that $\| f(c) - f(v) \|_{\ell_\infty}=|f_i(c)-f_i(v)|$.} 

\textcolor{black}{If $f_i(c)\geq f_i(v)$, then $$\| f(c) - f(v) \|_{\ell_\infty}=f_i(c)-f_i(v)\leq f_i(c_{j_i}^M)-f_i(v)\leq \| f(c_{j_i}^M) - f(v) \|_{\ell_\infty}.$$
As the two distances cannot be the same (no tie in the preferences), $c$ is ranked before $c_{j_i}^M$ by $v$.} 

\textcolor{black}{If $f_i(c)< f_i(v)$, then $$\| f(c) - f(v) \|_{\ell_\infty}=f_i(v)-f_i(c)\leq f_i(v)-f_i(c_{j_i}^m)\leq \| f(c_{j_i}^m) - f(v) \|_{\ell_\infty}.$$
Again, $c$ is ranked before $c_{j_i}^m$ by $v$. In both cases, $c$ is not ranked last.
}

\textcolor{black}{To show the tightness of the bound, we consider a profile on $2d$ candidates where $f(c_{2i-1})$ is $-1$ on coordinate $i$ and 0 on all other coordinates, and $f(c_{2i})=-f(c_{2i-1})$. There are also $2d$ voters, with $f(v_{i})=f(c_i)$ for $i\!=\!1,\ldots,2d$. Then it is easy to see that $c_{2i-1}$ is ranked last by $v_{2i}$, and $c_{2i}$ is ranked last by $v_{2i-1}$.\footnote{\textcolor{black}{Note that, defined like this, there are some ties in the distances among the candidates that are not ranked last, but these can be easily broken by slightly moving the positions, for instance moving $c_{2i-1}$ (resp. $c_{2i}$) by $-\epsilon_i$ (resp. $+\epsilon_i$) on the $i^{th}$ coordinate, with $\epsilon_i\neq \epsilon_j$ for $i\neq j$.}}}\\

\textcolor{black}{Let us now focus on \lunprofile profiles. For each vector $u$ in $\{-1,1\}^d$, let $c_u$ be a candidate which maximizes $u \cdot f(c)=\sum_{i=1}^d u_i \cdot f_i(c)$. As previously, consider a candidate $c$ which is not among these (at most) $2^d$ extreme candidates, and take a voter $v$. We show that $c$ cannot be ranked last by $v$. By definition, $\| f(c) - f(v) \|_{\ell_1}=\sum_{i=1}^d |f_i(c)-f_i(v)|$. Define the vector $u$ as $u_i=1$ if $f_i(c)\geq f_i(v)$ and $u_i=-1$ otherwise. Then:    
\begin{align*}
\| f(c) - f(v) \|_{\ell_1} & =  \textstyle\sum_{i=1}^d u_i \cdot (f_i(c)-f_i(v)) =u \cdot f(c)-u \cdot f(v) \\ & \leq u  \cdot f(c_u)-u \cdot f(v)  \leq \textstyle\sum_{i=1}^d |f_i(c_u)-f_i(v)| =\| c_u - v \|_{\ell_1}.
\end{align*}
As the two distances must be different (no tie in the preferences), $c$ is not ranked last by $v$.} 

\textcolor{black}{To show the tightness, let us consider the following profile on $2^d$ candidates and $2^d$ voters. For each vector $u\in \{-1,1\}^d$, we define a candidate $c_u$ with $f(c_u)=u$, and a voter $v_u$ with $f(v_u)=-u$. Then we have $\| f(c_u) - f(v_u) \|_{\ell_1}=2d$, while if $u\neq u'$ we have $\| f(c_{u'}) - f(v_u) \|_{\ell_1}\leq 2(d-1)$ (as $f_i(c_{u'})=f_i(v_u)$ on at least one coordinate $i$). So $c_u$ is ranked last by $v_u$\footnote{\textcolor{black}{As previously, the ties between distances among candidates that are not ranked last can be removed by slightly moving the positions.}}. 
}
\end{proof}

From now on, and throughout the remainder of the article, we fix $d\!=\!2$, i.e., we consider a representation of the preferences in the plane. For the sake of brevity, we omit to mention ``in $\mathbb{R}^2$'' in the following. Given Observation~\ref{obs:l1linf}, we can use $\ell_1$ or $\ell_\infty$ indifferently. We choose to use $\ell_1$.

\section{Properties of hypersurfaces under $\ell_1$ in the plane}\label{sec:prop}

We  give in this section some properties of (boundary) \hyps~under $\ell_1$. These properties will be useful to obtain the results on $\ell_1$-Euclidean profiles in Sections~\ref{sec:n=4} and~\ref{sec:n>=5}. They have also their own interest, giving some geometric insights on the differences between Euclidean profiles under $\ell_1$ and $\ell_2$. 

{\it Note.} For ease of notation, when the position of candidates are fixed, $c_i$ will denote both the candidate and her position in $\mathbb{R}^2$ (i.e., $f(c_i)$ in the above notation).

\subsection{Types of \hyps}\label{subsec:typehyper}

We first focus on the description of the \hyps separating two points $c_1$ and $c_2$. In contrast to the $\ell_2$ metrics where the \hyp is always a straight line (if $d\!=\!2$), the shape of this \hyp depends on the relative positions of $c_1$ and $c_2$ when using the $\ell_1$ metrics, as we will now show. We denote by $(x_1,y_1)$ (resp. $(x_2, y_2)$) the coordinates of $c_1$ (resp. $c_2$), and we use the notations  $\Delta x = \vert x_1 - x_2 \vert $ and $\Delta y = \vert y_1 - y_2 \vert $. \\ \\
\noindent
\begin{enumerate}
    \item Let us first consider the case $\Delta x\!\neq\! \Delta y$, with $\Delta x\!>\!0$ and $\Delta y\!>\!0$. This case is illustrated in Figure~\ref{fig:hyperplan_l1_base}. 
\begin{figure}[tb]
    \centering
    \begin{tikzpicture}
   \begin{axis}
   [axis x line=bottom,axis y line = left, 
   grid = major,
   axis equal image,
   ytick = {1,2,3,4,5,6,7,8,9},
   xtick = {1,2,3,4,5,6,7,8,9},
   xmin=0,
   xmax=10,
   ymin=0,
   ymax=10,
   nodes near coords,
   point meta=explicit symbolic]
   \addplot+[only marks] coordinates{(3,3)[$c_1$] (7,5)[$c_2$]};
   \draw[<->, thick](axis cs:1,3)--(axis cs:1,5) node[midway, right] {$\Delta y$};
   \draw[<->, thick](axis cs:3,7)--(axis cs:7,7) node[midway, above] {$\Delta x$};
    \addplot+[only marks, red] coordinates{(4,5)[$M_1$] (6,3)[$M_2$] (5,4)[$C$]};
    \draw[-, blue] (axis cs:3,3) rectangle (axis cs:7,5);
    \addplot+[mark = none, red, thick] coordinates{(4,9) (4,5) (6,3) (6,0)};
    \end{axis}
    \end{tikzpicture}
    \caption{A (boundary) \hyp separating $c_1$ and $c_2$: $\Delta x \neq \Delta y, \Delta x, \Delta y > 0$.}
    \label{fig:hyperplan_l1_base}
\end{figure}
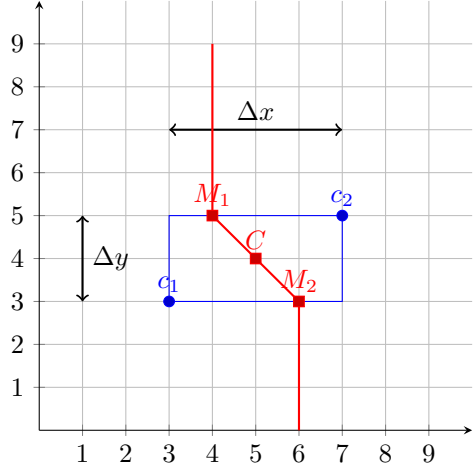
\noindent 
Without loss of generality, assume that $\Delta x > \Delta y$ (the case $\Delta x < \Delta y$ can be treated analogously). The positions $c_1$ and $c_2$ can be seen as two opposite vertices of a rectangle (see Fig. \ref{fig:hyperplan_l1_base}). \\
\noindent
\begin{itemize}
    \item By definition of the $\ell_1$ metrics, there are two points $M_1, M_2$ on the rectangle boundary that belong to $H(c_1, c_2)$: these are the points that are at distance $\frac{\Delta x + \Delta y}{2}$ from  both points $c_1$ and $c_2$. Points $M_1$ and $M_2$ are symmetric with respect to the rectangle center, and we observe that the segment $[M_1, M_2] $ belongs to $H(c_1, c_2)$ - in fact, we observe that: 
$$ [M_1, M_2] = \mathcal{S}^{\ell_1}_{\frac{\Delta x + \Delta y}{2}}(c_1) \cap \mathcal{S}^{\ell_1}_{\frac{\Delta x + \Delta y}{2}}(c_2) $$
   \item The half-line $\{(x_{M_1}, y): y \geq y_{M_1} \}$ also belongs to $H(c_1, c_2)$, where $x_{M_1}$ and $y_{M_1}$ denote the coordinates of $M_1$, as for $y\geq y_{M_1}$, each point $(x_{M_1},y)$ is at distance $\frac{\Delta x+\Delta y }{2}+y-y_{M_1}$ both from $c_1$ and $c_2$.
   \item Similarly, the half-line $\{(x_{M_2}, y) \vert y \leq y_{M_2} \}$ belongs to $H(c_1, c_2)$.
\end{itemize}
\noindent
To sum it up, we have identified three parts of $H(c_1, c_2)$: two vertical half-lines connected by a diagonal segment. We can easily prove that for each $z \in \mathbb{R}^2$ that does not belong to one of these parts, we have $\| z - c_1 \|_{\ell_1} \neq \| z - c_2 \|_{\ell_1}$. More precisely, the points to the left-hand side of the \hyp are closer to $c_1$, while the ones on the right-hand side are closer to $c_2$.  
\item Let us now consider the case $\Delta x = \Delta y >0$. This case is illustrated in Figure~\ref{fig:hyperplan_l1_carre}.
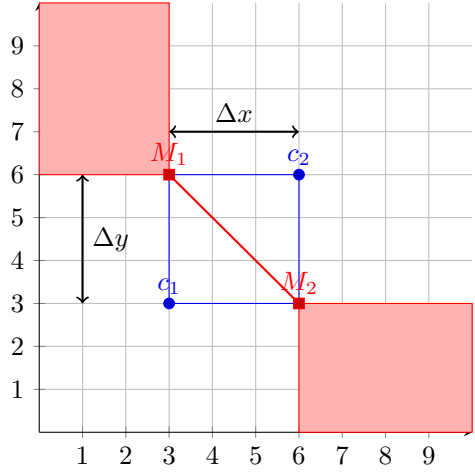
\begin{figure}[tb]
    \centering
    \begin{tikzpicture}
   \begin{axis}
   [axis x line=bottom,axis y line = left, 
   grid = major,
   axis equal image,
   ytick = {1,2,3,4,5,6,7,8,9},
   xtick = {1,2,3,4,5,6,7,8,9},
   xmin=0,
   xmax=10,
   ymin=0,
   ymax=10,
   nodes near coords,
   point meta=explicit symbolic]
   \addplot+[only marks] coordinates{(3,3)[$c_1$] (6,6)[$c_2$]};
   \draw[<->, thick](axis cs:1,3)--(axis cs:1,6) node[midway, right] {$\Delta y$};
   \draw[<->, thick](axis cs:3,7)--(axis cs:6,7) node[midway, above] {$\Delta x$};
    \addplot+[only marks, red] coordinates{(3,6)[$M_1$] (6,3)[$M_2$]};
    \draw[-, blue] (axis cs:3,3) rectangle (axis cs:6,6);
    \addplot+[mark = none, red, thick] coordinates{(3,6) (6,3)};
    \filldraw[draw=red,fill=red!30] (axis cs:0,6) rectangle (axis cs:3,10);
    \filldraw[draw=red,fill=red!30] (axis cs:6,0) rectangle (axis cs:10,3);
    \end{axis}
    \end{tikzpicture}
    \caption{The \hyp separating $c_1$ and $c_2$ if $\Delta x = \Delta y>0$.}
    \label{fig:hyperplan_l1_carre}
\end{figure}
\noindent
In this special case, the rectangle is a square where $c_1$ and $c_2$ are opposite vertices, and $M_1$ and $M_2$ are the two other opposite vertices. The \hyp $H(c_1,c_2)$ is then composed of the three following parts:
\begin{enumerate}
    \item the quadrant $\{ (x,y) \in \mathbb{R}^2 \vert x \leq x_{M_1}, y \geq y_{M_1} \}$,
    \item the segment $[M_1, M_2]$,
    \item the quadrant $\{ (x,y) \in \mathbb{R}^2 \vert x \geq x_{M_1}, y \leq y_{M_1} \}$.
\end{enumerate}
\item Consider now the case $\Delta x\!=\!0$ or $\Delta y\!=\!0$. Clearly, the \hyp is then the same as for the $\ell_2$ metrics (i.e., a straight line at equal $\ell_2$ distance from $c_1$ and $c_2$). 
\end{enumerate}
\noindent

The following result states that, to recognise an \lunprofile profile, we can assume without loss of generality that all \hyps are of the first type described above:

\begin{lem}\label{lemma:type1}
Let $\mathcal{P}$ be an \lunprofile profile. There exists a representation of $\mathcal{P}$ in which all \hyps are of type 1, i.e., $\Delta x\neq \Delta y$ and $\Delta x, \Delta y >0$.
\end{lem}
\begin{proof}
In an $\ell_1$-Euclidean representation of a preference profile, as we consider only strict preferences, we have for any candidates $c_i,c_j$ and voter $v$:
\begin{equation}\label{eq:epsilon}
\left|  \ \| f(v) - f(c_i) \|_{\ell_1} - \| f(v) - f(c_j) \|_{\ell_1} \  \right| >0 
\end{equation}
Then, let us denote by \textcolor{black}{$\varepsilon_d$} the minimum difference in absolute value of distances as in~(\ref{eq:epsilon}), over all pairs $\{c_i,c_j\}$ of candidates and voters $v$. \textcolor{black}{Moreover, let $(x_i, y_i)$ be the position of candidate $c_i$ in the representation, and $S_x$ (resp. $S_y$, $S_{xy}$) the set of pairs of candidates $\{c_i, c_j\}$ with $|x_i-x_j|>0$ 
(resp. $|y_i-y_j|>0$, $\Big| |x_i - x_j| - |y_i - y_j|\Big|>0$). We define also: 
\os{
\begin{align*} 
\varepsilon_x & = \min_{\{c_i, c_j\} \in S_x} | x_i - x_j|,\\
\varepsilon_y & = \min_{\{c_i, c_j\} \in S_y } | y_i - y_j|,\\
\mbox{and } \varepsilon_{xy} & =  \min_{ \{c_i, c_j\} \in S_{xy} } \Big| |x_i - x_j| - |y_i - y_j|\Big|.
\end{align*}
}
If there is a pair $\{ c_i, c_j\}$ such that $x_i\!=\!x_j$ (in other words, $\{c_i,c_j\}\!\notin\! S_x$), we can move one of these candidates, say $c_i$, 
by adding $\varepsilon$ to $x_i$ with $\varepsilon\!=\!\frac{1}{2} \min \{ \varepsilon_d, \varepsilon_x, \varepsilon_y, \varepsilon_{xy}\}$. We then get $|x_i - x_j|\!>\!0$. We note that after this operation, we have $S_x \leftarrow S_x \cup \{\{ c_i, c_j\}\}$ and $S_y$ and $S_{xy}$ are not modified. An analogous reasoning can be done for every pair $\{c_i, c_j\}$ of candidates such that $\{ c_i, c_j\} \notin S_y$ (by moving one candidate on the $y$-axis), resp. $\{ c_i, c_j\} \notin S_{xy}$ (by moving one candidate on one axis). This way, by iterating these modifications, we finally get a representation without the degenerated cases $\Delta x=\Delta y$, $\Delta x = 0$, or $\Delta y=0$.} 
\end{proof}

Hence, without loss of generality, we assume that all $\ell_1$-\hyps are of type 1 in the following. 

We can go further into the classification of the different \hyps of type 1. First, notice that
if $\Delta x\!<\!\Delta y$,  both half-line parts of the \hyp are horizontal. In the opposite case, when $\Delta x > \Delta y$, these half-lines are vertical. 

Now, let us look at the segment $[M_1, M_2]$ of the \hyp\hspace{-0.125cm}. In the following, the numbering of the quadrants of the Cartesian coordinate system goes counter-clockwise starting from the upper right quadrant. Without loss of generality, assume that $x_1 < x_2$, where $c_1 = (x_1, y_1)$ and $c_2 = (x_2, y_2)$. If $y_1 < y_2$, the segment $[M_1, M_2]$ is parallel to the II-IV quadrant diagonal, also called the ``minus diagonal'' (see the upper part of Figure~\ref{fig:hyps_overview}). If $y_1 > y_2$, the segment $[M_1, M_2]$ is parallel to the I-III quadrant diagonal, also called the ``plus diagonal'' (see the lower part of Figure~\ref{fig:hyps_overview}).

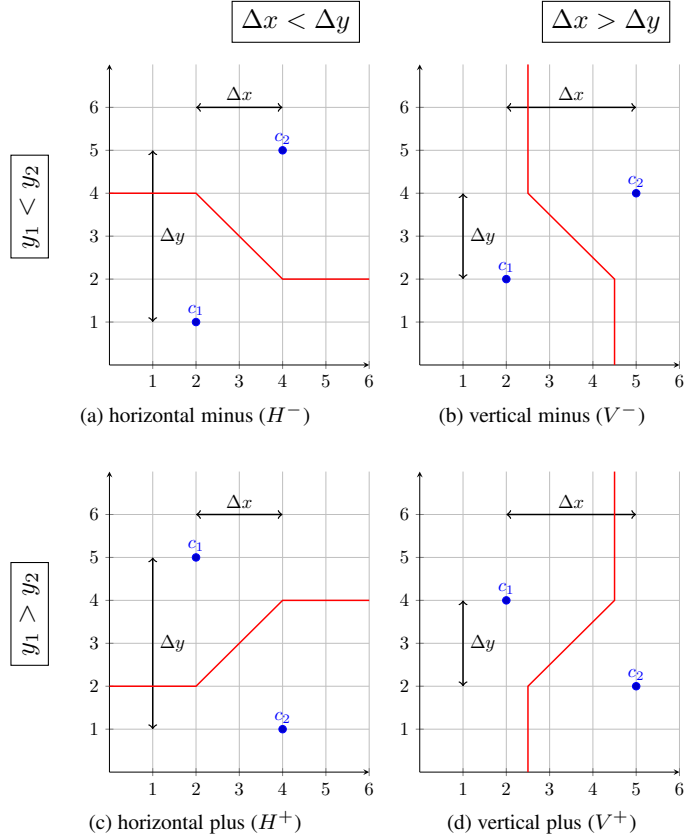
\begin{figure}[tb]
    \centering
    \subfloat[\centering horizontal minus ($H^-$)]{
    \centering
    \begin{tikzpicture}[scale=0.7]
   \node [draw, rotate=90] at (-1.5,3) {$y_1 < y_2$};
   \node [draw] at (3.5,6.5) {$\Delta x < \Delta y$};
   \begin{axis}
   [axis x line=bottom,axis y line = left, 
   grid = major,
   axis equal image,
   ytick = {1,2,3,4,5,6},
   xtick = {1,2,3,4,5,6},
   xmin=0,
   xmax=6,
   ymin=0,
   ymax=7,
   nodes near coords,
   point meta=explicit symbolic]
   \addplot+[only marks] coordinates{(2,1)[$c_1$] (4,5)[$c_2$]};
   \draw[<->, thick](axis cs:1,1)--(axis cs:1,5) node[midway, right] {$\Delta y$};
   \draw[<->, thick](axis cs:2,6)--(axis cs:4,6) node[midway, above] {$\Delta x$};
    \addplot+[mark = none, red, thick] coordinates{(0,4) (2,4) (4,2) (6,2)};
    \end{axis}
    \end{tikzpicture}
    }
    \subfloat[\centering vertical minus ($V^-$)]{
    \begin{tikzpicture}[scale = 0.7]
    \node [draw] at (3.5,6.5) {$\Delta x > \Delta y$};
   \begin{axis}
   [axis x line=bottom,axis y line = left, 
   grid = major,
   axis equal image,
   ytick = {1,2,3,4,5,6},
   xtick = {1,2,3,4,5,6},
   xmin=0,
   xmax=6,
   ymin=0,
   ymax=7,
   nodes near coords,
   point meta=explicit symbolic]
   \addplot+[only marks] coordinates{(2,2)[$c_1$] (5,4)[$c_2$]};
   \draw[<->, thick](axis cs:1,2)--(axis cs:1,4) node[midway, right] {$\Delta y$};
   \draw[<->, thick](axis cs:2,6)--(axis cs:5,6) node[midway, above] {$\Delta x$};
    \addplot+[mark = none, red, thick] coordinates{(4.5,0) (4.5,2) (2.5,4) (2.5,7)};
    \end{axis}
    \end{tikzpicture}
    } \\
    \subfloat[\centering horizontal plus ($H^+$)]{
    \begin{tikzpicture}[scale=0.7]
    \node [draw, rotate=90] at (-1.5,3) {$y_1 > y_2$};
   \begin{axis}
   [axis x line=bottom,axis y line = left, 
   grid = major,
   axis equal image,
   ytick = {1,2,3,4,5,6},
   xtick = {1,2,3,4,5,6},
   xmin=0,
   xmax=6,
   ymin=0,
   ymax=7,
   nodes near coords,
   point meta=explicit symbolic]
   \addplot+[only marks] coordinates{(2,5)[$c_1$] (4,1)[$c_2$]};
   \draw[<->, thick](axis cs:1,1)--(axis cs:1,5) node[midway, right] {$\Delta y$};
   \draw[<->, thick](axis cs:2,6)--(axis cs:4,6) node[midway, above] {$\Delta x$};
    \addplot+[mark = none, red, thick] coordinates{(0,2) (2,2) (4,4) (6,4)};
    \end{axis}
    \end{tikzpicture}
    }
    \subfloat[\centering vertical plus ($V^+$)]{
    \begin{tikzpicture}[scale = 0.7]
   \begin{axis}
   [axis x line=bottom,axis y line = left, 
   grid = major,
   axis equal image,
   ytick = {1,2,3,4,5,6},
   xtick = {1,2,3,4,5,6},
   xmin=0,
   xmax=6,
   ymin=0,
   ymax=7,
   nodes near coords,
   point meta=explicit symbolic]
   \addplot+[only marks] coordinates{(2,4)[$c_1$] (5,2)[$c_2$]};
   \draw[<->, thick](axis cs:1,2)--(axis cs:1,4) node[midway, right] {$\Delta y$};
   \draw[<->, thick](axis cs:2,6)--(axis cs:5,6) node[midway, above] {$\Delta x$};
    \addplot+[mark = none, red, thick] coordinates{(4.5,7) (4.5,4) (2.5,2) (2.5,0)};
    \end{axis}
    \end{tikzpicture}
    }
    \caption{\label{fig:hyps_overview}The different $\ell_1$-\hyps for $x_1 < x_2$.}
\end{figure}

Now that we have seen the shape of \hyps for $\ell_1$, we illustrate this by giving a representation of the profile on 3 candidates that includes all $3!\!=\!6$ possible strict preferences over the 3 candidates (\emph{complete} profile).

\begin{example}[Example~\ref{ex:n=3ell2} continued]
Figure~\ref{fig:l1_c3} shows that the complete profile on 3 candidates (with 6 preferences) is $\ell_1$-Euclidean, by providing a $\ell_1$-Euclidean representation of the profile. 
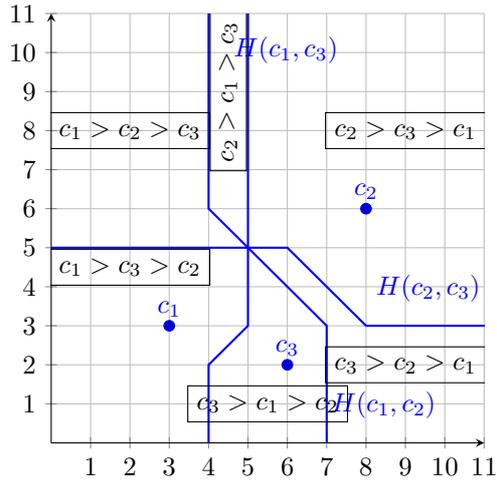
\begin{figure}[tb]
    \centering
    \begin{tikzpicture}
   \begin{axis}
   [axis x line=bottom,axis y line = left, 
   grid = major,
   axis equal image,
   ytick = {1,2,3,4,5,6,7,8,9,10,11},
   xtick = {1,2,3,4,5,6,7,8,9,10,11},
   xmin=0,
   xmax=11,
   ymin=0,
   ymax=11,
   nodes near coords,
   point meta=explicit symbolic]
   \addplot+[only marks] coordinates{(3,3)[$c_1$] (8,6)[$c_2$] (6,2)[$c_3$]};
    \addplot+[mark = none, blue, thick] coordinates{(4,12) (4,6) (7,3) (7,0)} node[xshift = 0.75cm, yshift = 0.5cm] {$H(c_1, c_2)$};
    
     \addplot+[mark = none, blue, thick] coordinates{(4,0) (4,2) (5,3) (5,12)} node[xshift = 0.5cm, yshift = -1cm] {$H(c_1, c_3)$};
     
     \addplot+[mark = none, blue, thick] coordinates{(0,5) (6,5) (8,3) (12,3)} node[xshift = -1.25cm, yshift = 0.5cm] {$H(c_2, c_3)$};
     
     \node[draw] at (axis cs:2,4.5) {$c_1 > c_3 > c_2 $};
     \node[draw] at (axis cs:2,8) {$c_1 > c_2 > c_3 $};
     \node[draw,rotate = 90] at (axis cs:4.5,9) {$c_2 > c_1 > c_3 $};
     \node[draw] at (axis cs:9,8) {$c_2 > c_3 > c_1 $};
     \node[draw] at (axis cs:9,2) {$c_3 > c_2 > c_1 $};
      \node[draw] at (axis cs:5.5,1) {$c_3 > c_1 > c_2 $};
    
    \end{axis}
    \end{tikzpicture}
    \caption{A $\ell_1$-Euclidean representation of the complete profile on 3 candidates.}
    \label{fig:l1_c3}
\end{figure}
\end{example}

More interestingly, we show that the profile on 4 candidates given in Example~\ref{ex:n=4ell2} is $\ell_1$-Euclidean (while it is not $\ell_2$-Euclidean).

\begin{example}[Example~\ref{ex:n=4ell2} continued]
Figure~\ref{fig:ex_intro_l1} shows that profile $\mathcal{P}$ of Example~\ref{ex:n=4ell2} is \lunprofile in $\mathbb{R}^2$, by providing a $\ell_1$-Euclidean representation of the profile.
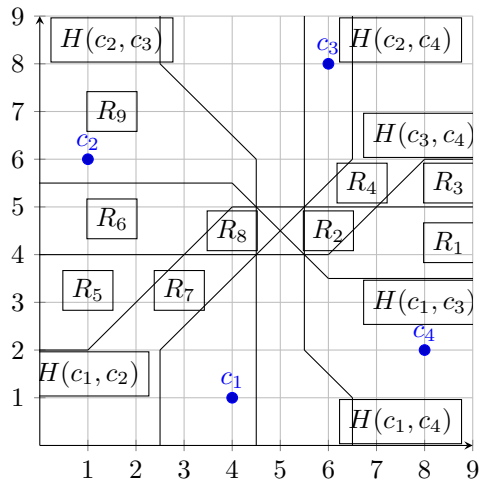
\begin{figure}[tb]
    \centering
    \begin{tikzpicture}[scale = 1]
   \begin{axis}
   [axis x line=bottom,axis y line = left, 
   grid = major,
   axis equal image,
   ytick = {1,2,3,4,5,6,7,8,9,10,11},
   xtick = {1,2,3,4,5,6,7,8,9,10,11},
   xmin=0,
   xmax=9,
   ymin=0,
   ymax=9,
   nodes near coords,
   point meta=explicit symbolic]
   \addplot+[only marks] coordinates{(4,1)[$c_1$] (1,6)[$c_2$] (6,8)[$c_3$] (8,2)[$c_4$]};
   \addplot+[mark = none, color = black] coordinates{(4,5.5) (6,3.5)};
   \addplot+[mark = none, color = black] coordinates{(-2,5.5) (4,5.5)};
   \addplot+[mark = none, color = black] coordinates{(6,3.5) (12,3.5)};
   \node[draw] at (axis cs:8,3) {$H(c_1, c_3)$};
    
   \draw[color = black] (axis cs:4,5) -- (axis cs:1,2);
   \draw[color = black] (axis cs:-2,2) -- (axis cs:1,2);
   \draw[color = black] (axis cs:4,5)--(axis cs:12,5);
   \node[draw] at (axis cs:1,1.5) {$H(c_1, c_2)$};
    
    \draw[color = black] (axis cs:6.5,1) -- (axis cs:5.5,2);
    \draw[color = black] (axis cs:6.5,-2) -- (axis cs:6.5,1);
    \draw[color = black] (axis cs:5.5,2)--(axis cs:5.5,12);
   
    \node[draw] at (axis cs:7.5,0.5) {$H(c_1, c_4)$};
    
    \draw[color = black] (axis cs:4.5,6) -- (axis cs:2.5,8);
    \draw[color = black] (axis cs:4.5,-2) -- (axis cs:4.5,6);
    \draw[color = black] (axis cs:2.5,8)--(axis cs:2.5,12);
    \node[draw] at (axis cs:1.5,8.5) {$H(c_2, c_3)$};
    
    \draw[color = black] (axis cs:6.5,6) -- (axis cs:2.5,2);
    \draw[color = black] (axis cs:2.5,-2) -- (axis cs:2.5,2);
    \draw[color = black] (axis cs:6.5,6)--(axis cs:6.5,12);
    \node[draw] at (axis cs:7.5,8.5) {$H(c_2, c_4)$};
    
    \draw[color = black] (axis cs:8,6) -- (axis cs:6,4);
    \draw[color = black] (axis cs:-2,4) -- (axis cs:6,4);
    \draw[color = black] (axis cs:8,6)--(axis cs:12,6);
    \node[draw] at (axis cs:8,6.5) {$H(c_3, c_4)$};
   
   \node[draw] at (axis cs:8.5,4.25) {$R_1$};
   \node[draw] at (axis cs:6,4.5) {$R_2$};
   \node[draw] at (axis cs:8.5,5.5) {$R_3$};
   \node[draw] at (axis cs:6.7,5.5) {$R_4$};
   \node[draw] at (axis cs:1,3.25) {$R_5$};
   \node[draw] at (axis cs:1.5,4.75) {$R_6$};
   \node[draw] at (axis cs:2.9,3.25) {$R_7$};
   \node[draw] at (axis cs:4,4.5) {$R_8$};
   
   \node[draw] at (axis cs:1.5,7) {$R_9$};
    \end{axis}
    \end{tikzpicture}
    \caption{A $\ell_1$-Euclidean representation of the profile $\mathcal{P}$ of Example~\ref{ex:n=4ell2}.}
    \label{fig:ex_intro_l1}
\end{figure}
\end{example}

\subsection{Intersection of boundary \hyps}\label{subsec:interhyper}

It will come as no surprise that many geometrical properties holding for $\ell_2$ do not hold for $\ell_1$-\hyps\hspace{-0.11cm}. Let us mention some of them that are useful for the rest of this paper. \\ \\  
\noindent
It is well-known that given two distinct lines (i.e., $\ell_2$-\hyps\hspace{-0.11cm}), the intersections of these lines is either empty (if the lines are parallel) or contains a unique point. In the case of $\ell_1$-\hyps\hspace{-0.11cm}, more situations may arise, as stated in the following proposition (\textcolor{black}{several examples of possible intersections are given in Figure~\ref{fig:intersections_exemple_main} for illustration. The proof, as well as complete figures illustrating the different situations, are given in Appendix~\ref{app:prop3}}). 
\begin{figure}
    \centering
    \subfloat[Example of empty intersection]{\begin{tikzpicture}[scale=0.5]
   \begin{axis}
   [axis x line=bottom,axis y line = left, 
   grid = major,
   axis equal image,
   ytick = {1,2,3,4,5,6},
   xtick = {1,2,3,4,5,6},
   xmin=0,
   xmax=6,
   ymin=0,
   ymax=7,
   nodes near coords,
   point meta=explicit symbolic]
   \addplot+[only marks] coordinates{(0,1)[\Large$c_1$] (2,2)[\Large$c_2$]};
   \addplot+[only marks] coordinates{(2.5,6)[\Large$c_3$] (6,4)[\Large$c_4$]};
  
    \addplot+[mark = none, blue, thick] coordinates{(1.5,0) (1.5,1) (0.5,2) (0.5,7)};
    \addplot+[mark = none, red, thick] coordinates{(3,0) (3,4) (5,6) (5,7)};
    \end{axis}
    \end{tikzpicture}}
    \hspace{0.07cm}
    \subfloat[Example of intersection containing one point]{\begin{tikzpicture}[scale=0.5]
   \begin{axis}
   [axis x line=bottom,axis y line = left, 
   grid = major,
   axis equal image,
   ytick = {1,2,3,4,5,6},
   xtick = {1,2,3,4,5,6},
   xmin=0,
   xmax=6,
   ymin=0,
   ymax=7,
   nodes near coords,
   point meta=explicit symbolic]
   \addplot+[only marks] coordinates{(0,1)[\Large$c_1$] (6,3)[\Large$c_2$]};
   \addplot+[only marks] coordinates{(1,6)[\Large$c_3$] (2,2)[\Large$c_4$]};
  
    \addplot+[mark = none, blue, thick] coordinates{(4,0) (4,1) (2,3) (2,7)};
    \addplot+[mark = none, red, thick] coordinates{(0,3.5) (1,3.5) (2,4.5) (6,4.5)};
    \end{axis}
    \end{tikzpicture}}
     \hspace{0.07cm}
    \subfloat[Example of intersection containing two points]{ \begin{tikzpicture}[scale=0.5]
   \begin{axis}
   [axis x line=bottom,axis y line = left, 
   grid = major,
   axis equal image,
   ytick = {1,2,3,4,5,6},
   xtick = {1,2,3,4,5,6},
   xmin=0,
   xmax=6,
   ymin=0,
   ymax=7,
   nodes near coords,
   point meta=explicit symbolic]
   \addplot+[only marks] coordinates{(0,1)[\Large$c_1$] (6,3)[\Large$c_2$]};
   \addplot+[only marks] coordinates{(0.5,3)[\Large$c_3$] (4,4)[\Large$c_4$]};
  
    \addplot+[mark = none, blue, thick] coordinates{(4,0) (4,1) (2,3) (2,7)};
    \addplot+[mark = none, red, thick] coordinates{(2.75,0) (2.75,3) (1.75,4) (1.75,7)};
    \end{axis}
    \end{tikzpicture}}
     \hspace{0.07cm}
    \subfloat[Example of intersection containing an infinite number of points]{ \begin{tikzpicture}[scale=0.5]
   \begin{axis}
   [axis x line=bottom,axis y line = left, 
   grid = major,
   axis equal image,
   ytick = {1,2,3,4,5,6},
   xtick = {1,2,3,4,5,6},
   xmin=0,
   xmax=6,
   ymin=0,
   ymax=7,
   nodes near coords,
   point meta=explicit symbolic]
   \addplot+[only marks] coordinates{(0,1)[\Large$c_1$] (6,3)[\Large$c_2$]};
   \addplot+[only marks] coordinates{(1,0.5)[\Large$c_3$] (3,5.5)[\Large$c_4$]};
  
    \addplot+[mark = none, blue, thick] coordinates{(4,0) (4,1) (2,3) (2,7)};
    \addplot+[mark = none, red, thick] coordinates{(0,4) (1,4) (3,2) (6,2)};
    \end{axis}
    \end{tikzpicture}}
    
    \caption{\textcolor{black}{The intersection of two distinct $\ell_1$-\hyps: several examples}}
    \label{fig:intersections_exemple_main}
\end{figure}
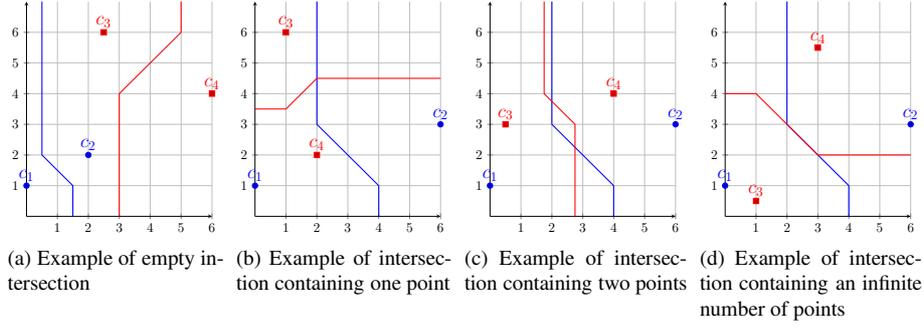
\begin{prop}\label{prop:nb_intersections}
The intersection of two distinct $\ell_1$-\hyps  is either empty or contains  a unique point, or two distinct points, or an infinite number of points. 
\end{prop}

The following result states that, to recognise a $\ell_1$-Euclidean profile, we can assume without loss of generality that the last case of Proposition~\ref{prop:nb_intersections} (corresponding to a degenerate case) never occurs. \textcolor{black}{Thus, in the remainder of the article, we assume w.l.o.g. that \hyps intersect in at most 2 points.}

\begin{lem}\label{lemma:noinfiniteinter}
Let $\mathcal{P}$ be an \lunprofile profile. There exists a representation of $\mathcal{P}$ in which any pair of \hyps intersect in at most 2 points.
\end{lem}
\begin{proof}
\textcolor{black}{The proof is similar to that of Lemma~\ref{lemma:type1}. Let us suppose that for a given $\ell_1$-Euclidean representation $f$, there are two \hyps $H(c_i, c_j)$ and $H(c_k, c_l)$ that intersect in infinitely many points. We will show that it is always possible to slightly change the position of one  of the points $f(c_i), f(c_j), f(c_k)$ and $f(c_l)$ so that the \hyps intersect in at most two points, without modifying the types of \hyps and the set of representation areas. \\ \\
Suppose that one of the hypersurfaces, say $H(c_i,c_j)$, is vertical (the other case being symmetrical). We move the point $f(c_i)$. To do so, we denote 
$$ \varepsilon_t = \min_{c_k \in C \setminus \{ c_i\}}\Big||x_i - x_k | - |y_i - y_k | \Big| $$
and, as in the Lemma~\ref{lemma:type1},
$$ \varepsilon_d = \min_{v \in V} \min_{c_i, c_j \in C}\left|  \ \| f(v) - f(c_i) \|_{\ell_1} - \| f(v) - f(c_j) \|_{\ell_1} \  \right|. $$
As we consider only strict preferences, $\varepsilon_d > 0$. Also, thanks to Lemma~\ref{lemma:type1} that excluded a degeneration $\Delta x = \Delta y$, we have $\varepsilon_t > 0$. Let $\varepsilon = \frac{1}{2}\min\{\varepsilon_d, \varepsilon_t\}$.\footnote{For completeness, we should also choose $\varepsilon$  smaller than $\varepsilon_x, \varepsilon_y$ and $\varepsilon_{xy}$ introduced in Lemma~\ref{lemma:type1}, to ensure that we do not create any degeneration excluded by this Lemma while moving the point $f(c_i)$.} We can now move the point $f(c_i)$ by adding $\varepsilon$ to $x_i$. 
As $\varepsilon < \varepsilon_d$, we do not change the set of preferences corresponding to representation areas. As $\varepsilon < \varepsilon_t$, we do not change the type of any \hyp involving $c_i$. Finally,  
as $x_i$ increased by $\varepsilon>0$, the value of $| x_i - x_j | + | y_i - y_j |$ changes, and the \hyp (both the vertical extremities and the middle segment) slightly moves to the right on the $x-$axis.
Therefore, $H(c_i, c_j)$ and $H(c_k, c_l)$ no more intersect in an infinity of points.  
} 
\end{proof}

In Euclidean geometry \textcolor{black}{under norm $\ell_2$}, the bisectors of the three sides of a \textcolor{black}{non-degenerate} triangle intersect in a unique point\textcolor{black}{, and do not intersect otherwise}. In terms of \hyps\hspace{-0.11cm}, given three points $c_1$, $c_2$ and $c_3$, the \hyps $H(c_1, c_2)$, $H(c_1, c_3)$ and $H(c_2, c_3)$ intersect in \textcolor{black}{at most one} point under $\ell_2$. \textcolor{black}{We have the following analogous result in case of} $\ell_1$-\hyps\hspace{-0.11cm}: 
\begin{prop}\label{prop:triangle_intersection}
Given three points $c_1$, $c_2$ and $c_3$:
\begin{itemize}
    \item If $H(c_1,c_2)$, $H(c_1, c_3)$ and $H(c_2, c_3)$ are all vertical (or all horizontal), then the intersection of each pair of \hyps is empty. In particular, the intersection of the 3 \hyps is empty.
    \item If two of them are vertical and one is horizontal (or vice-versa), then the intersection of the 3 \hyps is a unique point.
\end{itemize}  
\end{prop}

The proof of this Proposition (see Appendix~\ref{app:prop4}) uses the following easy lemma, that we will use in some other proofs as well.

\begin{lem}\label{lemma:inter}
Given three points $c_1$, $c_2$ and $c_3$, we have: 
$$ H(c_1, c_2) \cap H(c_1, c_3) \cap H(c_2, c_3) = H(c_i, c_j) \cap H(c_j, c_k)$$ 
for all $i,j,k$ such that $\{ i,j, k \} = \{ 1, 2, 3 \}$.
\end{lem}
\begin{proof}
The left-right inclusion is obvious. For the right-left inclusion, without loss of generality, assume that $i\!=\!1$, $j\!=\!2$ and $k\!=\!3$, and consider $x \in H(c_i, c_j) \cap H(c_i, c_k)$. Then, $$\| x - c_1 \|_{\ell_1} = \| x - c_2 \|_{\ell_1} = \| x - c_3 \|_{\ell_1} $$
because
$$\begin{array}{l} x\!\in\!H(c_i, c_j) \Rightarrow \| x - c_i \|_{\ell_1} = \| x - c_j \|_{\ell_1},\\ x\!\in\!H(c_i, c_k)\Rightarrow \| x - c_i \|_{\ell_1} = \| x - c_k \|_{\ell_1}.\end{array}$$ Hence, $x \in H(c_1, c_2) \cap H(c_1, c_3) \cap H(c_2, c_3)$. 
\end{proof}

The next result is a direct corollary of Proposition \ref{prop:triangle_intersection} (see Appendix~\ref{app:cor1} for the proof). 
\begin{cor}\label{cor:int}
Given three points $c_i$, $c_j$, $c_k$, the \hyps $H(c_i, c_j)$ and $H(c_i, c_k)$ intersect in at most one point. In other words, \textcolor{black}{given four points  $c_i$, $c_j$, $c_k$ and $c_l$}, if two hypersurfaces $H(c_i, c_j)$ and $H(c_k, c_l)$ intersect in two different points, then $c_i, c_j, c_k$ and $c_l$ are all distinct.
\end{cor}

Proposition \ref{prop:triangle_intersection} can be reformulated by giving conditions on the relative positions of $c_1, c_2$ and $c_3$ rather than the types of \hyps $H(c_1, c_2), H(c_1, c_3)$ and $H(c_2, c_3)$. For this reformulation, let us first define the following notion of parallelogram associated with the positions of 2 candidates (see Figure~\ref{fig:zone_interdit_l1} for an illustration).

\begin{df}\label{def:diag_rect}
Let $c_i$ and $c_j$ be two candidates, and $(x_i, y_i)$, $(x_j, y_j)$ their positions in the two-dimensional plane. Let us denote by:
\begin{itemize}
    \item $d^+_i = \{(x,y) \vert y = x - x_i + y_i \}$ the "+" diagonal going through the point $c_i$,
     \item $d^-_i = \{(x,y) \vert y = - x + x_i + y_i \}$ the "-" diagonal going through the point $c_i$, 
    \item $d^+_j = \{(x,y) \vert y = x - x_j + y_j \}$ the "+" diagonal going through the point $c_j$,
     \item $d^-_j = \{(x,y) \vert y = - x + x_j + y_j \}$ the "-" diagonal going through the point $c_j$. 
\end{itemize}
Let us call $A$ the intersection point of $d^+_i$ and $d^-_j$ and $B$ the intersection point of $d^-_i$ and $d^+_j$. We call \textit{parallelogram determined by $c_i$ and $c_j$} the parallelogram whose vertices are $c_i, A, c_j$ and  $B$, and we denote by $\drectCiCj$ the interior of the parallelogram\footnote{Note that as we consider non degenerated profile following Lemma~\ref{lemma:type1}, no point (besides $c_i$ and $c_j$) lies on one of the 4 diagonals - and in particular on the boarder of the parallelogram.}. \end{df}
\begin{figure}[tb]
    \centering
    \begin{tikzpicture}
   \begin{axis}
   [axis x line=bottom,axis y line = left, 
   grid = major,
   axis equal image,
   ytick = {1,2,3,4,5,6,7,8,9,10,11},
   xtick = {1,2,3,4,5,6,7,8,9,10,11},
   xmin=0,
   xmax=11,
   ymin=0,
   ymax=11,
   nodes near coords,
   point meta=explicit symbolic]
   \addplot+[only marks] coordinates{(3,3)[$c_i$] (7,5)[$c_j$] (4,2)[$B$] (6,6)[$A$]};
    \addplot+[mark = none, red, thick] coordinates{(0,0) (9,9)} node[xshift = -0.5cm] {$d^+_i: y = x - x_i + y_i$};
    \addplot+[mark = none, blue, thick] coordinates{(0,6) (6,0)} node[above,pos=1] {$d^-_i: y = - x + x_i + y_i$};
    \addplot+[mark = none, thick] coordinates{(2,0) (9,7)} node[xshift = -0.1cm, yshift = -0.2cm] {$d^+_j: y = x - x_j + y_j$};
     \addplot+[mark = none, color = brown, thick] coordinates{(9,3) (3,9)} node[above,pos=1] {$d^-_j: y = - x + x_j + y_j$};
    \end{axis}
    \end{tikzpicture}
    \caption{The parallelogram determined by $c_i$ and $c_j$. 
    }
    \label{fig:zone_interdit_l1}
\end{figure}
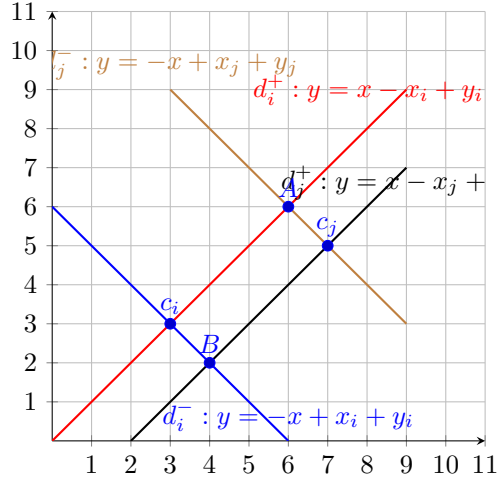

We are now able to reformulate Proposition~\ref{prop:triangle_intersection} as Proposition~\ref{prop:diag_rect} below.

\begin{prop}\label{prop:diag_rect}
Given three points $c_1 = (x_1, y_1)$, $c_2 = (x_2, y_2)$ and $c_3 = (x_3, y_3)$:
\begin{itemize}
    \item If $c_1$, $c_2$ or $c_3$ is inside the parallelogram determined by the two other points, then $H(c_1,c_2)$, $H(c_1, c_3)$ and $H(c_2, c_3)$ do not (pairwise) intersect. 
    \item Otherwise, the intersection of the three \hyps is a unique point. 
\end{itemize} 
\end{prop}
The proof of Proposition~\ref{prop:diag_rect} can be found in Appendix~\ref{app:prop5}. This proposition has a direct consequence on the preferences within an $\ell_1$-Euclidean profile. It is given in the following corollary, which will be used in Section~\ref{sec:n>=5} to show an upper bound on the number of candidates ranked last by at least one voter in a $\ell_1$-Euclidean profile. 
\begin{cor}\label{cor:jamais_dernier}
Let $\mathcal{P}\!=\!(V, C)$ be an \lunprofile profile, and consider three candidates $c_1\!=\!(x_1, y_1)$, $c_2\!=\!(x_2, y_2)$ and $c_3\!=\!(x_3, y_3)$ in a given $\ell_1$-Euclidean representation of $\mathcal{P}$. If $c_2\!\in\!\drect{c_1}{c_3}$, then there is no voter $v\!\in\!V$ for whom both $c_1\!>_v\!c_2$ and $c_3\!>_v\!c_2$. In other words, $c_2$ is never ranked last among $c_1, c_2, c_3$.
\end{cor}

\begin{proof}
Assume that $c_2$ is inside the parallelogram determined by $c_1$ and $c_3$. Proposition \ref{prop:diag_rect} implies that $H(c_1,c_2)$, $H(c_1,c_3)$ and $H(c_2,c_3)$ \textcolor{black}{do not (pairwise) intersect. Hence, they} are all horizontal, or all vertical  \textcolor{black}{(as a vertical \hyp always intersects a horizontal one)}. Without loss of generality, assume that all three \hyps are vertical, and that $x_1 < x_2 < x_3$. As each point $(x,y)$ of $H(c_i, c_j)$ satisfies $x_i < x < x_j$, we have $H(c_1, c_2)$ on the left of $H(c_2, c_3)$. \\ \\
\noindent
We now show by contradiction that $H(c_1, c_3)$ lies between these two hypersurfaces. Assume the left-to-right order of \hyps is $H(c_1, c_3), H(c_1, c_2), H(c_2, c_3)$. As $c_1$ lies in the leftmost area, it is necessarily the top-ranked candidate there.  
The second-ranked candidate in this area must be $c_3$, \textcolor{black}{the leftmost \hyp being $H(c_1, c_3)$}. Thus, the ranking of the leftmost area is $c_1\!>\!c_3\!>\!c_2$. \textcolor{black}{By moving from the leftmost to the rightmost area, we obtain consecutively (by crossing the \hyps one by one) the four following rankings:  $(c_1,c_3,c_2)$ (the leftmost one), $(c_3,c_1,c_2)$ (after crossing $H(c_1, c_3)$), $(c_3, c_2, c_1)$ (after crossing $H(c_1, c_2)$) and finally $(c_2, c_3, c_1)$ (the rightmost one, after crossing $H(c_2, c_3)$)}. We get a contradiction: as $c_3$ lies in the rightmost area (because we have $x_1 < x_2 < x_3$), it must be a top-ranked candidate there. \\ \\
\noindent 
The case where $H(c_1, c_3)$ is the rightmost \hyp can be treated similarly. Hence, the only possible order of \hyps is $H(c_1, c_2), H(c_1, c_3), H(c_2, c_3)$, and we see, with similar arguments as previously, that $c_2$ is never ranked last. 
\end{proof}

Note that Proposition \ref{prop:nb_intersections} only gives the possible number of intersection points between two \hyps\hspace{-0.095cm}, however, it does not specify the conditions in which each of the cases appears. The following result (see Appendix~\ref{app:prop6} for the proof) gives some more precise statement, which will be needed in the next sections in order to compute, based on geometrical arguments, the size of $\ell_1$-Euclidean profiles. 
\begin{prop}\label{prop:nb_intersections_l1_c4}
Given four points $c_1, c_2, c_3$ and $c_4$, there is at most one pair of \hyps $H(c_i, c_j), H(c_k, c_l)$ (with $\{ i,j,k,l\} = \{ 1,2,3,4 \}$) intersecting in two distinct points. 
\end{prop}

\section{Euclidean profiles on 4 candidates in the plane}\label{sec:n=4}

As we have seen, all the profiles with 3 candidates are $\ell_2$-Euclidean and $\ell_1$-Euclidean. We focus here on the case with 4 candidates. In Section~\ref{subsec:sizeprofile}, we study the maximum size of $\ell$-Euclidean profiles (for $\ell=\ell_1$ and $\ell=\ell_2$). In Section~\ref{subsec:carac}, we provide a concise characterization of $\ell_2$-Euclidean profiles. 

\subsection{Maximum size of a Euclidean profile on 4 candidates}\label{subsec:sizeprofile}

\cite{Bennett1960} gave a recursive formula to compute the maximum cardinality of $\ell_2$-Euclidean profiles in $\mathbb{R}^d$. For $d\!=\!2$ and 4 candidates, their formula gives the following result:  
\begin{prop}[\citealp{Bennett1960}]\label{prop:formule_profile_max_l2}
The max cardinality of a $\ell_2$-Euclidean profile on 4 candidates is 18.
\end{prop}

We examine this question for the norm $\ell_1$, and show that the maximum cardinality is 19 (Theorem~\ref{th:19}). The core of the proof is to show that it is at most 19 (Lemma~\ref{lemma:atmost19}): this is done by counting the (maximal) number of areas delimited by hypersurfaces. For this, we use several results of Section~\ref{sec:prop}, as well as Euler's formula for planar graphs.
An explicit construction of a $\ell_1$-Euclidean profile with 19 preferences is then given in Lemma~\ref{lemma:19}, which shows that the upper bound of Lemma~\ref{lemma:atmost19} is tight.

\begin{lem}\label{lemma:atmost19}
Any \lunprofile profile on 4 candidates has at most 19 (pairwise distinct) preferences. 
\end{lem}
\begin{proof}
To prove this proposition, given an $\ell_1$-Euclidean representation of a profile $\mathcal{P}$, we define a graph whose vertices are all \hyp intersections, and \textcolor{black}{where there is an edge between two intersections (denoted by $I_1$ and $I_2$) if and only if both $I_1$ and $I_2$ lie on the same hypersurface, and there is any other intersection on the segment of extremities $I_1$ and $I_2$} (see Figure~\ref{fig:euler} for an illustration).

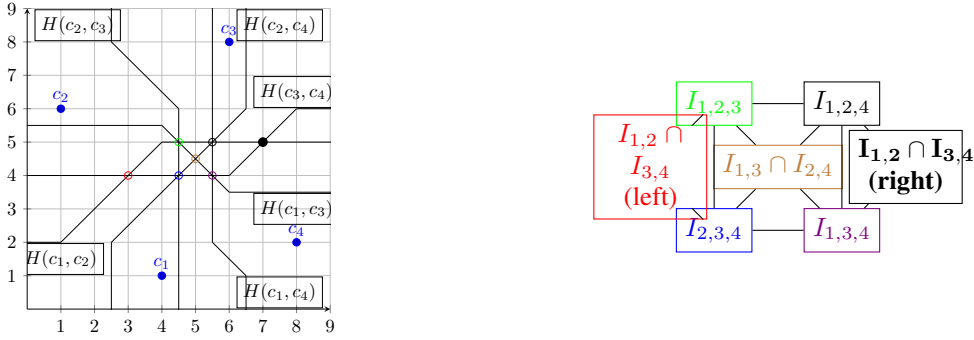
\begin{figure}[tb]
    \begin{tabular}{ccc}
    \begin{minipage}{0.45\textwidth}
   \begin{tikzpicture}[scale = 0.7]
   \begin{axis}
   [axis x line=bottom,axis y line = left, 
   grid = major,
   axis equal image,
   ytick = {1,2,3,4,5,6,7,8,9,10,11},
   xtick = {1,2,3,4,5,6,7,8,9,10,11},
   xmin=0,
   xmax=9,
   ymin=0,
   ymax=9,
   nodes near coords,
   point meta=explicit symbolic]
   \addplot+[only marks] coordinates{(4,1)[$c_1$] (1,6)[$c_2$] (6,8)[$c_3$] (8,2)[$c_4$]};
   \addplot+[mark = none, color = black] coordinates{(4,5.5) (6,3.5)};
   \addplot+[mark = none, color = black] coordinates{(-2,5.5) (4,5.5)};
   \addplot+[mark = none, color = black] coordinates{(6,3.5) (12,3.5)};
   \node[draw] at (axis cs:8,3) {$H(c_1, c_3)$};

     
     \addplot[mark = o, red]coordinates{(3,4)};
     \addplot[mark = o, blue]coordinates{(4.5,4)};
     \addplot[mark = o, green]coordinates{(4.5,5)};
     \addplot[mark = o, brown]coordinates{(5,4.5)};
     \addplot[mark = o, violet]coordinates{(5.5,4)};
      \addplot[mark = o, black]coordinates{(5.5,5)};
       \addplot[mark = *, black, very thick]coordinates{(7,5)};
    
   \draw[color = black] (axis cs:4,5) -- (axis cs:1,2);
   \draw[color = black] (axis cs:-2,2) -- (axis cs:1,2);
   \draw[color = black] (axis cs:4,5)--(axis cs:12,5);
   \node[draw] at (axis cs:1,1.5) {$H(c_1, c_2)$};
    
    \draw[color = black] (axis cs:6.5,1) -- (axis cs:5.5,2);
    \draw[color = black] (axis cs:6.5,-2) -- (axis cs:6.5,1);
    \draw[color = black] (axis cs:5.5,2)--(axis cs:5.5,12);
   
    \node[draw] at (axis cs:7.5,0.5) {$H(c_1, c_4)$};
    
    \draw[color = black] (axis cs:4.5,6) -- (axis cs:2.5,8);
    \draw[color = black] (axis cs:4.5,-2) -- (axis cs:4.5,6);
    \draw[color = black] (axis cs:2.5,8)--(axis cs:2.5,12);
    \node[draw] at (axis cs:1.5,8.5) {$H(c_2, c_3)$};
    
    \draw[color = black] (axis cs:6.5,6) -- (axis cs:2.5,2);
    \draw[color = black] (axis cs:2.5,-2) -- (axis cs:2.5,2);
    \draw[color = black] (axis cs:6.5,6)--(axis cs:6.5,12);
    \node[draw] at (axis cs:7.5,8.5) {$H(c_2, c_4)$};
    
    \draw[color = black] (axis cs:8,6) -- (axis cs:6,4);
    \draw[color = black] (axis cs:-2,4) -- (axis cs:6,4);
    \draw[color = black] (axis cs:8,6)--(axis cs:12,6);
    \node[draw] at (axis cs:8,6.5) {$H(c_3, c_4)$};
   
    \end{axis}
    \end{tikzpicture}
    \end{minipage}
    &
    ~
    &
    \begin{minipage}{0.41\textwidth}
    \begin{tikzpicture}[scale = 0.42]
        \node[draw,color = green] (I123) at (3,7) { $I_{1,2,3}$};
        \node[draw] (I124) at (7,7) { $I_{1,2,4}$};
        \node[draw, color = blue] (I234) at (3,3) { $I_{2,3,4}$};
        \node[draw, color = violet] (I134) at (7,3) { $I_{1,3,4}$};
        \node[draw, color = brown] (I1324) at (5,5) { $I_{1,3} \cap I_{2,4}$};
        \node[draw,text centered,text width=1.25cm] (I1234) at (9,5) {\textbf{$\mathbf{I_{1,2} \cap I_{3,4}}$ (right)}};
        \node[draw,text centered,text width=1.25cm, color = red] (I1234left) at (1,5) { $I_{1,2} \cap I_{3,4}$ (left)};
         
         \draw[] (I123) to (I124);
         \draw[] (I123) to (I234);
         \draw[] (I234) to (I134);
         \draw[] (I134) to (I124);
         \draw[] (I123) to (I1324);
         \draw[] (I123) to (I1234left);
         \draw[] (I1324) to (I134);
         \draw[] (I234) to (I1324);
         \draw[] (I234) to (I1234left);
         \draw[] (I1324) to (I124);
         \draw[] (I124) to (I1234);
         \draw[] (I1234) to (I134);
    \end{tikzpicture}
    \end{minipage}
    \end{tabular}
    \caption{A $\ell_1$-Euclidean representation of a profile $\mathcal{P}$, and its corresponding graph. The intersection of $H(c_i,c_j)$, $H(c_i,c_k)$, $H(c_j,c_k)$ (resp. $H(c_i,c_j)$ and $H(c_k,c_l)$) yields a vertex $I_{i,j,k}$ (resp. $I_{i,j}\cap I_{k,l}$).}
    \label{fig:euler}
\end{figure}

The corresponding graph is by construction planar. We note that each inner face of the graph corresponds to a bounded area in the representation of the profile, while unbounded areas in the representation of the preference profile are all merged into the outer face of the planar graph.

We can then use Euler's formula in the corresponding graph. It states that the number of faces of a planar graph is $n_f = n_e - n_v + 2 $, where $n_e$ is the number of edges and $n_v$ the number of vertices. 

Let us denote by $n_z$ the number of areas in the $\ell_1$-Euclidean representation of the profile. Note that each area corresponds to a single preference, so $n\leq n_z$. For 4 candidates, there are 6 hypersurfaces, leading to \textcolor{black}{at most} 12 unbounded areas. As mentioned above, these 12 unbounded areas are merged into the outer face of the planar graph. As the bounded areas yield $n_f\!-\!1$ inner faces, we have $n_z \textcolor{black}{\leq}n_f-1+12=n_f+11$, and therefore (by Euler's formula):
\begin{equation}\label{eqplan1} 
n_z \textcolor{black}{\leq} n_e - n_v + 13
\end{equation}
\textcolor{black}{If $k$ different \hyps intersect in a common point, we call this point a $k$-intersection.} We can assume there are only 2-intersections and 3-intersections: \textcolor{black}{let $f$ be a representation of a given $\ell_1$-Euclidean profile containing a 4-intersection $I$. As 3 points give only 3 different hypersurfaces, the 4 \hyps intersecting in $I$ involve the four points $f(c_1), f(c_2), f(c_3), f(c_4)$ corresponding to the positions of the four candidates $c_1, c_2, c_3, c_4$. By definition, $I$ is equidistant from all candidates - more formally, we have $\| f(c) - I\|_{\ell_1} = \delta > 0$ for each $c \in \{ c_1, c_2, c_3, c_4\}$. As in the Lemma~\ref{lemma:type1}, we define: 
\begin{equation*}
\varepsilon_d = \min_{v \in V} \min_{c_i, c_j \in C}\left|  \ \| f(v) - f(c_i) \|_{\ell_1} - \| f(v) - f(c_j) \|_{\ell_1} \  \right|.
\end{equation*} As we consider only strict preferences, $\varepsilon_d >0$. We can then add $\varepsilon=\frac{\varepsilon_d}{2}$ to $x_1$.\footnote{More precisely, $\varepsilon$ should be smaller than the minimum of $\varepsilon_d$ and $\min\{\varepsilon_x, \varepsilon_y,\varepsilon_{xy}\}$ as defined in Lemma~\ref{lemma:type1}, to ensure that we do not create one of the degenerations excluded by this Lemma.} Doing that, $I$ will be no more equidistant from all four points and therefore, there will no more be a 4-intersection in such a modified representation. By iterating the processus, all $k$-intersections can be excluded for any $k \geq 4$. }  \\ \\
As there are 4 candidates, there are at most four 3-intersections: 
\begin{itemize}
    \item $I_{123} = H(c_1, c_2) \cap H(c_1, c_3) \cap H(c_2, c_3)$,
    \item $I_{124} = H(c_1, c_2) \cap H(c_1, c_4) \cap H(c_2, c_4)$,
    \item $I_{134} = H(c_1, c_3) \cap H(c_1, c_4) \cap H(c_3, c_4)$,
    \item $I_{234} = H(c_2, c_3) \cap H(c_2, c_4) \cap H(c_3, c_4)$.
\end{itemize}
By Lemma \ref{lemma:inter}, we have covered all intersections of type 
$H(c_i, c_j) \cap H(c_i, c_k)$. That means, all 2-intersections will be of type $H(c_i, c_j) \cap H(c_k, c_l)$ with $i,j,k,l$ pairwise distinct. There are 3 pairs of \hyps of this type: 
\begin{itemize}
    \item $H(c_1, c_2) \cap H(c_3, c_4)$,
    \item $H(c_1, c_3) \cap H(c_2, c_4)$,
    \item $H(c_1, c_4) \cap H(c_2, c_3)$.
\end{itemize}
\textcolor{black}{Each of these three pairs can give us one 2-intersection. In addition, Proposition~\ref{prop:nb_intersections_l1_c4} implies that at most one of these pairs of \hyps can intersect twice. To sum up, we have at most four 2-intersections.} \textcolor{black}{Therefore}  $n_v\!\leq\!8$ (at most four 3-intersections and four 2-intersections).

\textcolor{black}{If $n_v\!=\!8$, there are four 2-intersections and four 3-intersections. Each 2-intersection generates four outgoing half-lines, and each 3-intersection generates six outgoing half-lines. We then get $4 \cdot 4 + 4 \cdot 6 = 40$ outgoing half-lines. However, 12 of them are delimiting outer non-bounded areas, so they are not responsible for any graph edge. Therefore, $40-12 = 28$ half-lines are left for forming edges. We observe that each of these half-lines is used in the creation of exactly one edge, and that each edge is a segment corresponding to the 
common part of exactly two half-lines (as each edge has two extremities which are two different intersections). Thus, we have $n_e = 28/2 = 14$. }
Finally, using Equation~\ref{eqplan1}:
$$ n_z \textcolor{black}{\leq} 14 - 8 + 13 = 19.$$

It is easy to check that if $n_v < 8$, then $n_z < 19$: \textcolor{black}{in fact, each 2-intersection (resp. 3-intersection) generate four (resp. six) outgoing half-lines. In both cases at most a half of them are delimiting outer non-bounded areas - which means that at least half of them has another 2-intersection or 3-intersection lying on it. Therefore, each vertex allows to create at least two edges, so in the Euler formula the benefit of
deleting a vertex is outweighted by the drawback of deleting two edges}. Thus, in any case, $n_z \leq 19$. The size $n$ of the profile therefore satisfies $n\leq n_z\leq 19$. 
\end{proof}

Let us now consider the following profile $\mathcal{P}^*_0$ with 19 voters and 4 candidates (\textcolor{black}{for more conciseness and readability, preferences are in columns, so for instance the first preference is $(c_1,c_2,c_3,c_4)$}).

$$
\mathcal{P}^*_0 = 
\begin{pmatrix}
 c_1 & c_1 & c_1 & c_1 & c_1 & c_1 & c_2 & c_2 & c_2 & c_2 & c_3 & c_3 & c_3 & c_4 & c_4 & c_4 & c_4 & c_4 & c_4\\
 c_2 & c_2 & c_3 & c_3 & c_4 & c_4 & c_1 & c_1 & c_4 & c_4  & c_1 & c_4 & c_4 & c_1 & c_1 & c_2 & c_2 & c_3 & c_3 \\ 
 c_3 & c_4 & c_2 & c_4 & c_2 & c_3 & c_3 & c_4 & c_1 & c_3  & c_4 & c_1 & c_2 & c_2 & c_3 & c_1 & c_3 & c_1 & c_2 \\
 c_4 & c_3 & c_4 & c_2 & c_3 & c_2 & c_4 & c_3 & c_3 & c_1  & c_2 & c_2 & c_1 & c_3 & c_2 & c_3 & c_1 & c_2 & c_1\\
\end{pmatrix} $$

\medskip

\begin{lem}\label{lemma:19}
$\mathcal{P}^*_0$ is \lunprofile\hspace{-0.09cm}. 
\end{lem}
\begin{proof}
Figure~\ref{fig:max_profile_l1_c4} provides a $\ell_1$-Euclidean representation of
$\mathcal{P}^*_0$. Preference $p_1$ corresponds to $(c_1,c_2,c_3,c_4)$ (the first column in $\mathcal{P}^*_0$), preference $p_7$ to $(c_2,c_1,c_3,c_4)$ (the 7th column in $\mathcal{P}^*_0$) as we cross $H(c_1,c_2)$ to go from $p_1$ to $p_7$, etc. \textcolor{black}{The representation function  $f: C \rightarrow \mathbb{R}^2$ leading to  Figure~\ref{fig:max_profile_l1_c4} corresponds to the following positions: $f(c_1) = (0,8), f(c_2) = (10,10), f(c_3) = (4,1)$ and $f(c_4) = (8,3)$. These positions are sufficient to plot the \hyps and to convince ourselves that there are 19 non-empty preference areas. For example, let us place a voter $v$ in the area corresponding to preference $p_2$, concretely on the coordinates $(5.5,8)$. We will check that her preference is indeed $p_2$. Denoting by $P_v$ the position of voter $v$ (i.e., the point $(5.5,8)$), we have
\begin{align*}
\| P_v - f(c_1)\|_{\ell_1} & = | 5.5 - 0| + | 8 - 8| = 5.5, \\
\| P_v - f(c_2)\|_{\ell_1} & = | 5.5 - 10| + | 8 - 10| = 6.5, \\
\| P_v - f(c_1)\|_{\ell_1} & = | 5.5 - 4| + | 8 - 1| = 8.5, \\
\| P_v - f(c_1)\|_{\ell_1} & = | 5.5 - 8| + | 8 - 3| = 7.5. 
\end{align*}
We see that, indeed, the preference of voter $v$ corresponds to $p_2\!=\!(c_1, c_2, c_4, c_3)$. }
\end{proof}

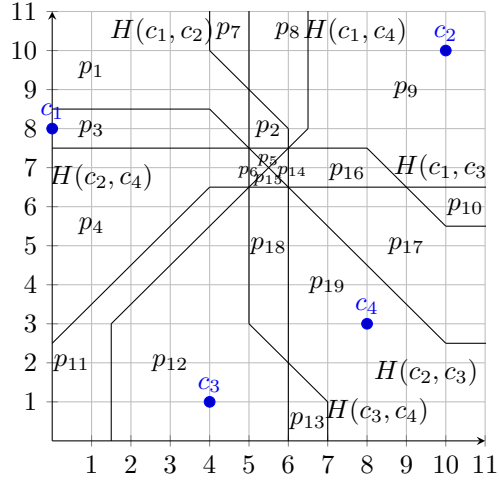
\begin{figure}[tb]
    \centering
    \begin{tikzpicture}[scale=1]
   \begin{axis}
   [axis x line=bottom,axis y line = left, 
   grid = major,
   axis equal image,
   ytick = {1,2,3,4,5,6,7,8,9,10,11},
   xtick = {1,2,3,4,5,6,7,8,9,10,11},
   xmin=0,
   xmax=11,
   ymin=0,
   ymax=11,
   nodes near coords,
   point meta=explicit symbolic]
   \addplot+[only marks] coordinates{(0,8)[$c_1$] (10,10)[$c_2$] (4,1)[$c_3$] (8,3)[$c_4$]};
   \addplot+[mark = none, color = black] coordinates{(6,8) (4,10)};
   \addplot+[mark = none, color = black] coordinates{(4,10) (4,12)};
   \addplot+[mark = none, color = black] coordinates{(6,8) (6,0)};
    \node[] at (axis cs:2.8,10.5) {$H(c_1, c_2)$};
    
   \draw[color = black] (axis cs:0,2.5) -- (axis cs:4,6.5);
   \draw[color = black] (axis cs:-2,2.5) -- (axis cs:0,2.5);
   \draw[color = black] (axis cs:4,6.5)--(axis cs:12,6.5);
    \node[] at (axis cs:10,7) {$H(c_1, c_3)$};
    
    \draw[color = black] (axis cs:6.5,8) -- (axis cs:1.5,3);
    \draw[color = black] (axis cs:1.5,3) -- (axis cs:1.5,-2);
    \draw[color = black] (axis cs:6.5,8)--(axis cs:6.5,12);
    \node[] at (axis cs:7.7,10.5) {$H(c_1, c_4)$};
    
    \draw[color = black] (axis cs:4,8.5) -- (axis cs:10,2.5);
    \draw[color = black] (axis cs:-2,8.5) -- (axis cs:4,8.5);
    \draw[color = black] (axis cs:10,2.5)--(axis cs:12,2.5);
    \node[] at (axis cs:9.5,1.75) {$H(c_2, c_3)$};
    
    \draw[color = black] (axis cs:8,7.5) -- (axis cs:10,5.5);
    \draw[color = black] (axis cs:-2,7.5) -- (axis cs:8,7.5);
    \draw[color = black] (axis cs:10,5.5)--(axis cs:12,5.5);
    \node[] at (axis cs:1.25,6.75) {$H(c_2, c_4)$};
    
    \draw[color = black] (axis cs:7,1) -- (axis cs:5,3);
    \draw[color = black] (axis cs:7,-2) -- (axis cs:7,1);
    \draw[color = black] (axis cs:5,3)--(axis cs:5,12);
    \node[] at (axis cs:8.25,0.75) {$H(c_3, c_4)$};
    
    \node[] at (axis cs:1,9.5) {$p_1$};
    \node[] at (axis cs:5.5,8) {$p_2$};
    \node[] at (axis cs:1,8) {$p_3$};
    \node[] at (axis cs:1,5.5) {$p_4$};
    
     \node[] at (axis cs:5.5,7.2) {\scriptsize $p_5$};
     \node[] at (axis cs:5,6.9) {\scriptsize $p_6$};
     
     \node[] at (axis cs:4.5,10.5) {$p_7$};
     \node[] at (axis cs:6,10.5) {$p_8$};
     \node[] at (axis cs:9,9) {$p_9$};
     \node[] at (axis cs:10.5,6) {$p_{10}$};
     
     \node[] at (axis cs:0.5,2) {$p_{11}$};
     \node[] at (axis cs:3,2) {$p_{12}$};
     \node[] at (axis cs:6.5,0.5) {$p_{13}$};
     \node[] at (axis cs:6.1, 6.9) {\scriptsize $p_{14}$};
     \node[] at (axis cs:5.5, 6.7) {\scriptsize $p_{15}$};
     
     \node[] at (axis cs: 7.5,6.9){$p_{16}$};
     \node[] at (axis cs: 9,5){$p_{17}$};
     
     \node[] at (axis cs: 5.5,5){$p_{18}$};
     \node[] at (axis cs: 7,4){$p_{19}$};
    \end{axis}
    \end{tikzpicture}
    \caption{An \lunprofile representation of a profile with 4 candidates and 19 (pairwise) distinct votes.}
    \label{fig:max_profile_l1_c4}
\end{figure}

As a direct consequence of Lemmata~\ref{lemma:atmost19} and~\ref{lemma:19}, we have the following result, which concludes the section.

\begin{thm}\label{th:19}
The maximum cardinality of an \lunprofile profile on 4 candidates is 19.
\end{thm}
We note that $\mathcal{P}^*_0$ is another example of a preference profile on 4 candidates which is $\ell_1$-Euclidean but not $\ell_2$-Euclidean (because there are more than 18 preferences).

\subsection{Characterization of $\ell_2$-Euclidean profiles}\label{subsec:carac}

A central question in structured preferences is to determine whether a given profile is structured or not. As we have seen before, with 4 candidates, any profile with more than 18 (resp. 19) preferences is not $\ell_2$-Euclidean (resp. not $\ell_1$-Euclidean). However, there are smaller profiles which are not $\ell_1$- or $\ell_2$-Euclidean (Example~\ref{ex:n=4ell2} gives such a profile for $\ell_2$). 

In the sequel, we give a concise description of all $\ell_2$-Euclidean profiles on 4 candidates, that moreover enables to  \textcolor{black}{easily} determine whether a given profile on 4 candidates is $\ell_2$-Euclidean or not. As noted in the introduction, this result has also been proved by \cite{kamiya2011ranking}. The interest of the proof presented below is that it is based on simpler and purely geometrical arguments. Furthermore, it might be possible to consider the same type of approach to prove a result of the same nature for $m\!\geq \!5$ candidates \emph{in the plane}, while the result by \citeauthor{kamiya2011ranking} only applies for $d\!=\!m\!-\!2$ (thus for 4 candidates in the plane). Also, a similar geometric approach might be useful for the $\ell_1$ norm, for which we conjecture that the profile $\mathcal{P}^*_0$ is the {\it unique} maximal $\ell_1$-Euclidean profile.

Back to $\ell_2$-Euclidean profiles on 4 candidates, we show that the number of maximal Euclidean profiles is very small. More precisely, we prove that there are only 3 maximal $\ell_2$-Euclidean  profiles $\mathcal{P}^*_1$, $\mathcal{P}^*_2$ and $\mathcal{P}^*_3$ (up to a permutation of the candidates), each of them of size 18. Thus, a profile is $\ell_2$-Euclidean if and only if it is a subprofile of $\mathcal{P}^*_1$, $\mathcal{P}^*_2$ or $\mathcal{P}^*_3 $ (up to a permutation of the candidates).

We say that two profiles are {\it isomorphic} if they contain the same set of preferences up to a permutation of the candidates.

Let us consider the three following profiles $\mathcal{P}^*_1$, $\mathcal{P}^*_2$ and $\mathcal{P}^*_3$:

$$
\mathcal{P}^*_1 = 
\begin{pmatrix}
 c_1 & c_1 & c_1 & c_2 & c_2 & c_2 & c_2 & c_2 & c_2 & c_3 & c_3 & c_3 & c_4 & c_4 & c_4 & c_4 & c_4 & c_4 \\
 c_2 & c_2 & c_4 & c_1 & c_1 & c_3 & c_3 & c_4 & c_4 & c_2 & c_2 & c_4 & c_1 & c_1 & c_2 & c_2 & c_3 & c_3 \\ 
 c_3 & c_4 & c_2 & c_3 & c_4 & c_1 & c_4 & c_1 & c_3 & c_1 & c_4 & c_2 & c_2 & c_3 & c_1 & c_3 & c_1 & c_2 \\
 c_4 & c_3 & c_3 & c_4 & c_3 & c_4 & c_1 & c_3 & c_1 & c_4 & c_1 & c_1 & c_3 & c_2 & c_3 & c_1 & c_2 & c_1 \\
\end{pmatrix}, $$
$$
\mathcal{P}^*_2 =
\begin{pmatrix}
 c_1 & c_1 & c_1 & c_1 & c_1 & c_1 & c_2 & c_2 & c_3 & c_3 & c_3 & c_3 & c_3 & c_3 & c_4 & c_4 & c_4 & c_4 \\
 c_2 & c_2 & c_3 & c_3 & c_4 & c_4 & c_1 & c_3 & c_1 & c_1 & c_2 & c_2 & c_4 & c_4 & c_1 & c_1 & c_3 & c_3 \\ 
 c_3 & c_4 & c_2 & c_4 & c_2 & c_3 & c_3 & c_1 & c_2 & c_4 & c_1 & c_4 & c_1 & c_2 & c_2 & c_3 & c_1 & c_2 \\
 c_4 & c_3 & c_4 & c_2 & c_3 & c_2 & c_4 & c_4 & c_4 & c_2 & c_4 & c_1 & c_2 & c_1 & c_3 & c_2 & c_2 & c_1 \\
\end{pmatrix}, $$

$$\mathcal{P}^*_3 =  
\begin{pmatrix}
c_1 & c_1 & c_1 & c_1 & c_2 & c_2 & c_2 & c_2 & c_2 & c_2 & c_3 & c_3 & c_3 & c_3 & c_4 & c_4 & c_4 & c_4 \\
c_2 & c_2 & c_3 & c_4 & c_1 & c_1 & c_3 & c_3 & c_4 & c_4 & c_1 & c_2 & c_2 & c_4 & c_1 & c_2 & c_2 & c_3 \\
c_3 & c_4 & c_2 & c_2 & c_3 & c_4 & c_1 & c_4 & c_1 & c_3 & c_2 & c_1 & c_4 & c_2 & c_2 & c_1 & c_3 & c_2 \\
c_4 & c_3 & c_4 & c_3 & c_4 & c_3 & c_4 & c_1 & c_3 & c_1 & c_4 & c_4 & c_1 & c_1 & c_3 & c_3 & c_1 & c_1 \\
\end{pmatrix}.
$$

\begin{thm}\label{th:3max}
A profile on 4 candidates is $\ell_2$-Euclidean if and only if it is isomorphic to a subprofile of $\mathcal{P}^*_1$, $\mathcal{P}^*_2$ or $\mathcal{P}^*_3$. 
\end{thm}

\begin{proof}
\os{For any  maximal profile there is a representation of it such that no pair of \hyps (which are simple lines of the plane in the present case) are parallel. In fact, there will be two parallel lines if there are (at least) three aligned candidates, or if two pairs of candidates are the extremities of two parallel segments. In each of these cases, we can always slightly move one of the candidates (using the same technique as in Lemma~\ref{lemma:type1}) so that the two concerned lines are no more parallel and such a modified mapping is still a representation of the given profile. \\ Assuming that,} we have:
\begin{itemize}
    \item One 2-intersection $H(c_i, c_j) \cap H(c_k, c_l)$ for each pair of \hyps with $i,j,k,l$ pairwise distinct. For 4 candidates, it yields three 2-intersections (because there are three such pairs). 
    \item One 3-intersection $I_{i,j,k} = H(c_i, c_j) \cap H(c_i, c_k) \cap H(c_j, c_k)$ for each triple of \hyps with $i,j,k$ pairwise distinct. For a profile on 4 candidates, it yields four 3-intersections (because there are four such triples). 
\end{itemize}

Let us study the relative positions of the 3-intersections in the plane. There are two possible scenarios (see Figure \ref{fig:4_intersections_case_of_fig}): 
\begin{enumerate}
    \item The 3-intersections are the vertices of a convex quadrilateral (left part of Figure~\ref{fig:4_intersections_case_of_fig}). No pair of opposite sides of this quadrilateral can be parallel, otherwise there would be two parallel hypersurfaces, and the profile would not be maximal. 
    \item Three of the 3-intersections are the vertices of a triangle, and the fourth one is inside this triangle (right part of Figure~\ref{fig:4_intersections_case_of_fig}). 
\end{enumerate}

\begin{figure}[tb]
    \centering
    \begin{tikzpicture}[scale = 0.6]
   \begin{axis}
   [
   xmin=0,
   xmax=10,
   ymin=0,
   ymax=10,
   nodes near coords,
   point meta=explicit symbolic]
   \addplot+[only marks] coordinates{(2,2)[] (4,8)[] (6,6)[] (7,3)[]};
    \end{axis}
    \end{tikzpicture}
    \begin{tikzpicture}[scale = 0.6]
   \begin{axis}
   [
   xmin=0,
   xmax=10,
   ymin=0,
   ymax=10,
   nodes near coords,
   point meta=explicit symbolic]
   \addplot+[only marks] coordinates{(2,2)[] (4,8)[] (4,4)[] (7,3)[]};
    \end{axis}
    \end{tikzpicture}
    \caption{The possible relative positions of the 3-intersections in the plane.}
    \label{fig:4_intersections_case_of_fig}
\end{figure}
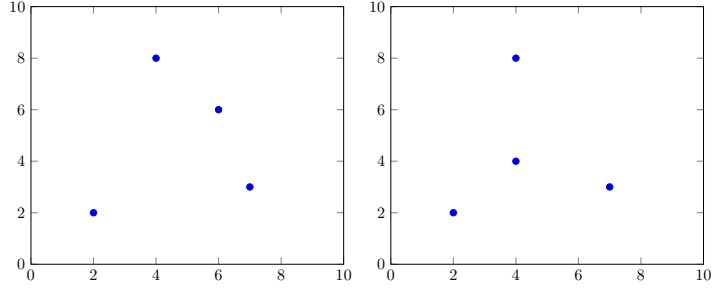
We will now take a closer look to each of these cases, and we will construct all maximal profiles corresponding to each of them. \\ \\
\noindent
\textbf{Case 1:} \\ 
\noindent
Assume that the 3-intersections are the vertices of a convex quadrilateral (see Figure \ref{fig:4_intersections_case_1}), as described above. There are $\binom{4}{2}\!=\!6$ hypersurfaces, each of them goes through exactly two 3-intersections (because, for an \hyp $H(c_i,c_j)$, there are two ways to choose $c_k$ with $k\!\not\in\!\{i,j\}$). \textcolor{black}{We recall that, without loss of generality, we can assume that there is no pair of parallel hypersurfaces.  There are then four \hyps that form the sides of a convex quadrilateral. Each of the two pairs of hypersurfaces corresponding to opposite sides of the quadrilateral results in an intersection outside the quadrilateral, which yields two distinct 2-intersections. The remaining two \hyps represent the diagonals of the quadrilateral, and will hence intersect inside it - it results in the third (and last) 2-intersection. Whatever the positions of the four candidates, if the 3-intersections form a convex quadrilateral, the partitioning of the plane will always look like in Figure~\ref{fig:4_intersections_case_1_with_c} (where $c_1 = (1,5), c_2 = (4,2), c_3 = (6,8)$ and $c_4 = (9,3)$): one 2-intersection lies inside the convex quadrilateral, and the two remaining 2-intersections (of \hyps forming opposite sides of the quadrilateral) outside of it.} 

\os{Note that in Figure~\ref{fig:4_intersections_case_1_with_c} some areas are small. For readability reasons, in what follows, we use instead Figure~\ref{fig:4_intersections_case_1_without_c} (with a similar arrangement of areas) where the areas are larger but without the explicit positions of candidates.}  

\noindent
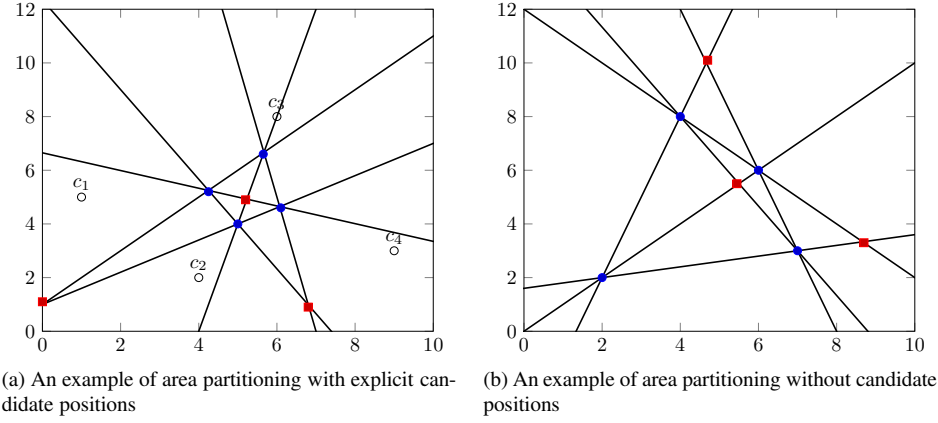
\begin{figure}[tb]
\subfloat[An example of area partitioning with 
explicit candidate positions]{
 \centering
    \begin{tikzpicture}[scale = 0.75]
   \begin{axis}
   [
   xmin=0,
   xmax=10,
   ymin=0,
   ymax=12,
   nodes near coords,
   point meta=explicit symbolic]
    \addplot+[only marks, color = blue] coordinates{(5,4)[] (6.1,4.6)[] (5.65,6.6)[] (4.25,5.2)[]};
   \addplot+[only marks, mark = square*] coordinates{(0,1.1)[] (6.8,0.9)[] (5.2,4.9)[]};
   \addplot+[only marks, mark=o, color=black] coordinates{(1,5)[$c_1$] (4,2)[$c_2$] (6,8)[$c_3$] (9,3)[$c_4$]};

   \draw[shorten >= -10cm, shorten <=-10cm, -, thick](axis cs:0,1)--(axis cs:10,11);
   \draw[shorten >= -10cm, shorten <=-10cm, -, thick](axis cs:0,12.345)--(axis cs:10,-4.355);
   \draw[shorten >= -10cm, shorten <=-10cm, -, thick](axis cs:0,-16)--(axis cs:10,24);
   \draw[shorten >= -10cm, shorten <=-10cm, -, thick](axis cs:0,6.65)--(axis cs:10,3.35);
   \draw[shorten >= -10cm, shorten <=-10cm, -, thick](axis cs:0,35)--(axis cs:10,-15);
   \draw[shorten >= -10cm, shorten <=-10cm, -, thick](axis cs:0,1)--(axis cs:10,7);
   
    \end{axis}
    \label{fig:4_intersections_case_1_with_c}
    \end{tikzpicture}
} 
\quad
\subfloat[An example of area partitioning without candidate positions]{
    \centering
    \begin{tikzpicture}[scale = 0.75]
   \begin{axis}
   [
   xmin=0,
   xmax=10,
   ymin=0,
   ymax=12,
   nodes near coords,
   point meta=explicit symbolic]
   \addplot+[only marks] coordinates{(2,2)[] (4,8)[] (6,6)[] (7,3)[]};
    \addplot+[only marks, mark = square*, color = red] coordinates{(8.7,3.3)[] (5.45,5.5)[] (4.7,10.1)[]};
   \draw[shorten >= -10cm, shorten <=-10cm, -, thick](axis cs:2,2)--(axis cs:4,8);
   \draw[shorten >= -10cm, shorten <=-10cm, -, thick](axis cs:2,2)--(axis cs:6,6);
   \draw[shorten >= -10cm, shorten <=-10cm, -, thick](axis cs:2,2)--(axis cs:7,3);
   \draw[shorten >= -10cm, shorten <=-10cm, -, thick](axis cs:4,8)--(axis cs:6,6);
   \draw[shorten >= -10cm, shorten <=-10cm, -, thick](axis cs:4,8)--(axis cs:7,3);
   \draw[shorten >= -10cm, shorten <=-10cm, -, thick](axis cs:6,6)--(axis cs:7,3);
   
    \end{axis}
    \label{fig:4_intersections_case_1_without_c}
    \end{tikzpicture}
}  
    \caption{Case 1: The plane is divided into 18 areas, with the 3-intersections forming a convex quadrilateral. The candidates are plotted with empty circles, 3-intersections with blue circles and 2-intersections with red squares.}
    \label{fig:4_intersections_case_1}
\end{figure}

To enumerate all possible maximal profiles corresponding to this configuration of the 3-intersections, the \hyps (and hence the intersections) need to be labeled so we can list the preferences corresponding to the different areas (see Figure~\ref{fig:4_intersections_case_1_labeled}). Without loss of generality, we label one of the 3-intersections as $I_{1,2,3}$, and one of the \hyps going through it as $H(c_1, c_2)$. The second 3-intersection involving $H(c_1, c_2)$ is then necessarily $I_{1,2,4}$. The two remaining \hyps going through $I_{1,2,3}$ are $H(c_1, c_3)$ and $H(c_2, c_3)$, that we can arbitrarily label (because it will turn out to be symmetrical). From these labels $I_{1,2,3}$, $I_{1,2,4}$, $H(c_1,c_2)$, $H(c_1,c_3)$ and $H(c_2,c_3)$, we can infer the labels of the two remaining 3-intersections, and so the labels of the remaining hypersurfaces. As mentioned earlier, both ways of labeling $H(c_1, c_3)$ and $H(c_2, c_3)$ are symmetric: it is sufficient to rename $c_1$ as $c_2$ and $c_2$ as $c_1$ to switch from one complete labeling to the other one (see Figure~\ref{fig:4_intersections_case_1_labeled}). Hence, the labels of $H(c_1, c_3)$ and $H(c_2, c_3)$ can be fixed without loss of generality, and there is only one possible complete labeling, up to a renaming of the candidates. 

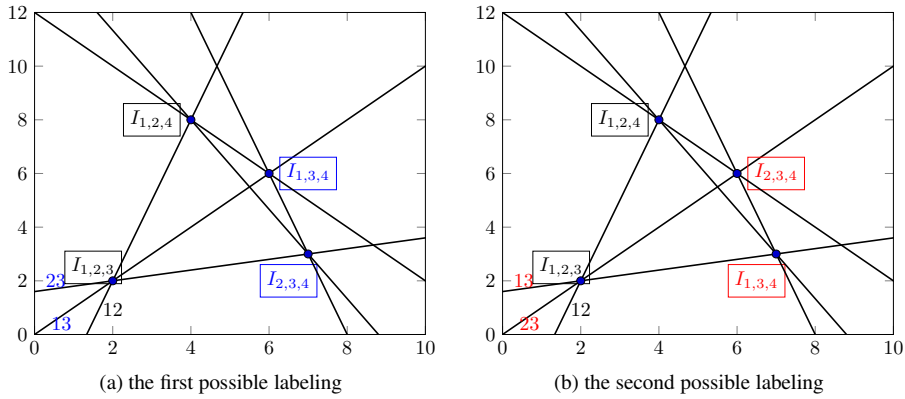
\begin{figure}[tb]
    \centering
    \subfloat[the first possible labeling]{
    \begin{tikzpicture}[scale = 0.75]
   \begin{axis}
   [
   xmin=0,
   xmax=10,
   ymin=0,
   ymax=12,
   nodes near coords,
   point meta=explicit symbolic]
   \addplot+[only marks, black] coordinates{(2,2) (4,8)[] (6,6)[] (7,3)[]};
   \draw[shorten >= -10cm, shorten <=-10cm, -, thick, ](axis cs:2,2)--(axis cs:4,8) node[at start, yshift = -0.5cm] {$12$};
   \draw[shorten >= -10cm, shorten <=-10cm, -, thick](axis cs:2,2)--(axis cs:6,6) node[at start, yshift = -0.75cm, xshift = -0.9cm, blue] {$13$};  
   \draw[shorten >= -10cm, shorten <=-10cm, -, thick](axis cs:2,2)--(axis cs:7,3) node[at start, xshift = -1cm, blue] {$23$}; 
   \node[draw, black] at (axis cs:1.5,2.5) {$I_{1,2,3}$};
   \node[draw, black] at (axis cs:3,8) {$I_{1,2,4}$};
   \node[draw, blue] at (axis cs:7,6) {$I_{1,3,4}$};
   \node[draw, blue] at (axis cs:6.5,2) {$I_{2,3,4}$};
   \draw[shorten >= -10cm, shorten <=-10cm, -, thick](axis cs:4,8)--(axis cs:6,6);
   \draw[shorten >= -10cm, shorten <=-10cm, -, thick](axis cs:4,8)--(axis cs:7,3);
   \draw[shorten >= -10cm, shorten <=-10cm, -, thick](axis cs:6,6)--(axis cs:7,3);
    \end{axis}
    \end{tikzpicture}
    }
    ~
    \subfloat[the second possible labeling]{
    \begin{tikzpicture}[scale = 0.75]
   \begin{axis}
   [
   xmin=0,
   xmax=10,
   ymin=0,
   ymax=12,
   nodes near coords,
   point meta=explicit symbolic]
   \addplot+[only marks, black] coordinates{(2,2) (4,8)[] (6,6)[] (7,3)[]};
   \draw[shorten >= -10cm, shorten <=-10cm, -, thick, ](axis cs:2,2)--(axis cs:4,8) node[at start, yshift = -0.5cm] {$12$};
   \draw[shorten >= -10cm, shorten <=-10cm, -, thick](axis cs:2,2)--(axis cs:6,6)  node[at start, yshift = -0.75cm, xshift = -0.9cm, red] {$23$};
   \draw[shorten >= -10cm, shorten <=-10cm, -, thick](axis cs:2,2)--(axis cs:7,3) node[at start, xshift = -1cm, red] {$13$};
   \node[draw, black] at (axis cs:1.5,2.5) {$I_{1,2,3}$};
   \node[draw, black] at (axis cs:3,8) {$I_{1,2,4}$};
   \node[draw, red] at (axis cs:7,6) {$I_{2,3,4}$};
   \node[draw, red] at (axis cs:6.5,2) {$I_{1,3,4}$};
   \draw[shorten >= -10cm, shorten <=-10cm, -, thick](axis cs:4,8)--(axis cs:6,6);
   \draw[shorten >= -10cm, shorten <=-10cm, -, thick](axis cs:4,8)--(axis cs:7,3);
   \draw[shorten >= -10cm, shorten <=-10cm, -, thick](axis cs:6,6)--(axis cs:7,3);
    \end{axis}
    \end{tikzpicture}
    }
    \caption{Labeled representation: $H(c_i, c_j)$ is noted as $ij$ due to lack of space.}
    \label{fig:4_intersections_case_1_labeled}
\end{figure}

Once the \hyps are labeled, we can list the preferences associated with the different areas. Let us focus on the \textcolor{black}{areas $A_1, A_2, A_3$ and $A_{13}$, as well as on the corresponding }preferences $p_1, p_2, p_3$ and $p_{13}$ in Figure~\ref{fig:labels_zones_c1}. To switch from $p_1$ to $p_2$, candidate $c_3$ is swapped with $c_4$ (because $H(c_3,c_4)$ is crossed \textcolor{black}{between areas $A_1$ and $A_2$}), while $c_4$ is swapped with $c_2$ to switch from $p_2$ to $p_3$ \textcolor{black}{(as $H(c_2, c_4)$ is crossed between areas $A_2$ and $A_3$)}, and finally $c_4$ is swapped with $c_1$ to obtain $p_{13}$. Necessarily, $c_4$ is ranked either in the first or in the last position in $p_1$ (resp. $p_4$), as it is successively swapped with all the remaining candidates. Hence, \textcolor{black}{the area $A_1$} corresponds to one of the following preferences: 
\begin{itemize}
    \item $p_1 = (c_1, c_2, c_3,c_4)$,
    \item $p'_1 = (c_4, c_3, c_2, c_1)$.
\end{itemize}
Once at least one preference is known, we can list all the preferences of the profile. Both profiles $\mathcal{P}\!=\!\{p_1,\ldots,p_{18}\}$ and $\mathcal{P}'\!=\!\{p_1',\ldots,p_{18}'\}$ are listed in Table~\ref{tab:profils_max_4c_1}. 
Profile $\mathcal{P}$ corresponds to $\mathcal{P}^*_1$ in the statement of the theorem, while $\mathcal{P}'$ corresponds to $\mathcal{P}^*_2$.

Note that, for each $1 \leq i \leq 18$, $p_i$ is the ``opposite'' of $p'_i$. Nevertheless, $\mathcal{P}'$ can not be obtained from $\mathcal{P}$ by renaming the candidates: indeed, while in $\mathcal{P}$ candidates $c_1$ and $c_3$ are each ranked first 3 times, and $c_2$ and $c_4$ ranked first 6 times, in $\mathcal{P}'$ in contrast, we have $c_1$ and $c_3$ that are ranked first 6 times, $c_2$  ranked first 2 times and $c_4$  ranked first 4 times. 
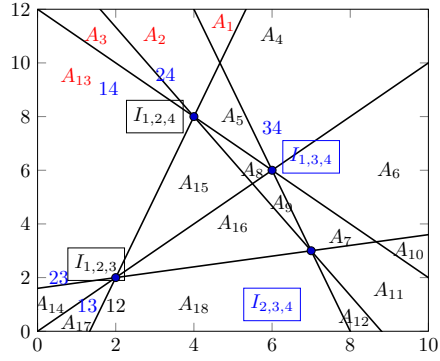
\begin{figure}[tb]
    \centering
    \begin{tikzpicture}[scale = 0.75]
   \begin{axis}
   [
   xmin=0,
   xmax=10,
   ymin=0,
   ymax=12,
   nodes near coords,
   point meta=explicit symbolic]
   \addplot+[only marks, black] coordinates{(2,2) (4,8)[] (6,6)[] (7,3)[]};
   \draw[shorten >= -10cm, shorten <=-10cm, -, thick, ](axis cs:2,2)--(axis cs:4,8) node[at start, yshift = -0.5cm] {$12$};
   \draw[shorten >= -10cm, shorten <=-10cm, -, thick](axis cs:2,2)--(axis cs:6,6) node[at start, yshift = -0.5cm, xshift = -0.5cm, blue]{$13$};
   \draw[shorten >= -10cm, shorten <=-10cm, -, thick](axis cs:2,2)--(axis cs:7,3) node[at start, xshift = -1cm, blue]{$23$};
   \node[draw, black] at (axis cs:1.5,2.5) {$I_{1,2,3}$};
   \node[draw, black] at (axis cs:3,8) {$I_{1,2,4}$};
   \node[draw, blue] at (axis cs:7,6.5) {$I_{1,3,4}$};
   \node[draw, blue] at (axis cs:6,1) {$I_{2,3,4}$};
   \node[red] at (axis cs:1,9.5) {$A_{13}$};
   \node[red] at (axis cs:1.5,11) {$A_3$};
   \node[red] at (axis cs:3,11) {$A_2$};
   \node[red] at (axis cs:4.75,11.5) {$A_1$};
   \node[] at (axis cs:6,11) {$A_4$};
   \node[] at (axis cs:5,8) {$A_5$};
   \node[] at (axis cs:4,5.5) {$A_{15}$};
   \node[] at (axis cs:5.5,6) {$A_8$};
   \node[] at (axis cs:9,6) {$A_6$};
   \node[] at (axis cs:5,4) {$A_{16}$};
   \node[] at (axis cs:6.25,4.75) {$A_{9}$};
   \node[] at (axis cs:7.75,3.5) {$A_{7}$};
   \node[] at (axis cs:0.3,1) {$A_{14}$};
   \node[] at (axis cs:1,0.3) {$A_{17}$};
   \node[] at (axis cs:4,1) {$A_{18}$};
   \node[] at (axis cs:8.1,0.5) {$A_{12}$};
   \node[] at (axis cs:9,1.5) {$A_{11}$};
   \node[] at (axis cs:9.5,3) {$A_{10}$};
   \draw[shorten >= -10cm, shorten <=-10cm, -, thick](axis cs:4,8)--(axis cs:6,6) node[at start, xshift = -1.5cm, yshift = 0.5cm, blue] {$14$};
   \draw[shorten >= -10cm, shorten <=-10cm, -, thick](axis cs:4,8)--(axis cs:7,3) node[at start, xshift =-0.5cm, yshift = 0.75cm, blue] {$24$};
   \draw[shorten >= -10cm, shorten <=-10cm, -, thick](axis cs:6,6)--(axis cs:7,3)
   node[at start, xshift =0cm, yshift = 0.75cm, blue] {$34$};
    \end{axis}
    \end{tikzpicture}
    
    \caption{\textcolor{black}{Listing the different areas $A_1, \ldots, A_{18}$ into which the plane is divided in case 1.}}
    \label{fig:labels_zones_c1}
\end{figure}

\begin{table}[H]
    \centering
    \begin{tabular}{ccc}\toprule
         $i$ & $p_i$ & $p'_i$ \\
         \midrule
         1 & $(c_1, c_2, c_3, c_4)$ & $(c_4, c_3, c_2, c_1)$\\
         2 & $(c_1, c_2, c_4, c_3)$ & $(c_3, c_4, c_2, c_1)$ \\
         3 & $(c_1, c_4, c_2, c_3)$  & $(c_3, c_2, c_4, c_1)$\\
         4 & $(c_2, c_1, c_3, c_4)$ & $(c_4, c_3, c_1, c_2)$\\
         5 & $(c_2, c_1, c_4, c_3)$ & $(c_3, c_4, c_1, c_2)$\\
         6 & $(c_2, c_3, c_1, c_4)$ & $(c_4, c_1, c_3, c_2)$\\
         7 & $(c_2, c_3, c_4, c_1)$ & $(c_1, c_4, c_3, c_2)$\\
         8 & $(c_2, c_4, c_1, c_3)$ & $(c_3, c_1, c_4, c_2)$\\
         9 & $(c_2, c_4, c_3, c_1)$ & $(c_1, c_3, c_4, c_2)$\\
         \bottomrule 
    \end{tabular}
    \begin{tabular}{ccc}\toprule
         $i$ & $p_i$ & $p'_i$ \\
         \midrule
         10 & $(c_3, c_2, c_1, c_4)$ & $(c_4, c_1, c_2, c_3)$\\
         11 &  $(c_3, c_2, c_4, c_1)$ & $(c_1, c_4, c_2, c_3)$\\
         12 & $(c_3, c_4, c_2, c_1)$  & $(c_1, c_2, c_4, c_3)$\\
         13 & $(c_4, c_1, c_2, c_3)$ & $(c_3, c_2, c_1, c_4)$\\
         14 & $(c_4, c_1, c_3, c_2)$ & $(c_2, c_3, c_1, c_4)$\\
         15 &  $(c_4, c_2, c_1, c_3)$ & $(c_3, c_1, c_2, c_4)$\\
         16 & $(c_4, c_2, c_3, c_1)$ & $(c_1, c_3, c_2, c_4)$ \\
         17 & $(c_4, c_3, c_1, c_2)$ & $(c_2, c_1, c_3, c_4)$\\
         18 & $(c_4, c_3, c_2, c_1)$ & $(c_1, c_2, c_3, c_4)$\\
         \bottomrule
    \end{tabular}
    \caption{The two maximal profiles $\mathcal{P}\!=\!\{p_1,\ldots,p_{18}\}$ and $\mathcal{P}'\!=\!\{p_1',\ldots,p_{18}'\}$ obtained in case 1.}
    \label{tab:profils_max_4c_1}
\end{table}

\noindent
\textbf{Case 2: } \\ 
\noindent
\textcolor{black}{To begin, let us denote by $T$ the triangle consisting of areas $A_5, A_6, A_7, A_8, A_9$ and $A_{10}$.} Using the same method as in the previous case, we note that there are two possible rankings for \textcolor{black}{ area $A_1$}  (see Figure \ref{fig:labels_zones_c2}, and the succession of areas \textcolor{black}{$A_1, A_2, A_4$ and $A_{15}$}): 
\begin{itemize}
    \item $p_1 = (c_1,c_2,c_3,c_4)$,
    \item $p'_1 = (c_4, c_3, c_2, c_1)$.
\end{itemize}
However, if $p'_1\!=\!(c_4, c_3, c_2, c_1)$, candidate $c_2$ is ranked in last position inside the triangle \textcolor{black}{$T$} : in fact, none of the \hyps crossing the triangle involves $c_2$. 
Let us now discuss the position of $c_2$ to show that $p_1'$ is not feasible:
\begin{itemize}
    \item \textcolor{black}{Denoting by $D(c_i,c_j)$ the set of points that are closer to $c_i$ than to $c_j$, we have: 
    $$ D(c_2, c_1) = A_1 \cup A_{2} \cup A_{3} \cup A_{4} \cup A_{11} \cup A_{15}.$$ In fact, $c_1$ is preferred to $c_2$ in triangle \textcolor{black}{$T$}.  Therefore, $c_1$ must lie on the same side of $H(c_1, c_2)$ as this triangle, and $c_2$ must then lie on the opposite side of $H(c_1, c_2)$, i.e. on the same side as the area $A_1$.}
    \item Analogously, candidate $c_2$ is necessarily \textcolor{black}{on the same side of} $H(c_2,c_3)$ \textcolor{black}{as the area $A_3$}: 
    $$ D(c_2, c_3) = A_3 \cup A_{11} \cup A_{12} \cup A_{13} \cup A_{14} \cup A_{18}.$$
    \item Finally, candidate $c_2$ is necessarily \textcolor{black}{on the same side} of  $H(c_2, c_4)$ \textcolor{black}{as the area $A_4$}: 
    $$ D(c_2, c_4) =  A_{4} \cup A_{14} \cup A_{15} \cup A_{16} \cup A_{17} \cup A_{18}.$$

\end{itemize}
As $c_2 \!\in\!D(c_2, c_i)$ for each $i\!\in\!\{1,3,4\}$, and as 
$D(c_2, c_1) \cap D(c_2, c_3) \cap D(c_2, c_4) = \emptyset$,
we cannot have $p'_1 = (4,3,2,1)$. 

The case $p_1 = (1,2,3,4)$ is feasible, leading to the profile described in Table~\ref{tab:profils_max_4c_2}, which corresponds to profile $\mathcal{P}^*_3$ in the statement of the theorem. 
\end{proof}

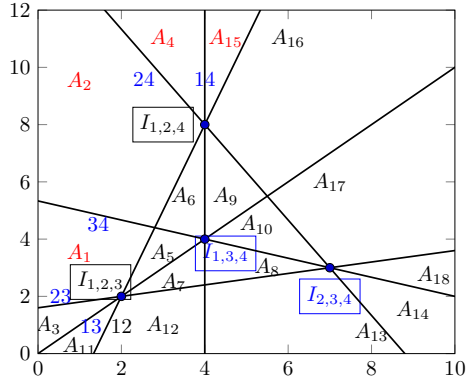
\begin{figure}[tb]
    \centering
    \begin{tikzpicture}[scale = 0.8]
   \begin{axis}
   [
   xmin=0,
   xmax=10,
   ymin=0,
   ymax=12,
   nodes near coords,
   point meta=explicit symbolic]
   \addplot+[only marks, black] coordinates{(2,2) (4,8)[] (4,4)[] (7,3)[]};
   \draw[shorten >= -10cm, shorten <=-10cm, -, thick, ](axis cs:2,2)--(axis cs:4,8) node[at start, yshift = -0.5cm] {$12$};
   \draw[shorten >= -10cm, shorten <=-10cm, -, thick](axis cs:2,2)--(axis cs:4,4) node[at start, yshift = -0.5cm, xshift = -0.5cm, blue]{$13$};
   \draw[shorten >= -10cm, shorten <=-10cm, -, thick](axis cs:2,2)--(axis cs:7,3) node[at start, xshift = -1cm, blue]{$23$};
   \node[draw, black] at (axis cs:1.5,2.5) {$I_{1,2,3}$};
   \node[draw, black] at (axis cs:3,8) {$I_{1,2,4}$};
   \node[draw, blue] at (axis cs:4.5,3.5) {$I_{1,3,4}$};
   \node[draw, blue] at (axis cs:7,2) {$I_{2,3,4}$};
   \node[red] at (axis cs:1,9.5) {$A_2$};
   \node[red] at (axis cs:3,11) {$A_4$};
   \node[red] at (axis cs:4.5,11) {$A_{15}$};
   \node[] at (axis cs:6,11) {$A_{16}$};
   \node[] at (axis cs:3.5,5.5) {$A_6$};
   \node[] at (axis cs:4.5,5.5) {$A_9$};
   \node[] at (axis cs:7,6) {$A_{17}$};
   \node[red] at (axis cs:1,3.5) {$A_1$};
   \node[] at (axis cs:3,3.5) {$A_5$};
   \node[] at (axis cs:5.25,4.5) {$A_{10}$};
   \node[] at (axis cs:3.25,2.5) {$A_7$};
   \node[] at (axis cs:5.5,3) {$A_8$};
   \node[] at (axis cs:0.25,1) {$A_3$};
   \node[] at (axis cs:1,0.25) {$A_{11}$};
   \node[] at (axis cs:3,1) {$A_{12}$};
   \node[] at (axis cs:8,0.75) {$A_{13}$};
   \node[] at (axis cs:9,1.5) {$A_{14}$};
   \node[] at (axis cs:9.5,2.75) {$A_{18}$};
   \draw[shorten >= -10cm, shorten <=-10cm, -, thick](axis cs:4,8)--(axis cs:4,4) node[at start, xshift = -1cm, yshift = 0.75cm, blue] {$24$};
   \draw[shorten >= -10cm, shorten <=-10cm, -, thick](axis cs:4,8)--(axis cs:7,3) node[at start, xshift =0cm, yshift = 0.75cm, blue] {$14$};
   \draw[shorten >= -10cm, shorten <=-10cm, -, thick](axis cs:4,4)--(axis cs:7,3)
   node[at start, xshift =-1.75cm, yshift = 0.25cm, blue] {$34$};
    \end{axis}
    \end{tikzpicture}
    
    \caption{\textcolor{black}{Listing the different areas $A_1, \ldots, A_{18}$ into which the plane is divided in case 2.}}
    \label{fig:labels_zones_c2}
\end{figure}

\begin{table}[H]
    \centering
    \begin{tabular}{cc}
         \toprule
         $i$ & $p_i$  \\
         \midrule
         1 & $(c_1, c_2, c_3, c_4)$ \\
         2 & $(c_1, c_2, c_4, c_3)$ \\
         3 & $(c_1, c_3, c_2, c_4)$ \\
         4 & $(c_1, c_4, c_2, c_3)$ \\
         5 & $(c_2, c_1, c_3, c_4)$ \\
         6 & $(c_2, c_1, c_4, c_3)$ \\
         7 & $(c_2, c_3, c_1, c_4)$ \\
         8 & $(c_2, c_3, c_4, c_1)$ \\
         9 & $(c_2, c_4, c_1, c_3)$ \\
         \bottomrule 
    \end{tabular}
    \begin{tabular}{cc}
         \toprule
         $i$ & $p_i$  \\
        \midrule
         10 & $(c_2, c_4, c_3, c_1)$ \\
         11 &  $(c_3, c_1, c_2, c_4)$\\
         12 & $(c_3, c_2, c_1, c_4)$ \\
         13 & $(c_3, c_2, c_4, c_1)$ \\
         14 & $(c_3, c_4, c_2, c_1)$ \\
         15 &  $(c_4, c_1, c_2, c_3)$ \\
         16 & $(c_4, c_2, c_1, c_3)$ \\
         17 & $(c_4, c_2, c_3, c_1)$\\
         18 & $(c_4, c_3, c_2, c_1)$\\
         \bottomrule
    \end{tabular}
    \caption{The maximal profile obtained in case 2.}
    \label{tab:profils_max_4c_2}
\end{table}

\section{Euclidean profiles on $m\!\geq\!5$ candidates in the plane}\label{sec:n>=5}

Let us now focus on the general case, by \textcolor{black}{giving some results} on the relative expressive power of $\ell_2$-Euclidean and $\ell_1$-Euclidean preference profiles. 
\textcolor{black}{We first note that, as shown in Proposition~\ref{prop:2d},} at most 4 candidates are ranked in last position (by at least one voter), regardless of the number of candidates in the profile. This is in sharp contrast to the $\ell_2$-Euclidean case, in which profiles where each candidate is ranked last at least once can easily be built, as mentioned in the introduction. 
 
This property might indicate that being Euclidean is much more restrictive for $\ell_1$ than for $\ell_2$. We show however that if we are interested in the maximum size of a Euclidean profile, then there is no such difference. We show indeed that the maximum size of a $\ell_1$-Euclidean profile on $m$ candidates is $\Theta(m^4)$ (Theorem~\ref{th:thetam4}), which is the same asymptotic bound as the one found by \citet{Bennett1960} for $\ell_2$.

Actually, a precise formula can be easily derived from their result: this maximal size is precisely $ \frac{m(3m-10)(m-1)(m+1)}{24} + m(m-1) + 1$. While such a precise formula seems to be tricky to establish for $\ell_1$ and is left as an open question, we show that the asymptotical bound is the same:  
\begin{thm}\label{th:thetam4}
The maximum size of an \lunprofile profile in $\mathbb{R}^2$ over $m$ candidates is in $\Theta(m^4)$. 
\end{thm}
\begin{proof}[sketch of proof]

We first show that the size of such a profile is in $O(m^4)$. There are $\frac{m(m-1)}{2}$ hypersurfaces. \textcolor{black}{With a non-degenerated profile, each pair of \hyps intersects at most twice, hence, there are at most $2(\frac{m(m-1)}{2})^2$ points of intersections. As in the case of norm $\ell_2$, we have at most 3 \hyps intersecting in one point. If a point is at the intersection of 2 (resp 3) hypersurfaces, it is incident to (i.e., a vertex of) 4 areas (resp. 6 areas). Then, as each area has at least one intersection point in its border, the number of areas is upper bounded by $6$ times the number of intersection points, i.e., in $O(m^4)$.} \\ \\
\noindent
Let us now show a profile for which this bound is reached. The idea is quite straightforward: as there are only vertical or horizontal hypersurfaces, and as each vertical and each horizontal \hyp intersect, the positions of candidates $c_1$ to $c_m$ will be iteratively fixed in such a way that approximately half of \hyps are vertical and half are horizontal. The number of intersections will then be in $\Theta(m^4)$, and the construction ensures that the number of areas is in $\Theta(m^4)$. 

The explicit construction of this profile is deferred to Appendix~\ref{app:thetam4}. 
\end{proof}

\section{Future work}\label{sec:conclu}

Because of their novelty, multiple avenues of research regarding $\ell_1$-Euclidean preference profiles can be considered. For instance, we conjecture that there is a unique maximal $\ell_1$-Euclidean preference profile for four candidates (and numerical tests seem to confirm this), but it remains to be proved. A broader research question is to investigate the existence of a general formula giving the maximal size of a $\ell_1$-Euclidean preference profile (as there is for $\ell_2$). Regarding the computational aspects, \cite{peters2017recognising} proved that the problem of recognising $\ell_1$-Euclidean preference profiles in $\mathbb{R}^d$ is in NP, but a more specific complexity class remains to be determined, and efficient recognition procedures are still to be proposed.

Although $\ell_2$-Euclidean preferences have been more studied than $\ell_1$-Euclidean preferences, various works can also be considered following those presented here, among which the identification of the maximal $\ell_2$-Euclidean preference profiles in $\mathbb{R}^2$ for $m\!\ge\!5$, or a thorough study of $\ell_2$-Euclidean preference profiles in $\mathbb{R}^3$.


\bibliographystyle{spbasic}      
\bibliography{prefeuclid}


%
%

\appendix

\newpage

\section{Missing proofs of Section~\ref{sec:prop}}\label{app:sec2}

\subsection{Proof of Proposition~\ref{prop:nb_intersections}}\label{app:prop3}

\normalsize

{\bf Proposition~\ref{prop:nb_intersections}.} {\it 
The intersection of two distinct $\ell_1$-\hyps  is either empty or contains  a unique point, two distinct points or an infinite number of points. }

\begin{proof}
\textcolor{black}{We can assume, without loss of generality, that the \hyps are given by two distinct pairs of points. Let us denote by $c_i = (x_i,y_i), i \in \{ 1, 2, 3, 4\}$ these points and their coordinates. Still without loss of generality, let $H(c_1, c_2)$ be of type $V^-$. There are four basic cases to distinguish (see Figs~\ref{fig:hyps_overview2}--\ref{fig:hyps_overview4} for illustrations):
\begin{enumerate}
    \item $H(c_1, c_2)$ is of type $V^-$ and $H(c_3, c_4)$ of type $H^+$ (see Figure~\ref{fig:hyps_overview2}): \\
    In this case, the \hyps intersect in a unique point as the half-lines (resp. the middle segments) of $H(c_1, c_2)$ and $H(c_3, c_4)$ are of opposite orientations. 
    \item $H(c_1, c_2)$ is of type $V^-$ and $H(c_3, c_4)$ of type $H^-$ (see Figure~\ref{fig:hyps_overview3}): \\
    As in the previous case, there will be at least one intersection as a horizontal hypersurface and a vertical \hyp always intersect. Contrary to the above, the middle segments of both \hyps have the same orientation, so they can overlap: in such a case, the intersection contains this overlapping segment, thus an infinite number of points.  
    \item $H(c_1, c_2)$ is of type $V^-$ and $H(c_3, c_4)$ of type $V^+$ (see Figure~\ref{fig:hyps_overview5}): \\
    In this case, the \hyps may not intersect: let us assume that $\max\{ x_1, x_2\}\!<\!\min\{ x_3, x_4\}$. By definition, we have $x\!\in\![ \min\{ x_i, x_j\}, \max\{ x_i, x_j\}]$ for each $(x,y)\!\in\!H(c_i, c_j)$. The above inequality then implies that the intersection of $H(c_1, c_2)$ and $H(c_3, c_4)$ is empty (graphically, $H(c_1, c_2)$ will be ``on the left'' of $H(c_3, c_4)$ - see the first case of Figure~\ref{fig:hyps_overview5}). \\
    The \hyps may also intersect in a unique point: a middle segment of one of the \hyps can intersect one of the half-lines of the second hypersurface, or its middle segment, as the middle segments are not of the same type (see the second case of Figure~\ref{fig:hyps_overview5}). \\
    Finally, as the half-lines of both \hyps are of the same type, one of the half-lines of $H(c_1, c_2)$ may (partially) overlap one of the half-lines of $H(c_3,c_4)$ (see the third case of Figure~\ref{fig:hyps_overview5}). In this case, the intersection will contain an infinity of points.
     \item Both \hyps $H(c_1, c_2)$ and $H(c_3, c_4)$ are of type $V^-$ (see Figure~\ref{fig:hyps_overview4}): \\
     This is the most complex case. For the same reason as above, the \hyps may not intersect. They may also intersect in a unique point if the middle segment of one \hyp intersects one of the half-lines of the second one (see the first case of the Figure~\ref{fig:hyps_overview4}). As the types of half-lines and middle segments are both the same for $H(c_1, c_2)$ and $H(c_3,c_4)$, they can also intersect in two distinct points if the middle segment of $H(c_1, c_2)$ intersects one of the half-lines of $H(c_3, c_4)$ and the middle segment of $H(c_3, c_4)$ intersects one of the half-lines of $H(c_1, c_2)$ (see the second case of Figure~\ref{fig:hyps_overview4}). Finally, the intersection can contain an infinity of points: as the half-lines are of the same type, a half-line of $H(c_1, c_2)$  may (partially) overlap a half-line of $H(c_3,c_4)$. In addition, the middle segments being also of the same type, they can (partially) overlap. See cases 3 and 4 of Figure~\ref{fig:hyps_overview4} (the case when both the half-lines and the middle segments overlap is not presented in Figure~\ref{fig:hyps_overview4}, but it is obviously possible). 
\end{enumerate}}

\end{proof}
\begin{figure}[htb]
    \centering
    \begin{tikzpicture}[scale=0.515]
   \begin{axis}
   [axis x line=bottom,axis y line = left, 
   grid = major,
   axis equal image,
   ytick = {1,2,3,4,5,6},
   xtick = {1,2,3,4,5,6},
   xmin=0,
   xmax=6,
   ymin=0,
   ymax=7,
   nodes near coords,
   point meta=explicit symbolic]
   \addplot+[only marks] coordinates{(0,1)[\Large$c_1$] (6,3)[\Large$c_2$]};
   \addplot+[only marks] coordinates{(1,6)[\Large$c_3$] (2,2)[\Large$c_4$]};
  
    \addplot+[mark = none, blue, thick] coordinates{(4,0) (4,1) (2,3) (2,7)};
    \addplot+[mark = none, red, thick] coordinates{(0,3.5) (1,3.5) (2,4.5) (6,4.5)};
    \end{axis}
    \end{tikzpicture}
    \caption{Intersection of two hypersurfaces: $H(c_1, c_2)$ is of type $V^-$ and $H(c_3, c_4)$ of type $H^+$.}
    \label{fig:hyps_overview2}
\end{figure}
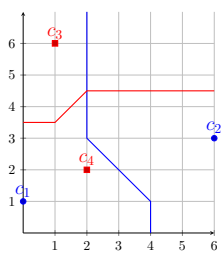
\begin{figure}[htb]
    \centering
    \begin{tikzpicture}[scale=0.515]
   \begin{axis}
   [axis x line=bottom,axis y line = left, 
   grid = major,
   axis equal image,
   ytick = {1,2,3,4,5,6},
   xtick = {1,2,3,4,5,6},
   xmin=0,
   xmax=6,
   ymin=0,
   ymax=7,
   nodes near coords,
   point meta=explicit symbolic]
   \addplot+[only marks] coordinates{(0,1)[\Large$c_1$] (6,3)[\Large$c_2$]};
   \addplot+[only marks] coordinates{(1,0.5)[\Large$c_3$] (3,5.5)[\Large$c_4$]};
  
    \addplot+[mark = none, blue, thick] coordinates{(4,0) (4,1) (2,3) (2,7)};
    \addplot+[mark = none, red, thick] coordinates{(0,4) (1,4) (3,2) (6,2)};
    \end{axis}
    \end{tikzpicture}
    \begin{tikzpicture}[scale=0.515]
   \begin{axis}
   [axis x line=bottom,axis y line = left, 
   grid = major,
   axis equal image,
   ytick = {1,2,3,4,5,6},
   xtick = {1,2,3,4,5,6},
   xmin=0,
   xmax=6,
   ymin=0,
   ymax=7,
   nodes near coords,
   point meta=explicit symbolic]
   \addplot+[only marks] coordinates{(0,1)[\Large$c_1$] (6,3)[\Large$c_2$]};
   \addplot+[only marks] coordinates{(1,2.5)[\Large$c_3$] (3,5.5)[\Large$c_4$]};
  
    \addplot+[mark = none, blue, thick] coordinates{(4,0) (4,1) (2,3) (2,7)};
    \addplot+[mark = none, red, thick] coordinates{(0,5) (1,5) (3,3) (6,3)};
    \end{axis}
    \end{tikzpicture}
    \caption{Intersection of two hypersurfaces: $H(c_1, c_2)$ is of type $V^-$ and $H(c_3, c_4)$ of type $H^-$.}
    \label{fig:hyps_overview3}
\end{figure}
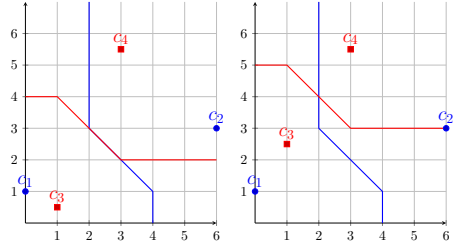
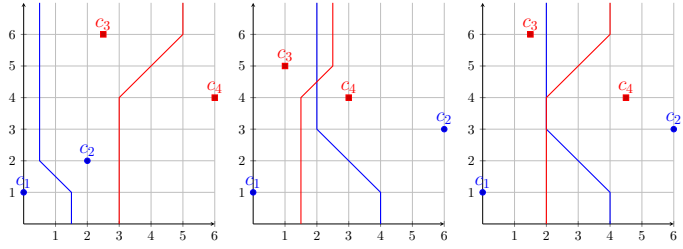
\begin{figure}[htb]
    \centering
    \begin{tikzpicture}[scale=0.515]
   \begin{axis}
   [axis x line=bottom,axis y line = left, 
   grid = major,
   axis equal image,
   ytick = {1,2,3,4,5,6},
   xtick = {1,2,3,4,5,6},
   xmin=0,
   xmax=6,
   ymin=0,
   ymax=7,
   nodes near coords,
   point meta=explicit symbolic]
   \addplot+[only marks] coordinates{(0,1)[\Large$c_1$] (2,2)[\Large$c_2$]};
   \addplot+[only marks] coordinates{(2.5,6)[\Large$c_3$] (6,4)[\Large$c_4$]};
  
    \addplot+[mark = none, blue, thick] coordinates{(1.5,0) (1.5,1) (0.5,2) (0.5,7)};
    \addplot+[mark = none, red, thick] coordinates{(3,0) (3,4) (5,6) (5,7)};
    \end{axis}
    \end{tikzpicture}
    \begin{tikzpicture}[scale=0.515]
   \begin{axis}
   [axis x line=bottom,axis y line = left, 
   grid = major,
   axis equal image,
   ytick = {1,2,3,4,5,6},
   xtick = {1,2,3,4,5,6},
   xmin=0,
   xmax=6,
   ymin=0,
   ymax=7,
   nodes near coords,
   point meta=explicit symbolic]
   \addplot+[only marks] coordinates{(0,1)[\Large$c_1$] (6,3)[\Large$c_2$]};
   \addplot+[only marks] coordinates{(1,5)[\Large$c_3$] (3,4)[\Large$c_4$]};
  
    \addplot+[mark = none, blue, thick] coordinates{(4,0) (4,1) (2,3) (2,7)};
    \addplot+[mark = none, red, thick] coordinates{(1.5,0) (1.5,4) (2.5,5) (2.5,7)};
    \end{axis}
    \end{tikzpicture}
    \begin{tikzpicture}[scale=0.515]
   \begin{axis}
   [axis x line=bottom,axis y line = left, 
   grid = major,
   axis equal image,
   ytick = {1,2,3,4,5,6},
   xtick = {1,2,3,4,5,6},
   xmin=0,
   xmax=6,
   ymin=0,
   ymax=7,
   nodes near coords,
   point meta=explicit symbolic]
   \addplot+[only marks] coordinates{(0,1)[\Large$c_1$] (6,3)[\Large$c_2$]};
   \addplot+[only marks] coordinates{(1.5,6)[\Large$c_3$] (4.5,4)[\Large$c_4$]};
  
    \addplot+[mark = none, blue, thick] coordinates{(4,0) (4,1) (2,3) (2,7)};
    \addplot+[mark = none, red, thick] coordinates{(2,0) (2,4) (4,6) (4,7)};
    \end{axis}
    \end{tikzpicture}
    \caption{Intersection of two hypersurfaces: $H(c_1, c_2)$ is of type $V^-$ and $H(c_3, c_4)$ of type $V^+$.}
    \label{fig:hyps_overview5}
\end{figure}
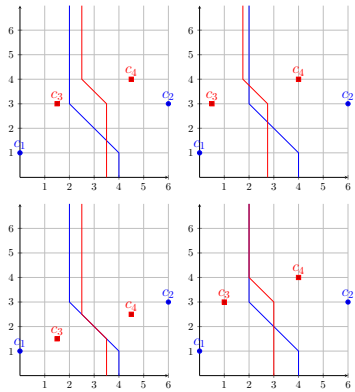
\begin{figure}[htb]
    \centering
    \begin{tikzpicture}[scale=0.4]
   \begin{axis}
   [axis x line=bottom,axis y line = left, 
   grid = major,
   axis equal image,
   ytick = {1,2,3,4,5,6},
   xtick = {1,2,3,4,5,6},
   xmin=0,
   xmax=6,
   ymin=0,
   ymax=7,
   nodes near coords,
   point meta=explicit symbolic]
   \addplot+[only marks] coordinates{(0,1)[\Large$c_1$] (6,3)[\Large$c_2$]};
   \addplot+[only marks] coordinates{(1.5,3)[\Large$c_3$] (4.5,4)[\Large$c_4$]};
  
    \addplot+[mark = none, blue, thick] coordinates{(4,0) (4,1) (2,3) (2,7)};
    \addplot+[mark = none, red, thick] coordinates{(3.5,0) (3.5,3) (2.5,4) (2.5,7)};
    \end{axis}
    \end{tikzpicture}
    \begin{tikzpicture}[scale=0.4]
   \begin{axis}
   [axis x line=bottom,axis y line = left, 
   grid = major,
   axis equal image,
   ytick = {1,2,3,4,5,6},
   xtick = {1,2,3,4,5,6},
   xmin=0,
   xmax=6,
   ymin=0,
   ymax=7,
   nodes near coords,
   point meta=explicit symbolic]
   \addplot+[only marks] coordinates{(0,1)[\Large$c_1$] (6,3)[\Large$c_2$]};
   \addplot+[only marks] coordinates{(0.5,3)[\Large$c_3$] (4,4)[\Large$c_4$]};
  
    \addplot+[mark = none, blue, thick] coordinates{(4,0) (4,1) (2,3) (2,7)};
    \addplot+[mark = none, red, thick] coordinates{(2.75,0) (2.75,3) (1.75,4) (1.75,7)};
    \end{axis}
    \end{tikzpicture}

    \begin{tikzpicture}[scale=0.4]
   \begin{axis}
   [axis x line=bottom,axis y line = left, 
   grid = major,
   axis equal image,
   ytick = {1,2,3,4,5,6},
   xtick = {1,2,3,4,5,6},
   xmin=0,
   xmax=6,
   ymin=0,
   ymax=7,
   nodes near coords,
   point meta=explicit symbolic]
   \addplot+[only marks] coordinates{(0,1)[\Large$c_1$] (6,3)[\Large$c_2$]};
   \addplot+[only marks] coordinates{(1.5,1.5)[\Large$c_3$] (4.5,2.5)[\Large$c_4$]};
  
    \addplot+[mark = none, blue, thick] coordinates{(4,0) (4,1) (2,3) (2,7)};
    \addplot+[mark = none, red, thick] coordinates{(3.5,0) (3.5,1.5) (2.5,2.5) (2.5,7)};
    \end{axis}
    \end{tikzpicture}
        \begin{tikzpicture}[scale=0.4]
   \begin{axis}
   [axis x line=bottom,axis y line = left, 
   grid = major,
   axis equal image,
   ytick = {1,2,3,4,5,6},
   xtick = {1,2,3,4,5,6},
   xmin=0,
   xmax=6,
   ymin=0,
   ymax=7,
   nodes near coords,
   point meta=explicit symbolic]
   \addplot+[only marks] coordinates{(0,1)[\Large$c_1$] (6,3)[\Large$c_2$]};
   \addplot+[only marks] coordinates{(1,3)[\Large$c_3$] (4,4)[\Large$c_4$]};
    \addplot+[mark = none, blue, thick] coordinates{(4,0) (4,1) (2,3) (2,7)};
    \addplot+[mark = none, red, thick] coordinates{(2,7) (2,4) (3,3) (3,0)};
    \end{axis}
    \end{tikzpicture}
    \caption{Intersection of two hypersurfaces: both \hyps $H(c_1, c_2)$ and $H(c_3, c_4)$ are of type $V^-$.}
    \label{fig:hyps_overview4}
\end{figure}

\subsection{Proof of Proposition~\ref{prop:triangle_intersection}}\label{app:prop4}

{\bf Proposition \ref{prop:triangle_intersection}.}
{\it 
Given three points $c_1$, $c_2$ and $c_3$:
\begin{itemize}
    \item If $H(c_1,c_2)$, $H(c_1, c_3)$ and $H(c_2, c_3)$ are all vertical (or all horizontal), then the intersection of each pair of \hyps is empty. In particular, the intersection of the 3 \hyps is empty.
    \item If two of them are vertical and one is horizontal (or vice-versa), then the intersection of the 3 \hyps is a unique point.
\end{itemize}  
}

\begin{proof}
Assume first that the three \hyps are vertical. Let $(x_1, y_1), (x_2, y_2)$ and $(x_3, y_3)$ denote the positions in the plane of $c_1, c_2$ and $c_3$. Without loss of generality, we assume that $x_1 < x_2 < x_3$ and that $H(c_1, c_3)$ is of type $V^-$. \\
\noindent
Given a vertical \hyp $H(c_i, c_j)$, for each point $(x,y)\!\in\!H(c_i, c_j)$ we have $x_i\!<\!x\!<\! x_j$, because $x_i\!<\!x_{M_1}\!<\!x_j$ and $x_i\!<\!x_{M_2}\!<\!x_j$ for the extremities $M_1, M_2$ of the middle segment of the \hyp (see Figure~\ref{fig:hyperplan_l1_base}, page~\pageref{fig:hyperplan_l1_base}). Thus, $H(c_1, c_2)$ and $H(c_2, c_3)$ do not intersect, as we have $x_1\!<\!x_2\!<\!x_3$. Using Lemma \ref{lemma:inter} (page~\pageref{lemma:inter}), we conclude that $H(c_1,c_2)\cap H(c_1, c_3)\cap H(c_2, c_3) = \emptyset$.\\ 

Let us now assume that two \hyps are vertical (resp. horizontal) and the third one is horizontal (resp. vertical).  \textcolor{black}{Without loss of generality, we can assume that $H(c_1, c_2)$ is horizontal and both remaining \hyps $H(c_1, c_3)$ and $H(c_2, c_3)$ are vertical. Any vertical \hyp intersects any horizontal \hyp in a unique point (by assuming w.l.o.g. that the representation is non-degenerate, see Lemma~\ref{lemma:noinfiniteinter}). In particular, $\vert H(c_1, c_2) \cap H(c_1, c_3)\vert\!=\!1$. Lemma~\ref{lemma:inter}  states that $H(c_1, c_2) \cap H(c_1, c_3) \cap H(c_2, c_3) = H(c_i, c_j) \cap H(c_i, c_k)$ for $\{ i,j,k \}\!=\!\{ 1,2,3\}$. We have therefore the three \hyps intersecting in a unique point. }
\end{proof}

\subsection{Proof of Corollary~\ref{cor:int}}\label{app:cor1}

{\bf Corollary~\ref{cor:int}.} {\it 
Given three points $c_i$, $c_j$, $c_k$, the \hyps $H(c_i, c_j)$ and $H(c_i, c_k)$ intersect in at most one point. In other words, if two hypersurfaces $H(c_i, c_j)$ and $H(c_k, c_l)$ intersect in two different points, then $c_i, c_j, c_k$ and $c_l$ are all distinct.}

\begin{proof}
By contradiction, assume that $H(c_i, c_j)$ and $H(c_i, c_k)$ intersect in two distinct points. According to Lemma \ref{lemma:inter} (page~\pageref{lemma:inter}), $\vert H(c_i, c_j) \cap H(c_i, c_k) \cap H(c_j, c_k) \vert \geq 2$. We get a contradiction with Proposition \ref{prop:triangle_intersection} which states that if the three \hyps intersect, then the point of intersection is unique.  
\end{proof}

\subsection{Proof of Proposition~\ref{prop:diag_rect}}\label{app:prop5}
{\bf Proposition \ref{prop:diag_rect}}
{\it 
Given three points $c_1 = (x_1, y_1)$, $c_2 = (x_2, y_2)$ and $c_3 = (x_3, y_3)$:
\begin{itemize}
    \item If one of these 3 points is inside the parallelogram determined by the two other points, then the three \hyps $H(c_1,c_2)$, $H(c_1, c_3)$ and $H(c_2, c_3)$ do not (pairwise) intersect. 
    \item Otherwise, the intersection of the three \hyps is a unique point. 
\end{itemize}  
}

\begin{proof}
This proposition is a direct consequence of Proposition \ref{prop:triangle_intersection}. 
\textcolor{black}{To prove the first point, we assume without loss of generality that $c_2$ lies inside the parallelogram determined by $c_1$ and $c_3$, and we prove that in this case, the three \hyps are all vertical or horizontal - Proposition \ref{prop:triangle_intersection} then implies that they do not (pairwise) intersect. \\ Up to exchanging the roles of $c_1$ and $c_3$, we can assume, still without loss of generality, that $x_1 < x_3$. There are then 4 cases to distinguish (see Figure~\ref{fig:paral_prop48}):}
\begin{itemize}
{
    \item[(a) ] $H(c_1, c_3)$ is of type $V^-$ (see Figure~\ref{subfig:paral_prop48_a}). \\ In this case, we have $y_1 < y_3$ (see the classification of \hyps given in Figure~\ref{fig:hyps_overview}). Moreover, we have $x_1 < x_2 < x_3$. In the parallelogram given in Figure~\ref{subfig:paral_prop48_a}, $c_2$ lies then above the diagonal $d_1^-$ (i.e,  $y_2 > -x_2 + x_1 + y_1$) and below the diagonal $d_1^+$ (i.e, $y_2 < x_2 - x_1 + y_1$). Put together, we get 
    $$ -x_2 + x_1 < y_2 - y_1 < x_2 - x_1.$$
    In other words, $|y_1 - y_2| < x_2 - x_1$, hence $H(c_1, c_2)$ is vertical. We show similarly that $H(c_2, c_3)$ is vertical, as $c_2$ lies above the diagonal $d_3^+$ and below the diagonal $d_3^-$. All three \hyps being vertical, they do not (pairwise) intersect. 
    \item[(b)] Let us now suppose that $H(c_1, c_3)$ is of type $H^-$ (see Figure~\ref{subfig:paral_prop48_b}). \\ We have $y_1 < y_2 < y_3$. As $c_2$ lies above diagonals $d_1^-$ and $d_1^+$, we have $y_2 > - x_2 + x_1 + y_1$ and $y_2 > x_2 - x_1 + y_1$. Put together, we have $y_2 - y_1 > x_1 - x_2$ and $y_2 - y_1 > -(x_1 - x_2)$ - in other words, $y_2 - y_1 > | x_1 - x_2 |$. Therefore, the \hyp $H(c_1, c_2)$ is horizontal. We show similarly that $H(c_2, c_3)$ is horizontal, as $c_2$ lies below diagonals $d_3^+$ and $d_3^-$, so we obtain $y_3 - y_2 > | x_2 - x_3|$. 
    \item[(c)] We suppose here that $H(c_1, c_3)$ is of type $V^+$. We have $y_1 > y_3$  and $x_1 < x_2 < x_3$. Analogously to the previous case, we show that $x_2 - x_1 > | y_2 - y_1|$, so $H(c_1, c_2)$ is vertical, and that $x_3 - x_2 > | y_2 - y_3 |$, which implies that $H(c_2, c_3) $ is also vertical. 
    \item[(d)] Finally, we consider $H(c_1, c_3)$ of type $H^+$. We have $y_1 > y_2 >y_3$. As in previous cases, we show that $H(c_1, c_2)$ is horizontal as $c_2$ lies below diagonals $d_1^+$ and $d_1^-$, and $H(c_2, c_3)$ is also horizontal as $c_2$ lies above diagonals $d_3^+$ and $d_3^-$. 
    }
\end{itemize}
    \textcolor{black}{To prove the second point of the proposition, we suppose that any point does not lie in the parallelogram determined by the remaining two points, and we will show that in such a case, there is at least one horizontal and  one vertical hypersurface. As a vertical \hyp and a horizontal \hyp intersect in a unique point, Lemma~\ref{lemma:inter} allows us to conclude that the three \hyps intersect in a unique point. \\
    Suppose first that $H(c_1, c_3)$ is vertical (see Figure~\ref{subfig:paral_prop48_2_a}). The diagonals $d_1^+, d_1^-, d_3^+$ and $d_3^-$ divide the plane into 9 areas $A_1, A_2, \hdots, A_9$. The point $c_2$ does not lie in $A_1$ (resp. $A_5$, $A_9$) because $c_1$ (resp. $c_2$, $c_3$) does not lie in the parallelogram determined by the remaining two points. If $c_2$ lies in area $A_2$, $A_3$, $A_4$ or $A_7$, the \hyp $H(c_1, c_2)$ is horizontal, so it intersects the vertical \hyp $H(c_1, c_3)$. If $c_2$ lies in one of the remaining areas $A_6$ or $A_8$, the \hyp $H(c_2, c_3)$ is horizontal, so it intersects $H(c_1, c_3)$. We note that whether the oblique middle-segment of the \hyp $H(c_1, c_3)$ is ascending (case (c) of Figure~\ref{fig:paral_prop48}) or descending (case (a)) 
    has no impact on this reasoning  and it can therefore be used without change for both cases (a) and (c) of Figure~\ref{fig:paral_prop48}. \\ Analogously, we treat the case in which $H(c_1, c_3)$ is horizontal: $c_2$ cannot lie in areas $A_3$, $A_5$ and $A_7$ as any point does not lie within the parallelogram determined by the remaining two points. If $c_2$ lies in $A_1$, $A_2$, $A_6$ or $A_9$, the \hyp $H(c_1, c_2)$ is vertical. If it lies in one of the two remaining areas $A_4$ or $A_8$, the \hyp $H(c_2, c_3)$ is vertical. To sum up, there is always at least one horizontal and one vertical hypersurface. } 
  
\end{proof}
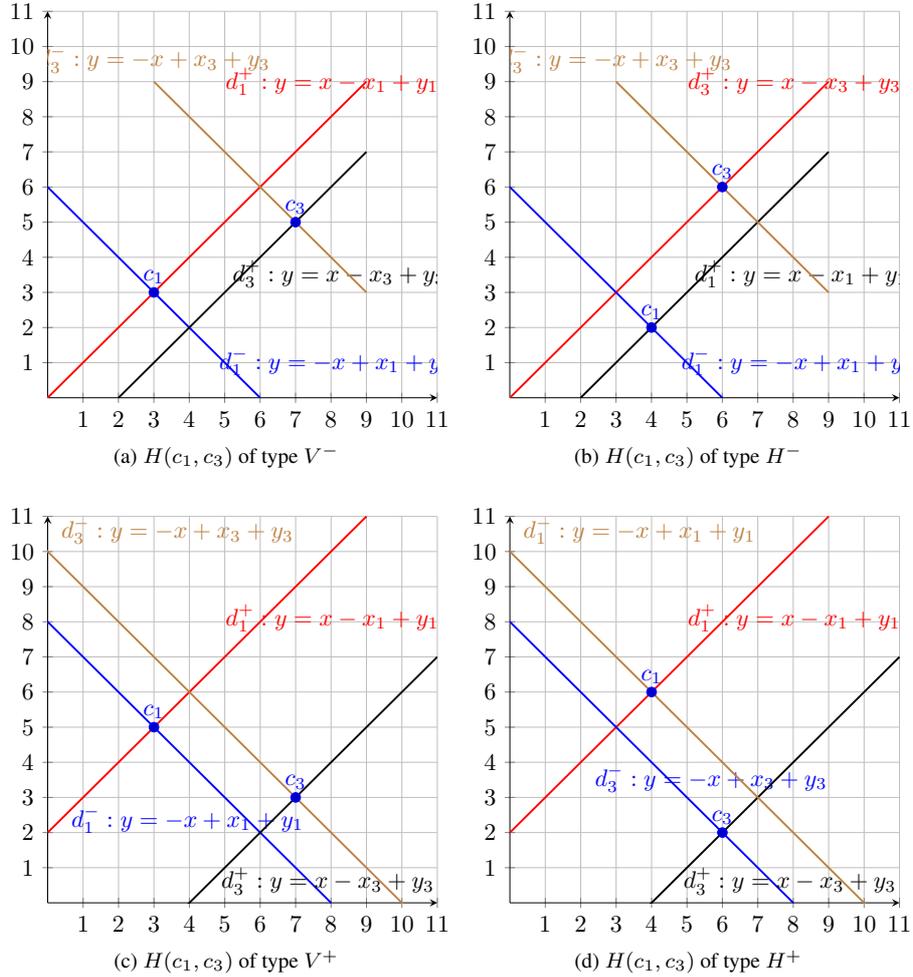
\begin{figure}[htb]
    \centering
    \subfloat[$H(c_1,c_3)$ of type $V^-$]{
    \begin{tikzpicture}[scale=0.9]
   \begin{axis}
   [axis x line=bottom,axis y line = left, 
   grid = major,
   axis equal image,
   ytick = {1,2,3,4,5,6,7,8,9,10,11},
   xtick = {1,2,3,4,5,6,7,8,9,10,11},
   xmin=0,
   xmax=11,
   ymin=0,
   ymax=11,
   nodes near coords,
   point meta=explicit symbolic]
   \addplot+[only marks] coordinates{(3,3)[$c_1$] (7,5)[$c_3$]};
    \addplot+[mark = none, red, thick] coordinates{(0,0) (9,9)} node[xshift = -0.5cm] {$d^+_1: y = x - x_1 + y_1$};
    \addplot+[mark = none, blue, thick] coordinates{(0,6) (6,0)} node[xshift = 1.1cm, yshift = 0.2cm, above,pos=1] {$d^-_1: y = - x + x_1 + y_1$};
    \addplot+[mark = none, thick] coordinates{(2,0) (9,7)} node[xshift = -0.4cm, yshift = -1.8cm] {$d^+_3: y = x - x_3 + y_3$};
     \addplot+[mark = none, color = brown, thick] coordinates{(9,3) (3,9)} node[above,pos=1] {$d^-_3: y = - x + x_3 + y_3$};
    \end{axis}
    \end{tikzpicture}\label{subfig:paral_prop48_a}}
    \subfloat[$H(c_1, c_3)$ of type $H^-$]{
     \begin{tikzpicture}[scale=0.9]
   \begin{axis}
   [axis x line=bottom,axis y line = left, 
   grid = major,
   axis equal image,
   ytick = {1,2,3,4,5,6,7,8,9,10,11},
   xtick = {1,2,3,4,5,6,7,8,9,10,11},
   xmin=0,
   xmax=11,
   ymin=0,
   ymax=11,
   nodes near coords,
   point meta=explicit symbolic]
   \addplot+[only marks] coordinates{(4,2)[$c_1$] (6,6)[$c_3$]};
    \addplot+[mark = none, red, thick] coordinates{(0,0) (9,9)} node[xshift = -0.5cm] {$d^+_3: y = x - x_3 + y_3$};
    \addplot+[mark = none, blue, thick] coordinates{(0,6) (6,0)} node[xshift = 1.1cm, yshift = 0.2cm, above,pos=1] {$d^-_1: y = - x + x_1 + y_1$};
    \addplot+[mark = none, thick] coordinates{(2,0) (9,7)} node[xshift = -0.4cm, yshift = -1.8cm] {$d^+_1: y = x - x_1 + y_1$};
     \addplot+[mark = none, color = brown, thick] coordinates{(9,3) (3,9)} node[above,pos=1] {$d^-_3: y = - x + x_3 + y_3$};
    \end{axis}
    \end{tikzpicture}\label{subfig:paral_prop48_b}} \\
    \subfloat[$H(c_1, c_3)$ of type $V^+$]{
    \begin{tikzpicture}[scale=0.9]
   \begin{axis}
   [axis x line=bottom,axis y line = left, 
   grid = major,
   axis equal image,
   ytick = {1,2,3,4,5,6,7,8,9,10,11},
   xtick = {1,2,3,4,5,6,7,8,9,10,11},
   xmin=0,
   xmax=11,
   ymin=0,
   ymax=11,
   nodes near coords,
   point meta=explicit symbolic]
   \addplot+[only marks] coordinates{(3,5)[$c_1$] (7,3)[$c_3$]};
    \addplot+[mark = none, red, thick] coordinates{(0,2) (9,11)} node[xshift = -0.5cm, yshift = -1.5cm] {$d^+_1: y = x - x_1 + y_1$};
    \addplot+[mark = none, blue, thick] coordinates{(0,8) (8,0)} node[xshift = -2.1cm, yshift = 0.9cm, above,pos=1] {$d^-_1: y = - x + x_1 + y_1$};
    \addplot+[mark = none, thick] coordinates{(2,-2) (11,7)} node[xshift = -1.6cm, yshift = -3.3cm] {$d^+_3: y = x - x_3 + y_3$};
     \addplot+[mark = none, color = brown, thick] coordinates{(-1,11)(11,-1) } node[xshift = -3.8cm, yshift = 6cm] {$d^-_3: y = - x + x_3 + y_3$};
    \end{axis}
    \end{tikzpicture}\label{subfig:paral_prop48_c}}
    \subfloat[$H(c_1, c_3)$ of type $H^+$]{
     \begin{tikzpicture}[scale=0.9]
   \begin{axis}
   [axis x line=bottom,axis y line = left, 
   grid = major,
   axis equal image,
   ytick = {1,2,3,4,5,6,7,8,9,10,11},
   xtick = {1,2,3,4,5,6,7,8,9,10,11},
   xmin=0,
   xmax=11,
   ymin=0,
   ymax=11,
   nodes near coords,
   point meta=explicit symbolic]
   \addplot+[only marks] coordinates{(4,6)[$c_1$] (6,2)[$c_3$]};
    \addplot+[mark = none, red, thick] coordinates{(0,2) (9,11)} node[xshift = -0.5cm, yshift = -1.5cm] {$d^+_1: y = x - x_1 + y_1$};
    \addplot+[mark = none, blue, thick] coordinates{(0,8) (8,0)} node[xshift = -1.2cm, yshift = 1.5cm, above,pos=1] {$d^-_3: y = - x + x_3 + y_3$};
    \addplot+[mark = none, thick] coordinates{(2,-2) (11,7)} node[xshift = -1.6cm, yshift = -3.3cm] {$d^+_3: y = x - x_3 + y_3$};
     \addplot+[mark = none, color = brown, thick] coordinates{(-1,11)(11,-1) } node[xshift = -3.8cm, yshift = 6cm] {$d^-_1: y = - x + x_1 + y_1$};
    \end{axis}
    \end{tikzpicture}\label{subfig:paral_prop48_d}}
    
    \caption{The parallelogram determined by $c_1$ and $c_3$ with $x_1 < x_3$: 4 cases to distinguish. 
    }
    \label{fig:paral_prop48}
\end{figure}
\begin{figure}[htb]
     \subfloat[$H(c_1,c_3)$ is vertical]{
    \begin{tikzpicture}[scale=0.9]
   \begin{axis}
   [axis x line=bottom,axis y line = left, 
   grid = major,
   axis equal image,
   ytick = {1,2,3,4,5,6,7,8,9,10,11},
   xtick = {1,2,3,4,5,6,7,8,9,10,11},
   xmin=0,
   xmax=11,
   ymin=0,
   ymax=11,
   nodes near coords,
   point meta=explicit symbolic]
   \addplot+[only marks] coordinates{(3,3)[$c_1$] (7,5)[$c_3$]};
    \addplot+[mark = none, red, thick] coordinates{(0,0) (11,11)} node[xshift = -1.2cm,yshift = -0.8cm] {$d^+_1: y = x - x_1 + y_1$};
    \addplot+[mark = none, blue, thick] coordinates{(0,6) (6,0)} node[xshift = 1.1cm, yshift = 0.2cm, above,pos=1] {$d^-_1: y = - x + x_1 + y_1$};
    \addplot+[mark = none, thick] coordinates{(2,0) (11,9)} node[xshift = -1.5cm, yshift = -3cm] {$d^+_3: y = x - x_3 + y_3$};
     \addplot+[mark = none, color = brown, thick] coordinates{(12,0) (0,12)} node[above,pos=1,xshift = 2cm, yshift = -1cm] {$d^-_3: y = - x + x_3 + y_3$};
     
     \node[] at (axis cs:1,3) {$A_{1}$};
     \node[] at (axis cs:2.5,1.5) {$A_{2}$};
     \node[] at (axis cs:4,0.5) {$A_{3}$};
     \node[] at (axis cs:2.5,6.5) {$A_{4}$};
     \node[] at (axis cs:5,4) {$A_{5}$};
     \node[] at (axis cs:8,2) {$A_{6}$};
     \node[] at (axis cs:6,8) {$A_{7}$};
     \node[] at (axis cs:8,7) {$A_{8}$};
     \node[] at (axis cs:10,5) {$A_{9}$};
    \end{axis}
    \end{tikzpicture}\label{subfig:paral_prop48_2_a}}
     \subfloat[$H(c_1, c_3)$ is horizontal]{
     \begin{tikzpicture}[scale=0.9]
   \begin{axis}
   [axis x line=bottom,axis y line = left, 
   grid = major,
   axis equal image,
   ytick = {1,2,3,4,5,6,7,8,9,10,11},
   xtick = {1,2,3,4,5,6,7,8,9,10,11},
   xmin=0,
   xmax=11,
   ymin=0,
   ymax=11,
   nodes near coords,
   point meta=explicit symbolic]
   \addplot+[only marks] coordinates{(4,2)[$c_1$] (6,6)[$c_3$]};
    \addplot+[mark = none, red, thick] coordinates{(0,0) (11,11)} node[xshift = -1.2cm,yshift = -0.8cm] {$d^+_1: y = x - x_3 + y_3$};
    \addplot+[mark = none, blue, thick] coordinates{(0,6) (6,0)} node[xshift = 1.1cm, yshift = 0.2cm, above,pos=1] {$d^-_1: y = - x + x_1 + y_1$};
    \addplot+[mark = none, thick] coordinates{(2,0) (11,9)} node[xshift = -1.5cm, yshift = -3cm] {$d^+_3: y = x - x_1 + y_1$};
     \addplot+[mark = none, color = brown, thick] coordinates{(12,0) (0,12)} node[above,pos=1,xshift = 2cm, yshift = -1cm] {$d^-_3: y = - x + x_3 + y_3$};
     
     \node[] at (axis cs:1,3) {$A_{1}$};
     \node[] at (axis cs:2.5,1.5) {$A_{2}$};
     \node[] at (axis cs:4,0.5) {$A_{3}$};
     \node[] at (axis cs:2.5,6.5) {$A_{4}$};
     \node[] at (axis cs:5,4) {$A_{5}$};
     \node[] at (axis cs:8,2) {$A_{6}$};
     \node[] at (axis cs:6,8) {$A_{7}$};
     \node[] at (axis cs:8,7) {$A_{8}$};
     \node[] at (axis cs:10,5) {$A_{9}$};
    \end{axis}
    \end{tikzpicture}\label{subfig:paral_prop48_2_b}}
    \centering
    \caption{\textcolor{black}{The parallelogram determined by $c_1$ and $c_3$ with $x_1\!<\!x_3$, and the possible placements of $c_2$.} 
    }
    \label{fig:paral_prop48_2}
\end{figure}
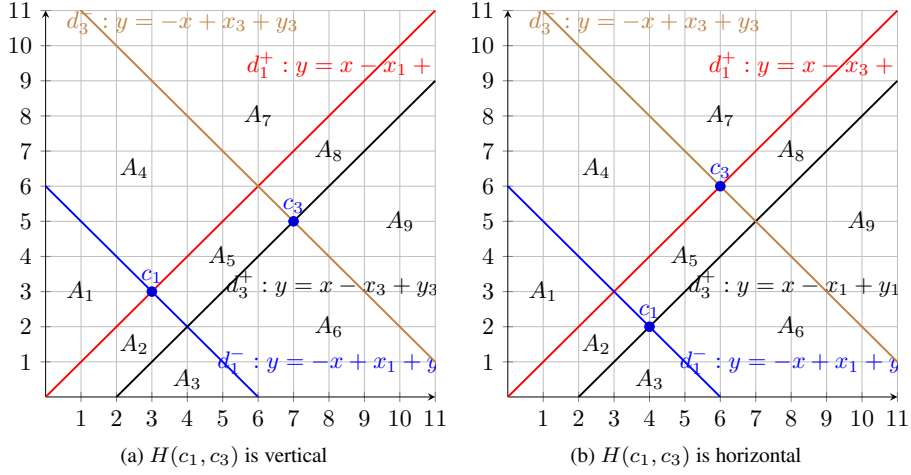

\subsection{Proof of Proposition~\ref{prop:nb_intersections_l1_c4}}\label{app:prop6}

{\bf Proposition \ref{prop:nb_intersections_l1_c4}.}
{\it Given four points $c_1, c_2, c_3$ and $c_4$, there is at most one pair of \hyps $H(c_i, c_j), H(c_k, c_l)$ (with $\{ i,j,k,l\} = \{ 1,2,3,4 \}$) intersecting in two distinct points.}

\begin{proof}
\textcolor{black}{For $i\!\in\!\{ 1,2,3,4\}$, we denote by $(x_i, y_i)$ the position of candidate $c_i$ in the plane}. Assume there are two pairs of \hyps intersecting in two distinct points. \textcolor{black}{Thanks to Corollary \ref{cor:int}, we can assume, w.l.o.g.,} that the first pair involves the \hyps $H(c_1, c_2)$ and $H(c_3, c_4)$. \textcolor{black}{Moreover, still w.l.o.g, we can assume that they} are both of type $V^-$, and that $H(c_1,c_2)$ is ``on the left'' of $H(c_3,c_4)$ (as in Figure~\ref{fig:inthyps}), and that $x_1 < x_2$ and $x_3 < x_4$. According \textcolor{black}{to the classification of \hyps (see Figure~\ref{fig:hyps_overview}, page~\pageref{fig:hyps_overview})}, as $H(c_1, c_2)$ and $H(c_3, c_4)$ are of type $V^-$, we have $y_1 < y_2$ and $y_3 < y_4$.  Note that we necessarily have: 
\begin{equation}\label{eq:b}
    \{x_1,x_3\}<\{x_2,x_4\}
\end{equation}
and 
\begin{equation}\label{eq:a}
    y_2>\{y_1,y_4\}>y_3.
\end{equation}
Equation~(\ref{eq:b}) directly follows from the fact that for each point $(x,y) $ of a \hyp $H(c_i, c_j)$, we have $x \in [x_i,x_j]$\textcolor{black}{: indeed, if $x_3 > x_2$, we would have $x_1 < x_2 < x_3 < x_4$ and the x-coordinate of each point of $H(c_1, c_2)$ would be smaller than the x-coordinate of each point of $H(c_3, c_4)$. In other words, the \hyps would not intersect. An analogous reasoning can be done to show that $x_1 < x_4$}. Equation~(\ref{eq:a}) follows from the fact that the \hyps do not even intersect if these inequalities are not satisfied. 

Furthermore, Equation~(\ref{eq:a}) means that when two vertical \hyps $H(c_1,c_2)$ and $H(c_3,c_4)$ intersect twice, if the highest point in $\{c_1,\ldots,c_4\}$ belongs to $\{c_1,c_2\}$ (resp. $\{c_3,c_4\}$) then the lowest point belongs to $\{c_3,c_4\}$ (resp. $\{c_1,c_2\}$). 

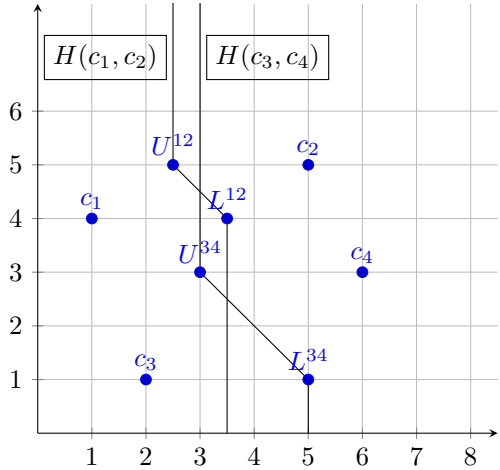
\begin{figure}[htb]
            \centering
            \begin{tikzpicture}[scale=1]
            \begin{axis}
        [axis x line=bottom,axis y line = left, 
        grid = major,
        axis equal image, 
        ytick = {1,2,3,4,5,6},
        xtick = {1,2,3,4,5,6,7,8,9},
        xmin=0,
        xmax=8.5,
        ymin=0,
        ymax=8,
        nodes near coords,
        point meta=explicit symbolic]
        \addplot+[only marks] coordinates{(2,1)[$c_3$] (6,3)[$c_4$] (1,4)[$c_1$] (5,5)[$c_2$] (3,3)[$U^{34}$] (3.5,4)[$L^{12}$] (2.5,5)[$U^{12}$] (5,1)[$L^{34}$]};

        \addplot+[mark = none,black] coordinates{(5,0) (5,1) (3,3) (3,10)};
        \addplot+[mark = none,black] coordinates{(3.5,0) (3.5,4) (2.5,5) (2.5,10)};

         \node[draw] at (axis cs:4.25,7) {$H(c_3, c_4)$};
         \node[draw] at (axis cs:1.25,7) {$H(c_1, c_2)$};
         
        \end{axis}
        \end{tikzpicture}
            \caption{Two intersecting hypersurfaces.}
            \label{fig:inthyps}
        \end{figure}

Assume first that the second pair of \hyps that intersect twice are also vertical. As $y_3=\min\{y_1,y_2,y_3,y_4\}$ and $y_2=\max\{y_1,y_2,y_3,y_4\}$, by the discussion above $H(c_2,c_3)$ and $H(c_1,c_4)$ cannot intersect twice vertically. So the unique possibility is that $H(c_1,c_3)$ and $H(c_2,c_4)$ intersect twice. This is, however, not possible: any point in $H(c_1,c_3)$ has $x$-coordinate in $[x_1,x_3]$, any point in  $H(c_2,c_4)$ has $x$-coordinate in $[x_2,x_4]$, but $[x_1,x_3]\cap [x_2,x_4]=\emptyset$ by Equation~(\ref{eq:b}).\\

Suppose now that the second pair of \hyps that intersect twice are horizontal. This pair can be either $H(c_1,c_4)$ and $H(c_2,c_3)$, or $H(c_1,c_3)$ and $H(c_2,c_4)$. 
\begin{itemize}
    \item Let us first consider the case where it is $H(c_1,c_4)$ and $H(c_2,c_3)$, which is illustrated in Figure~\ref{fig:inthyps2}. Let us look at the preference $p_1$ in the upper left part. We have $c_1>c_2$ and $c_3>c_4$ (by the positions of $H(c_1,c_2)$ and $H(c_3,c_4)$). As $H(c_2,c_3)$ is horizontal, we have $c_2>c_3$ (because $y_2>y_3$ from Equation~(\ref{eq:a}), see also Figure~\ref{fig:inthyps}). Therefore $p_1=(c_1,c_2,c_3,c_4)$. As $H(c_1,c_4)$ and $H(c_2,c_3)$ are horizontal, $H(c_2,c_3)$ is necessarily \emph{above} $H(c_1,c_4)$ on the (infinite) left part of the figure, since starting from $p_1=(c_1,c_2,c_3,c_4)$ and going down we need to cross $H(c_2,c_3)$ before $H(c_1,c_4)$: \textcolor{black}{in fact, going down from the area corresponding to $p_1$, we will not cross nor $H(c_1, c_2)$ neither $H(c_3,c_4)$ as they are vertical. We can only cross the remaining \hyps $H(c_1, c_3), H(c_1,c_4), H(c_2, c_3)$ and $H(c_2, c_4)$. However, the \hyp $H(c_1, c_3)$ (resp. $H(c_1, c_4)$, $H(c_2, c_4)$) cannot be the first \hyp to be crossed, as $c_2$ is ranked between $c_1$ and $c_3$ (resp. $c_2$ and $c_3$ are ranked between $c_1$ and $c_4$, $c_3$ between $c_2$ and $c_4$). Therefore, the first \hyp to be crossed is necessarily $H(c_2, c_3)$ ($c_2$ and $c_3$ are ranked one next to other in $p_1$, so they can be swapped).} 
    
    Similarly, we get $p_2\!=\!(c_4,c_3,c_2,c_1)$ in the lower right part, thus $H(c_2,c_3)$ is \emph{below} $H(c_1,c_4)$ on the (infinite) right part of the figure, since starting from $p_2=(c_4,c_3,c_2,c_1)$ and going up we need to cross $H(c_2,c_3)$ before $H(c_1,c_4)$, \textcolor{black}{using the same reasoning as in the case of $p_1$.}
    
 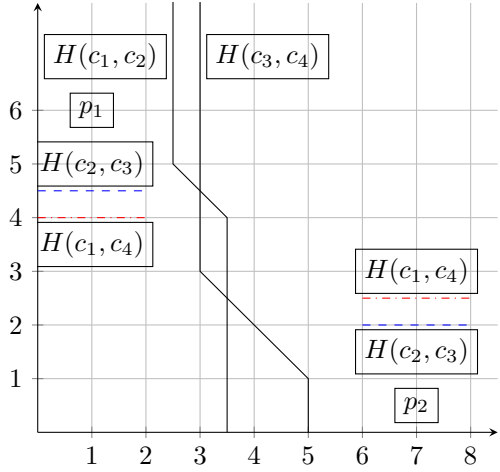
\begin{figure}[htb]
            \centering
            \begin{tikzpicture}[scale=1]
            \begin{axis}
        [axis x line=bottom,axis y line = left, 
        grid = major,
        axis equal image,
        ytick = {1,2,3,4,5,6},
        xtick = {1,2,3,4,5,6,7,8,9},
        xmin=0,
        xmax=8.5,
        ymin=0,
        ymax=8,
        nodes near coords,
        point meta=explicit symbolic]

        \addplot+[mark = none,black] coordinates{(5,0) (5,1) (3,3) (3,10)};
        \addplot+[mark = none,black] coordinates{(3.5,0) (3.5,4) (2.5,5) (2.5,10)};
        
         \addplot+[mark=none,blue,dashed] coordinates{(0,4.5) (2,4.5)};
         
        \addplot+[mark=none,red,dashdotted] coordinates{(0,4) (2,4)};
                         
                          \addplot+[mark=none,blue,dashed] coordinates{(6,2) (8,2)};
                        
           \addplot+[mark=none,red,dashdotted] coordinates{(6,2.5) (8,2.5)};  
         
         \node[draw] at (axis cs:4.25,7) {$H(c_3, c_4)$};
         \node[draw] at (axis cs:1.25,7) {$H(c_1, c_2)$};
                  \node[draw] at (axis cs:1,5) {$H(c_2, c_3)$};
         \node[draw] at (axis cs:1,3.5) {$H(c_1, c_4)$};
          \node[draw] at (axis cs:7,1.5) {$H(c_2, c_3)$};
         \node[draw] at (axis cs:7,3) {$H(c_1, c_4)$};
         
         \node[draw] at (axis cs:1,6) {$p_1$};
         \node[draw] at (axis cs:7,0.5) {$p_2$};
        \end{axis}
        \end{tikzpicture}
            \caption{Relative positions of $H(c_1,c_4)$ and $H(c_2,c_3)$.}
            \label{fig:inthyps2}
        \end{figure}   
    
    Hence, $H(c_2,c_3)$ and $H(c_1,c_4)$ cannot intersect twice (otherwise the same \hyp would be above the other one both on the left part and the right part of the figure).
    
    \item Let us finally focus on the case where  $H(c_1,c_3)$ and $H(c_2,c_4)$ are horizontal and intersect twice. \textcolor{black}{They are necessarily both of type $H^+$ or both of type $H^-$ (otherwise, they cannot intersect in two different points). We will show that none of these two cases is possible - in other words, that we cannot have $H(c_1, c_3)$ and $H(c_2, c_4)$ intersecting twice.}
    \begin{enumerate}
        \item  \textcolor{black}{Firstly, let us assume that $H(c_1,c_3)$ and $H(c_2, c_4)$ are of type $H^-$: \\ We recall that $y_2 > \{ y_1, y_4\} > y_3$ and $\{ x_1, x_3\} < \{ x_2, x_4\}$ (see Equations~\ref{eq:a} and \ref{eq:b}). However, the information on the type of $H(c_1,c_3)$ and $H(c_2, c_4)$ allows us to complete these partial orders on the coordinates of the candidates: using the classification of \hyps (see Figure~\ref{fig:hyps_overview}), we must have $x_3 < x_1$ and $x_4 < x_2$ if the \hyps are of type $H^-$. Put together, we have $x_3 < x_1 < x_4 < x_2$. Moreover, the necessary order on $y$-coordinates is $y_2 > y_1 > y_4 > y_3$ - if the order would be $y_2 > y_4 > y_1 > y_3 $, $H(c_1, c_3)$ and $H(c_2, c_4)$ would not intersect as any point of $H(c_i, c_j)$ has its $y$-coordinate in $[x_i, x_j]$. Let we denote by $U^{ij} = (U^{ij}_x, U^{ij}_y)$ (resp. $L^{ij} = (L^{ij}_x, L^{ij}_y)$) the upper extreme point (resp. the lower extreme point) of the middle-segment of $H(c_i, c_j)$. As $H(c_1, c_2)$ and $H(c_3, c_4)$ are of type $V^-$ and intersect twice, the given orders on both $x$-coordinates and $y$-coordinates of candidates implies that $U^{12}_x < U^{34}_x$ and $L^{12}_x < L^{34}_x$ (see Figure~\ref{fig:inthyps}). When we express the segment extremities positions using the candidates coordinates, these two inequalities rewrite, after simplifying, as follows: 
    \begin{align*}
        x_2 + x_1 + y_1 - y_2 & < x_4 + x_3 + y_3 - y_4 \\
        x_1 + x_2 + y_2 - y_1 & < x_3 + x_4 + y_4 - w_3
    \end{align*}
    If we sum the both inequalities, we get: 
    $$ 2 (x_4 + x_3) > 2 (x_1 + x_2),$$
    in other words, 
    $$ x_4 - x_1 > x_2 - x_3.$$
    But this is in contradiction with the order on $x$-coordinates which states that $x_3 < x_1 < x_4 < x_2$. Therefore, the \hyps $H(c_1, c_3)$ and $H(c_2, c_4)$ cannot be of type $H^-$.}
    \item  \textcolor{black}{Let us now assume that} $H(c_1, c_3)$ and $H(c_2, c_4)$ are of type $H^+$: \\ This case is illustrated in Figure~\ref{fig:doubleinthyps}. Note that the ``upper'' horizontal part of $H(c_1,c_3)$ (starting at the $x$-position $x_3$) is below the line $(c_1,L^{12})$, as $y_3<y_1$. Similarly, the lower horizontal part of $H(c_2,c_4)$  (ending at the $x$-position $x_2$) is above the line $(U^{34},c_4)$ as $y_2>y_4$. If $H(c_1,c_3)$ and $H(c_2,c_4)$ intersect, then $H(c_1,c_3)$ is above $H(c_2,c_4)$ in the central part, see Figure~\ref{fig:doubleinthyps}.\\
     Then in the (non empty) rectangle delimited by the 4 \hyps (in the center of Figure~\ref{fig:doubleinthyps}), we have: $c_1>c_2$, $c_4>c_3$, $c_2>c_4$ and $c_3>c_1$, which yields $c_1>c_2>c_4>c_3>c_1$, a contradiction.
    \end{enumerate}
    \textcolor{black}{To conclude, we have proved by contradiction that $H(c_1, c_3)$ and $H(c_2, c_4)$ can neither be both of type $H^+$ nor both of type $H^-$. 
    Therefore, they cannot intersect twice. \qed}

    \begin{figure}[htb]
            \centering
            \begin{tikzpicture}[scale=1]
            \begin{axis}
        [axis x line=bottom,axis y line = left, 
        grid = major,
        axis equal image,
        ytick = {1,2,3,4,5,6},
        xtick = {1,2,3,4,5,6,7,8,9},
        xmin=0,
        xmax=8.5,
        ymin=0,
        ymax=8,
        nodes near coords,
        point meta=explicit symbolic]
        \addplot+[only marks] coordinates{(2,1)[$c_3$] (6,3)[$c_4$] (1,4)[$c_1$] (5,5)[$c_2$] (3,3)[$U^{34}$] (3.5,4)[$L^{12}$] };

        \addplot+[mark = none,black] coordinates{(5,0) (5,1) (3,3) (3,10)};
        \addplot+[mark = none,black] coordinates{(3.5,0) (3.5,4) (2.5,5) (2.5,10)};
        
               \addplot+[mark = none,blue] coordinates{(0,3.5) (5,3.5) (6,4.5) (10,4.5)};
        
                   \addplot+[mark = none,blue,dotted] coordinates{(2,3.75) (10,3.75)};

         \node[draw] at (axis cs:4.25,7) {$H(c_3, c_4)$};
         \node[draw] at (axis cs:1.25,7) {$H(c_1, c_2)$};
         \node[draw,blue] at (axis cs:7,5.2) {$H(c_2, c_4)$};
         \node[draw,blue] at (axis cs:7.5,3.25) {$H(c_1, c_3)$};
         
        \end{axis}
        \end{tikzpicture}
            \caption{The case where $H(c_1,c_3)$ and $H(c_2,c_4)$ are horizontal and intersect twice.}
            \label{fig:doubleinthyps}
        \end{figure}
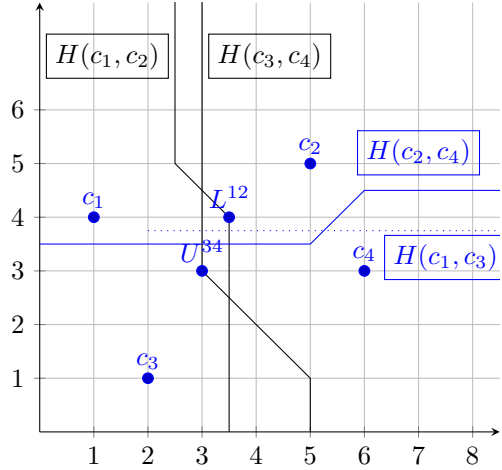
       
\end{itemize}
\end{proof}

\newpage

\section{Missing proofs of Section~\ref{sec:n>=5}}

\subsection{Missing part in the proof of Theorem~\ref{th:thetam4}}\label{app:thetam4}

{\bf Explicit construction of a family of profiles with $\Theta(m^4)$ distinct preferences}

\begin{proof}

\noindent
We set $c_1\!=\!(0,0)$ and $c_2\!=\!(1,2)$. According to the classification of hypersurfaces, $H(c_1, c_2)$ is horizontal (more precisely of type $H^-$). We then place $c_3$ in such a way that both $H(c_1, c_3)$ and $H(c_2, c_3)$ are vertical. To do so, we need to fix the values of $x_3$ and $y_3$ (coordinates of $c_3$) such that: 
\begin{align*}
    \vert x_1 - x_3 \vert & > \vert y_1 - y_3 \vert \\
    \mbox{and } \vert x_2 - x_3 \vert & > \vert y_2 - y_3 \vert. 
\end{align*}
\noindent 
\textcolor{black}{This can be done by setting, for instance, }
$$ \textcolor{black}{y_3 = \frac{y_1 + y_2}{2}}$$
\textcolor{black}{and} 
$$ \textcolor{black}{x_3 = \max\{ x_1, x_2\} + 2|y_1 - y_2 |}$$
\textcolor{black}{We check that, indeed, for $i \in \{ 1,2\}$, we have  
\begin{align*}
    |x_3 - x_i | & = | \, \max\{ x_1, x_2\} - x_i + 2|y_1 - y_2| \; | \\ & \geq  2|y_1 - y_2 |  > \frac{1}{2} |y_1 - y_2 | \geq |\frac{y_1 + y_2}{2} - y_i| = |y_3 - y_i | 
\end{align*}}
\textcolor{black}{where the strict inequality follows from the fact that $y_1\!\neq\! y_2$. 
Geometrically, choosing $y_3$ between $y_1$ and $y_2$ ensures that $|y_3 - y_i|$ is upper bounded by $|y_1 - y_2 | $. To guarantee that $|x_i - x_3| > |y_i - y_3|$, it is then sufficient  that $x_3$ is taken large enough - here, the distance from $\max\{x_1, x_2\}$ (and so in particular from both $x_1$ and $x_2$) to $x_3$ is greater than the above mentioned upper bound $|y_1 - y_2|$. \\ 
We will now generalize the idea: we want $H(c_{2k}, c_i)$ to be horizontal for all $k \geq 1, i < 2k$, and $H(c_{2k+1}, c_i)$ to be vertical for $k \geq 1, i < 2k+1$. Let us detail only the case of horizontal \hyps (the case of vertical ones being symmetric). \\
We set
$$ x_{2k} = \frac{\max\limits_{i < 2k}\{ x_i\}+ \min\limits_{i < 2k}\{x_i\}}{2}$$ and 
$$ y_{2k} = \max_{i < 2k}\{ y_i\} + 2\left(\max_{i < 2k}\{x_i\} - \min_{i < 2k}\{ x_i\}\right).$$
The geometrical intuition remains the same - as we need, for all $i\!<\!2k$, $|x_{2k} - x_i | < |y_{2k} - y_i|$, we choose the value of $x_{2k}$ so that $|x_{2k} - x_i|$ is upper bounded by $\max_{i < 2k}\{x_i\} - \min_{i < 2k}\{x_i\}$, and we chose then $y_{2k}$ in such a way that $|y_{2k} - y_i |$ is  greater than this upper bound. Formally, we have: 
\begin{align*}
    |y_{2k} - y_i | & = \left| \max_{i < 2k}\{ y_i\} + 2\left(\max_{i < 2k}\{x_i\} - \min_{i < 2k}\{ x_i\}\right) - y_i \right| \\
    & = \left| \max_{i < 2k}\{ y_i\} - y_i + 2\left(\max_{i < 2k}\{x_i\} - \min_{i < 2k}\{ x_i\}\right) \right| \\
    & \geq 2\left(\max_{i < 2k}\{x_i\} - \min_{i < 2k}\{x_i\} \right) \\
    & > \frac{1}{2}\left(\max_{i < 2k}\{x_i\} - \min_{i < 2k}\{x_i\} \right) \\
    & 
    \geq 
    \left| \frac{\max\limits_{i < 2k}\{ x_i\}+ \min\limits_{i < 2k}\{x_i\}}{2} - x_i \right| = |x_{2k} - x_i |
\end{align*}
Therefore we have $|y_{2k} - y_i | > |x_{2k} - x_i |$, so $H(c_{2k}, c_i)$ is horizontal. \\
Analogously, we set
$$ y_{2k+1} = \frac{\max\limits_{i < 2k+1}\{ y_i\}+ \min\limits_{i < 2k+1}\{y_i\}}{2}$$ and 
$$ x_{2k+1} = \max_{i < 2k+1}\{ x_i\} + 2\left(\max_{i < 2k+1}\{y_i\} - \min_{i < 2k+1}\{ y_i\}\right).$$
We prove as above (just by swapping the roles of $x$ and $y$) that, in this case, the \hyps $H(c_{2k+1}, c_i)$ are all vertical. \\}

For $k \leq m$, we denote by $H_k$ (resp. $V_k$) the number of horizontal (resp. vertical) \hyps after adding the $k$-th candidate. 

As all horizontal \hyps intersect all vertical hypersurfaces, these intersections already define $(H_k+1)(V_k+1)$ different areas (with distinct preferences). Hence, denoting by $A_m$ the number of areas after adding the $k$-th candidate, we have $A_m\geq (H_m+1)(V_m+1)$. 

Each time we add a candidate $c_k$, we obtain $k-1$ new \hyps $H(c_1, c_k),$ $\hdots,$ $H(c_{k-1}, c_k)$, all horizontal if $k$ is even, or all vertical if $k$ is odd. Consequently:
\begin{itemize}
    \item if $k$ is even, $H_k=H_{k-1}+(k-1)$ and $V_k=V_{k-1}$;
    \item if $k$ is odd, $H_k=H_{k-1}$ and $V_k=V_{k-1}+k-1$. 
\end{itemize}

We can deduce that $H_m\!\in\!\Theta(m^2)$ and $V_m\!\in\!\Theta(m^2)$, and thus $A_m\!\in\! \Omega(m^4)$.
\end{proof}

\end{document}